\RequirePackage{ifpdf} %
\ifpdf
\documentclass[10pt,pdftex]{amsart}%
\else
\documentclass[10pt,dvips]{amsart}%
\fi

\ifpdf
\fi
\usepackage[bookmarks, %
colorlinks=false, backref=false,plainpages=false]{hyperref}
\usepackage[english]{babel}

\providecommand{\ran}{\operatorname{ran}}

\usepackage{geometry}
\usepackage{amsmath, amssymb}

\usepackage{graphicx}
\usepackage{amsmath,amssymb,amsthm}

\usepackage{booktabs}

\newtheorem{thm}{Theorem}[section]
\newtheorem{cor}[thm]{Corollary}
\newtheorem{lem}[thm]{Lemma}
\newtheorem{prop}[thm]{Proposition}

\theoremstyle{definition}
\newtheorem{defn}[thm]{Definition}
\newtheorem{assum}[thm]{Assumption}
\theoremstyle{remark}
\newtheorem{rem}[thm]{Remark}

\theoremstyle{plain}

\usepackage{amsaddr}

\def\XXint#1#2#3{{\setbox0=\hbox{$#1{#2#3}{\int}$}
\vcenter{\hbox{$#2#3$}}\kern-.5\wd0}}

\expandafter\let\expandafter\ltxt\csname OT1\string\t\endcsname

\renewcommand\t{\theta}

\newcommand{\bPsi}{\mbox{\boldmath$\Psi$}}
\providecommand{\cZ}{\mathcal{Z}}
\providecommand{\cX}{\mathcal{X}}
\providecommand{\dd}{\mathrm{d}}
\providecommand{\cJ}{\mathcal{J}}
\newcommand{\bsigma}{\mbox{\boldmath$\sigma$}}
\newcommand{\btau}{\mbox{\boldmath$\tau$}}

\def\bb{{\bf b}}
\def\be{{\bf e}}
\def\bff{{\bf f}}

\def\bn{{\bf n}}

\def\bt{{\bf t}}
\def\bu{{\bf u}}
\def\bv{{\bf v}}
\def\bw{{\bf w}}
\def\bzz{{\bf z}}

\def\bD{{\bf D}}
\def\bE{{\bf E}}
\def\bF{{\bf F}}
\def\bH{{\bf H}}
\def\bW{{\bf W}}

\def\cA{{\mathcal{A}}}

\def\cE{{\mathcal{E}}}

\def\cJ{{\mathcal{J}}}

\def\cL{{\mathcal{L}}}
\def\cP{{\mathcal{P}}}
\def\cQ{{\mathcal{Q}}}

\def\cU{{\mathcal{U}}}

\def\p{\partial}

\def\divv{\operatorname{div}}   %

\newcommand{\grad}{\operatorname{grad}}

\DeclareMathOperator{\Span}{Span}
\def\divvr{\operatorname{div}}
\def\curll{\operatorname{curl}}
\def\curl{\operatorname{\mathbf{curl}}}

\newcommand{\tri}{|\!|\!|}

\newcommand{\la}{\langle}
\newcommand{\ra}{\rangle}

\newcommand{\beq}{\begin{equation}}
\newcommand{\eeq}{\end{equation}}

\newcommand{\Hdiv}{H(\divvr;\,\Omega)}

\def\grad{\operatorname{grad}}   %

\def\cA{{\mathcal A}}

\def\cH{{\mathcal{H}}}

\def\cN{{\mathcal{N}}}

\def\cE{{\mathcal{E}}}
\def\cP{{\mathcal{P}}}

\def\cL{{\mathcal{L}}}
\def\cJ{{\mathcal{J}}}

\def\R{\mathbb{R}}

\newcommand{\jump}[1]{[\![ #1]\!]}

\newcommand{\LS}{{\mathtt{LS}}}

\numberwithin{equation}{section}

\newcommand{\hdn}[1]{\|#1\|_{H(\divvr)}}
\newcommand{\bbet}{\boldsymbol{\beta}}
\newcommand{\one}{\mathbf{1}}
\newcommand{\bh}{\mathbf{h}}
\newcommand{\Rop}{\mathcal{R}}

\title[Native de Rham neural classes and FOSLS-\deRhaNN]{FOSLS-\deRhaNN: native de Rham neural classes for \texorpdfstring{$H(\divvr)$}{H(div)} and \texorpdfstring{$H(\curll)$}{H(curl)} with applications to first-order system least-squares neural network methods for partial differential equations}
\author[S. Zhang]{Shun Zhang}
\address{Department of Mathematics, City University of Hong Kong, Kowloon Tong, Hong Kong, China}
\email{shun.zhang@cityu.edu.hk}
\thanks{This work was supported in part by
Research Grants Council of the Hong Kong SAR, China, under the GRF Grant Project No. CityU 11305025}
\date{\today}

\hypersetup{pdftitle={FOSLS-deRhaNN: native de Rham neural classes for H(div) and H(curl) with applications to first-order system least-squares neural network methods for partial differential equations}, pdfauthor={Shun Zhang}}
\DeclareRobustCommand{\deRhaNN}{deRhaNN}   %
\begin{document}

\begin{abstract}
We construct neural network approximation classes that are native to the graph spaces $H(\operatorname{div})$ and $H(\operatorname{curl})$, in two and three dimensions and, for
$H(\operatorname{div})$, in arbitrary dimension. Every realization lies in the space for every value of the
trainable parameters, and with kinked potentials, such as ReLU networks, and with the interface modules, the admissible jumps of the space are
present at finite width. The classes are images of ordinary scalar
and componentwise networks under fixed operators of the continuous de Rham complex, a Poincar\'e right inverse of
the divergence and a regular decomposition, and they do not involve a mesh, a
projection, an interpolation or a finite element emulation. For $H(\operatorname{div})$ in $\mathbb{R}^n$ two
native classes are given on an equal footing, with a skew-symmetric matrix potential $A$: $\operatorname{Div}A+\mathcal{R}_nq+\mathbf{h}$, in which
the divergence $q$ is an explicit unknown, and $\operatorname{Div}A+\mathbf{z}$ with an $H^1$ field $\mathbf{z}$, from a regular decomposition, which is evaluated
pointwise; for $H(\operatorname{curl})$ the analogous classes are $\operatorname{grad}\phi+\mathcal{S}r+\mathbf{h}$ in two dimensions and $\operatorname{grad}\phi+\mathbf{z}$ in two and three dimensions. In all of them every interface jump
of the field is carried by the potential term, $\operatorname{Div}A$ or $\operatorname{grad}\phi$, while the remaining part has
no interface jump (in the regular-decomposition classes it is an $H^1$ field). The classes with $\mathbf{z}$ are
the componentwise approach enriched by this term. Known or learned interface geometry enters the potential through factors with
trainable amplitudes, and, when the divergence also jumps, the remaining part as well. These native graph-space
classes lead to the FOSLS-deRhaNN method, first-order system least squares with de Rham neural networks, in which
the graph-space unknowns are represented by the native neural classes and the loss is the first-order system least-squares functional posed
in the natural spaces of the weak formulation; for second-order elliptic equations the formulation works
at the level of the weak problem, including $H^{-1}$ right-hand sides and $H^{1/2}$ Dirichlet data. Second-order elliptic equations with
discontinuous coefficients and curl--curl problems are treated as instances, in which the functional is equivalent to the error; linear
transport with discontinuous solutions and scalar conservation laws with shocks are treated with the same native flux classes.
\end{abstract}

\maketitle

\section{Introduction}\label{sec:intro}

Many physical problems are naturally posed as first-order systems in $\grad$, $\divv$ and
$\curll$: a constitutive law relates a flux to the gradient of a
potential, as in Darcy's law, Fourier's law of heat conduction and
Ohm's law, and a conservation law prescribes the divergence of the flux; Maxwell's equations and
linear elasticity, with the stress as flux, admit analogous first-order
formulations in compatible graph spaces. The second-order equation arises by eliminating
the flux, which differentiates the coefficient. With discontinuous coefficients the first-order system is the natural
form: it holds in the whole domain, with the flux in $H(\divvr)$ and the
gradient in $H(\curll)$, and the interface conditions, continuity of the normal flux and of
the tangential gradient, are the intrinsic continuity of these spaces, whereas
the classical strong second-order equation holds on the coefficient subdomains, together with the
transmission conditions. Neural network methods for partial differential equations,
going back to \cite{LLF:98}, were transformed by physics-informed neural networks (PINNs) \cite{PINN}: the
residual of the equation at collocation points is the loss, the
method is mesh free, and it treats forward and
inverse problems, nonlinear equations and the assimilation of data in one framework; it has become
one of the most widely used tools of scientific machine learning \cite{KKLPWY:21}.
In its vanilla form the residual is that of the strong form, which
for a second-order divergence-form equation requires second derivatives of the trial function
and pointwise differentiation of the coefficient; for nonsmooth activations
these derivatives exist only piecewise, and discontinuous coefficients require an explicit treatment of
the interface conditions. The flux is not an unknown, and the
residual is defined only for data more regular than those
of the weak formulation. The deep Ritz method \cite{EY:18} minimizes the variational energy
instead and is therefore tied to problems with a coercive energy, in particular symmetric elliptic problems.

First-order system least squares (FOSLS) \cite{CLMM:94,CMM:97,Jiang:98,BG:09} keeps the structure of the physical laws: the potential and the flux are both unknowns, each in its natural space, and the constitutive law and the conservation law are separate residuals. It works at the level of the weak theory of the equation. For general second-order elliptic equations the coefficients need only be bounded, with $\mathbf A$ uniformly positive definite, the right-hand side is any element of $H^{-1}$, written as
$f_1+\divv\bff_2$ with $f_1\in L^2$ and $\bff_2\in L^2(\Omega)^n$ (for instance $\bff_2$ discontinuous, or a source on a curve), the Dirichlet data are in $H^{1/2}$, and the only assumption is that the weak problem is uniquely solvable; under exactly these conditions the functional is equivalent to the error in the natural norms, also for nonsymmetric operators \cite{Zhang:23}, and it is an a posteriori error estimator \cite{LiZhang:25,WuZhang:26}. The datum $\bff_2$
enters through the constitutive law, pointwise and never differentiated (Remark~\ref{rem:Hminus1}), so the functional is defined for all data of the weak theory, including those for which no strong form exists. This makes it a natural loss for neural network methods: minimizing it is a minimization at the level of the weak formulation, with the flux and its conservation law as part of the problem.

\subsection{Natural graph spaces and intrinsic continuity}\label{subsec:native}

In a FOSLS formulation the unknowns lie in
$H^1$, $H(\divvr)$ and $H(\curll)$, and the graph space of each
unknown, determined by which distributional derivatives are required
to lie in $L^2$, determines its intrinsic continuity. For
a field that is smooth on the elements of a
partition, and an interior face $F$ with unit
normal $\bn_F$, the singular parts of the derivatives across $F$ are, schematically,
\beq\label{distributional_jumps}
 (\grad u)_{\rm sing}=\jump{u}\,\bn_F\,\delta_F,\qquad
 (\curll\bE)_{\rm sing}=\bn_F\times\jump{\bE}\,\delta_F\ \ (n=3),\qquad
 (\divv\btau)_{\rm sing}=\jump{\btau\cdot\bn_F}\,\delta_F .
\eeq
An $L^2$ derivative excludes exactly the corresponding surface
measure, so the intrinsic continuity across $F$ is $\jump{u}=0$
in $H^1(\Omega)$, $\bn_F\times\jump{\bE}=0$ in $H(\curll;\Omega)$ and $\jump{\btau\cdot\bn_F}=0$ in
$H(\divvr;\Omega)$. In two dimensions the $H(\curll)$ condition is continuity
of the scalar tangential component; in the language
of forms, $H\Lambda^k(\Omega)$ requires a single-valued tangential trace
of the $k$-form. For the second-order equation \eqref{pde}
with discontinuous $\mathbf A$ and the weak data
of Section~\ref{sec:problem}, the flux variable $\bsigma=-\mathbf A\grad u+\bbet u-\bff_2$
is in $H(\divvr)$, while the physical constitutive flux
$-\mathbf A\grad u+\bbet u=\bsigma+\bff_2$ need not be when
$\bff_2$ jumps; $\bsigma$ is in general neither in $H(\curll)$
nor in $H^1(\Omega)^n$, and $\grad u$ is in $H(\curll)$
but in general neither in $H(\divvr)$ nor in
$H^1(\Omega)^n$ \cite{CZ:09,CZ:10a,CHZ:17zz}: the normal component of $\bsigma$ and
the tangential component of $\grad u$ are continuous
across interfaces, the other components jump. This is the
reason for the classical finite element choices, defined
through face and edge degrees of freedom on
a mesh: continuous Lagrange elements reproduce the continuity of
$H^1$, Raviart--Thomas and BDM elements the normal continuity of $H(\divvr)$
while allowing tangential jumps, and N\'ed\'elec elements the
tangential continuity of $H(\curll)$ while allowing normal jumps \cite{BBF:13,Gatica:2014,Hip:02,AFW:2006,AFW:10,Arnold:2018}.

Standard networks represent vector fields componentwise. A scalar network with a Lipschitz
activation is Lipschitz on a bounded domain and hence in
$H^1(\Omega)$, so a componentwise network lies in $H^1(\Omega)^n\subset H(\divvr;\Omega)\cap H(\curll;\Omega)$: it is
conforming, but it imposes more continuity than either graph space requires. We use
\emph{native} in this structural sense: the approximation respects the intrinsic continuity of the graph
space rather than entering it only through a stronger componentwise $H^1$ regularity. The
distinction is visible at finite complexity. A target $\btau\in H(\divvr)$ may have
a genuine tangential jump with continuous normal trace, and $\bE\in H(\curll)$ a genuine normal
jump with continuous tangential trace. A componentwise continuous approximation contains no
such field at any finite size, since a network with a continuous activation
is continuous: it can only replace the jump by a continuous transition layer
whose width tends to zero as the approximation is sharpened. Least-squares ReLU networks
approximate discontinuous scalar functions well in this way \cite{CCL:21,CCL:24}; the discontinuity is then
obtained by training, not by the construction of the class. The classes of this
paper supply the missing jumps by a potential term from the de Rham complex.

\subsection{de Rham representations}\label{subsec:realizations}

We therefore start from the graph space of the
unknown, not from the network: ordinary networks approximate $H^1$
or $L^2$ functions, and fixed differential and right-inverse
operators associated with the de Rham complex assemble them into
the graph-space field. For $0\le k\le n$ let
$H\Lambda^k(\Omega):=\{\omega\in L^2\Lambda^k(\Omega):\dd\omega\in L^2\Lambda^{k+1}(\Omega)\}$. The de Rham complex is
\beq\label{derham_principle}
 H\Lambda^0(\Omega)\xrightarrow{\ \dd\ }H\Lambda^1(\Omega)\xrightarrow{\ \dd\ }\cdots\xrightarrow{\ \dd\ }H\Lambda^{n-1}(\Omega)\xrightarrow{\ \dd\ }H\Lambda^n(\Omega),
\eeq
with $H\Lambda^0=H^1$ and $H\Lambda^n\simeq L^2$; in vector proxies $H\Lambda^{n-1}\simeq H(\divvr)$ in any dimension, and in three dimensions $H\Lambda^1\simeq H(\curll)$. The identity
$\dd^2=0$ gives $\operatorname{ran}\dd\subset\ker\dd$, the finite-dimensional cohomology measures the failure of exactness, and regular-potential results allow potentials with $H^1$ coefficients.

The $H(\divvr)$ construction uses the two maps adjacent to $H\Lambda^{n-1}$, and it leads to two representations, which we treat on an equal footing. The map on the
right is $\divv:H(\divvr;\Omega)\to L^2(\Omega)$, so $q:=\divv\btau$ is a natural unknown, and a bounded right inverse gives $\Rop_nq\in H(\divvr)$ with $\divv\Rop_nq=q$. What remains is divergence free, and the map
on the left represents this kernel, modulo cohomology, by a potential. Alternatively, the divergence can be left to an $H^1$ field, by a regular decomposition. Hence
\beq\label{hdiv_principle}
 \btau=\operatorname{Div}A+\Rop_nq+\bh,\quad q=\divv\btau,\qquad\text{or}\qquad \btau=\operatorname{Div}A+\bzz,\quad \bzz\in H^1(\Omega)^n,\ \ \divv\bzz=\divv\btau ,
\eeq
with $A$ a skew-symmetric matrix field with $H^1$ entries and $\bh$ in a finite-dimensional cohomology space; in two dimensions $\operatorname{Div}A=\curl\psi$, and $\btau=\curl\psi+\Rop_2 q+\bh$
or $\btau=\curl\psi+\bzz$. The first keeps the divergence as an explicit $L^2$ unknown: an equation for it can be built into the class, and
a jump of the divergence is integrated by $\Rop_n$ into a continuous field (under the segment condition of Lemma~\ref{lem:Rcont}). The second needs no line integrals and no harmonic fields:
it is the componentwise approach of \cite{CCLL:20,LC:26p2nn}, $\bzz$ a componentwise network, enriched by the divergence-free field $\operatorname{Div}A$, which carries the tangential jumps that
a componentwise $H^1$ field cannot; a jump of the divergence is carried by a kink
of $\bzz$. We call such a formula a \emph{de Rham representation} of the field, and the
auxiliary functions in it its \emph{representing functions}. For the first they are the potential $A$, the divergence $q=\divv\btau$ and the cohomology part $\bh=\Pi_{\cH_n}\btau$ (Sections~\ref{subsec:hdivd}
and~\ref{subsec:hdiv_cohom}); $q$ and $\bh$ are determined by $\btau$, and $A$ modulo $\ker\operatorname{Div}$, or uniquely once the bounded
linear lifting of Section~\ref{subsec:hdiv_split} is fixed. For the second they are $A$ and $\bzz$ (Section~\ref{subsec:hdiv_regular}).

For $H(\curll)$ the left side is the same, $H^1(\Omega)\xrightarrow{\grad}H(\curll;\Omega)\xrightarrow{\curll}H(\divvr;\Omega)$ in
three dimensions, with the gradient generating the exact curl-free part;
but the curl cannot be an unconstrained $L^2$ function, since
$\divv(\curll\bE)=0$. In general, the exterior derivative $\dd\omega$ in an intermediate degree is constrained by
$\dd(\dd\omega)=0$, which a generic network does not satisfy. For $H(\curll)$
in three dimensions we therefore use the regular decomposition
\beq\label{regular_principle}
 \bE=\grad\phi+\bzz,\qquad \phi\in H^1(\Omega),\quad \bzz\in H^1(\Omega)^3 .
\eeq
No separate harmonic part is needed in this form, since any topological component is absorbed into
the $H^1$ remainder $\bzz$. In two dimensions $H(\curll)$ is a rotation of $H(\divvr)$, and both realizations are
available, $\bE=\grad\phi+\mathcal Sr+\bh$ with a rotated right inverse $\mathcal S$ of the curl, and $\bE=\grad\phi+\bzz$. The
second $H(\divvr)$ representation and \eqref{regular_principle} are the same construction in two degrees, $\btau=\operatorname{Div}A+\bzz$ in degree $n-1$
and $\bE=\grad\phi+\bzz$ in degree $1$. Only in degree $n-1$, where the exterior derivative lands in top
degree, can the exterior derivative itself be an unconstrained trainable function; this gives the first $H(\divvr)$ representation.

The trainable functions are the representing functions in \eqref{hdiv_principle} and \eqref{regular_principle}. A scalar network with a Lipschitz activation lies in $W^{1,\infty}(\Omega)\subset H^1(\Omega)$; continuous piecewise affine finite element functions are
ReLU networks of a size fixed by the mesh \cite{HLXZ:20}, so every finite element approximation of a representing function is available at
a corresponding network size, and for smooth activations classical Sobolev approximation results apply \cite{Hornik1991}. For $H(\divvr)$ we approximate $A$ in $H^1$ and $q$ in $L^2$ and assemble $\btau_{\theta,\eta}=\operatorname{Div}A_\theta+\Rop_nq_\eta+\bh$, or $A$ and $\bzz$ in
$H^1$ and assemble $\btau_{\theta,\eta}=\operatorname{Div}A_\theta+\bzz_\eta$; for $H(\curll)$ we approximate $\phi$ and $\bzz$ in $H^1$ and assemble $\bE_{\theta,\eta}=\grad\phi_\theta+\bzz_\eta$ for $n=2,3$, or, for $n=2$, $\phi$ in $H^1$ and $r$ in $L^2$ and assemble
$\bE_{\theta,\eta}=\grad\phi_\theta+\mathcal Sr_\eta+\bh$. The identities $\divv\btau_{\theta,\eta}=q_\eta$, $\divv\btau_{\theta,\eta}=\divv\bzz_\eta$ and $\curll\bE_{\theta,\eta}=\curll\bzz_\eta$ hold for every value of the trainable parameters, so conformity is produced by the continuous de Rham
structure and not enforced by training. The classes do not involve a mesh, a projection or a finite element interpolation. It remains to show, for each
graph space, that the representation map is onto, with representing functions bounded by the field, and that approximation of
the representing functions transfers to the graph norm; this is done for $H(\divvr)$ in Sections~\ref{sec:hdiv} and~\ref{sec:hdiv_reg} and for $H(\curll)$ in Section~\ref{sec:hcurl}.

In all these representations every interface jump of the field is carried by
its first part, the potential term $\operatorname{Div}A$ or $\grad\phi$ given by the map
on the left of the complex: $\Rop_nq+\bh$ and $\mathcal Sr+\bh$ are continuous for bounded piecewise
continuous $q$ and $r$ under the segment condition of Lemma~\ref{lem:Rcont}, and the $H^1$ remainder $\bzz$ has a single-valued
trace across an interior interface. The continuity of the field is
weaker than that of its representing functions only through the potential. In two
dimensions a continuous piecewise affine potential $\psi_\theta$ has a piecewise constant $\curl\psi_\theta$ with continuous
normal trace and possibly a tangential jump; in any dimension a continuous
piecewise affine $\phi_\theta$ has a piecewise constant $\grad\phi_\theta$ with continuous tangential trace and
possibly a normal jump. This is how the finite-parameter realizations reproduce the intrinsic continuity of
$H(\divvr)$ and $H(\curll)$ rather than the stronger continuity of a
componentwise $H^1$ network, and it is why the
classes with $\bzz$ are the componentwise approach enriched by the
potential term. The interface geometry of the jump of the field, known
or learned, is accordingly carried by the potential, through a factor $|d_\Gamma|$ or $\operatorname{ReLU}(g_\eta)$
with trainable amplitudes (Sections~\ref{subsec:interface} and~\ref{subsec:learned_interface}); when the divergence also jumps, the same interface information
is used in the remaining part, in $q$ or in a kinked $\bzz$ (Section~\ref{subsec:div_jumps}).

\subsection{Related work and outline}\label{subsec:related}

Least-squares neural network methods were developed systematically by Cai
and coauthors. The deep least-squares method \cite{CCLL:20} uses
first-order system least-squares functionals as the loss for
elliptic problems, with the flux $\bsigma\in H(\divvr)$ approximated
by a componentwise network and its divergence obtained
by differentiating the components. In the physics-preserving neural
network (P$^2$NN) methods for elliptic equations \cite{LC:26p2nn} the flux
is again a componentwise network, and its divergence
is replaced by the discrete divergence operator of \cite{CCL:23_divergence}, which is
based on the definition of the divergence through the Gauss
theorem, $\divv\btau(x)=\lim_{\epsilon\to0}|B_\epsilon(x)|^{-1}\int_{\partial B_\epsilon(x)}\btau\cdot\bn\,dS$: at each integration point the outward flux
through the boundary of a control volume, computed
by quadrature, is divided by its measure. The second
class of this paper keeps such a componentwise
flux and enriches it by a potential term.
The least-squares ReLU neural network (LSNN) method for
linear advection--reaction equations \cite{CCL:21} showed that ReLU networks approximate discontinuous
solutions along an interface without oscillation and without
artificial diffusion; the network structure required for discontinuity interfaces
is analyzed in \cite{CCL:24}, and non-constant jumps along curved
interfaces in \cite{CCL:24jumps}. The method was extended to scalar
nonlinear hyperbolic conservation laws \cite{CCL:22nonlinear}, where the discrete divergence
operator of \cite{CCL:23_divergence} computes the divergence of the
flux from its values in the conservation form; the
LSNN method for hyperbolic problems is surveyed in
\cite{LC:26lsnnchapter}. The computational tools of this line of work
are used throughout the present paper: adaptive network
enhancement \cite{LCC:22ane1,LC:22ane2,CL:24adaptiveNN}, which adds neurons where a local
error indicator is large; the separation of the linear
output parameters from the nonlinear hidden parameters, with structure-guided Gauss--Newton
and block Newton iterations \cite{CDLLX:25sggn,Cai_shallowRitz1D:24,Cai_shallowRitzRD1D:24,CDFH:26bn}; derivatives of networks
through moments computed from values, as in the
discrete divergence operator; and adaptive quadrature \cite{LCR:23}. Dual neural
network formulations are developed in \cite{LCR:25}. The present
paper adds to this framework native representations of
graph-space fields, in particular $H(\divvr)$ fluxes and $H(\curll)$
fields: each field is represented so that its intrinsic
continuity holds for every parameter value, and its jumps
are carried by the kinks of its potentials.

Compatible structure has also been built into neural networks in
other ways. Networks with ReLU and binary step
activations represent the lowest-order Raviart--Thomas and N\'ed\'elec spaces exactly
\cite{DeRhamDNN:23}. First-order system least squares with deep neural
networks is developed in \cite{OPS:24}, with networks that emulate these de Rham
compatible finite element spaces on a mesh; there the least-squares
residual is the loss and a computable error
estimator. Finite element interpolated networks combine networks with compatible
finite element spaces \cite{BLML24,BWM:25comp}, and exact discrete de
Rham complexes of fixed-neuron ReLU$^k$ networks are constructed in
\cite{HuWangXu2026}; these obtain the de Rham structure at the discrete level. In
all of them the de Rham structure is constructed,
not trained: it is that of a given mesh
\cite{DeRhamDNN:23,OPS:24,BLML24,BWM:25comp} or of a given set of neurons
\cite{HuWangXu2026}, so the partition on which jumps can occur
is fixed in advance. In the present classes
the structure comes from the continuous de Rham operators, and
conformity holds for every value of all trainable
parameters, including those that place the kinks and
the learned interfaces. In the deep Ritz setting, scalar,
vector and antisymmetric tensor potentials from the de Rham
complex are used in \cite{ChenJiYuZhang2026} for Poisson problems
in arbitrary dimension without boundary penalties, and divergence-free fields
whose potentials have scalar networks as entries are
used in \cite{RichterPowellLipmanChen2022}, with approximation rates for linearized ReLU$^k$
networks in \cite{HeLiuTian2026}. The representation of divergence-free fields by
potentials (stream function, vector potential, and in general
dimension a skew-symmetric matrix of potentials) is classical, also for
finite elements \cite{BGHRR:25,AFW:2006,AFW:10}. In the present classes these
potentials form the divergence-free part of the $H(\divvr)$
representation, together with the Poincar\'e right inverse of
the divergence \cite{Cartan1967,CostabelMcIntosh2010} and the cohomology part; for Poincar\'e operators,
regular decompositions and potentials in the language of differential complexes see
the survey \cite{Hu:26review}. In one dimension $H(\divvr;I)=H^1(I)$, and
the first-order system least-squares shallow neural network method
for general second-order elliptic equations is developed in \cite{LZhang:25}.

This paper constructs neural approximation classes native to $H(\divvr)$ and $H(\curll)$, in two and three dimensions and, for $H(\divvr)$, in arbitrary dimension: every realization lies in the space for every value of the trainable parameters, and, with kinked potentials
such as ReLU networks and with the interface modules, the admissible jump structure of the space is present already at finite width (Propositions~\ref{prop:jump} and~\ref{prop:hcurl_jump}). The classes are images of scalar and componentwise networks under fixed operators,
\[
 \btau=\curl\psi+\Rop_2 q+\bh\ \ \text{or}\ \ \btau=\curl\psi+\bzz\ \ (n=2),\qquad
 \btau=\operatorname{Div}A+\Rop_nq+\bh\ \ \text{or}\ \ \btau=\operatorname{Div}A+\bzz\ \ (\R^n),
\]
\[
 \bE=\grad\phi+\mathcal Sr+\bh_{\curll}\ \ \text{or}\ \ \bE=\grad\phi+\bzz\ \ (n=2),\qquad
 \bE=\grad\phi+\bzz\ \ (n=3),
\]
with scalar networks, a skew-symmetric matrix of scalar networks $A$, a
componentwise network $\bzz$, and finite-dimensional harmonic parts; the field
being approximated need not lie in $H^1(\Omega)^n$, only its representing functions
do. The native $H(\divvr)$ and $H(\curll)$ classes, which we call \emph{deRhaNN classes} (de Rham
neural network classes), are combined with first-order system least squares to obtain the \emph{FOSLS-deRhaNN} method, first-order
system least squares with de Rham neural networks.\footnote{In the name deRhaNN the
final m of de Rham is replaced by NN, for neural
networks.} The approximation classes carry the graph-space conformity by construction, while the
FOSLS functional is posed at the natural regularity level of
the weak problem; the least-squares part is standard: the
error is bounded by the best approximation error of the
classes (Proposition~\ref{prop:best_approx}), with the regularity of the weak solution and no more.

Sections~\ref{sec:hdiv}--\ref{sec:hdiv2} carry out the construction for $H(\divvr)$ in arbitrary
dimension: Section~\ref{sec:hdiv} the stable de Rham representation with a right inverse
of the divergence and the first neural class, Section~\ref{sec:hdiv_reg} the regular decomposition and the
second class, both with two and three dimensions as special cases, and
Section~\ref{sec:hdiv2} the jumps at finite width, in the flux and in its divergence, for both classes,
and the three levels of interface information. Section~\ref{sec:hcurl} does the same for $H(\curll)$ in
two and three dimensions. Section~\ref{sec:framework} states the FOSLS-deRhaNN method with the standard best-approximation
estimate; Sections~\ref{sec:method} and~\ref{sec:curlcurl} are instances (second-order elliptic equations, posed at the level
of the weak formulation, and curl--curl problems), each with numerical examples (Sections~\ref{subsec:ell_numerics} and~\ref{subsec:curl_numerics}). Section~\ref{sec:transport} treats
linear transport, whose functional is the error in a norm weaker than that of the natural spaces, and Section~\ref{sec:conslaw}
uses the native space-time flux classes in the flux least-squares formulation of
scalar conservation laws, where the interface is an unknown shock.

\section{Native neural approximation in \texorpdfstring{$H(\divvr)$}{H(div)} I: right inverse and stable de Rham representation}\label{sec:hdiv}

\subsection{Motivation: the de Rham sequence}\label{subsec:hdiv_derham}

In two and three dimensions the end of the de Rham sequence is
\[
 H^1(\Omega)\xrightarrow{\ \curl\ }H(\divvr;\Omega)\xrightarrow{\ \divv\ }L^2(\Omega)\longrightarrow0\quad(n=2),
\]
\[
 H^1(\Omega)^3\xrightarrow{\ \curll\ }H(\divvr;\Omega)\xrightarrow{\ \divv\ }L^2(\Omega)\longrightarrow0\quad(n=3),
\]
with $\curl\psi=(\p_2\psi,-\p_1\psi)$ in two dimensions. Two facts give the representation of a flux $\btau\in H(\divvr;\Omega)$. First, $\divv$ is onto
$L^2(\Omega)$ and has a bounded right inverse $\Rop$, $\divv\Rop q=q$; so $q:=\divv\btau$ can be taken as an independent $L^2$ function, and
$\btau-\Rop q$ is divergence free. Second, $\divv\curl=0$, and conversely a divergence-free field is the curl of a potential, up to finitely many fields determined
by the topology of $\Omega$ (none if $\Omega$ is simply connected for $n=2$, or if $\partial\Omega$ is connected for $n=3$). Hence
\[
 \btau=\curl\psi+\Rop q+\bh,\qquad q=\divv\btau ,
\]
with a potential $\psi$ in $H^1$, scalar for $n=2$ and vector valued
for $n=3$, the divergence $q$ in $L^2$, and $\bh$ in
a finite-dimensional space. The two maps next to $H(\divvr)$ enter in
different roles: the curl on the left describes the divergence-free part
by a potential, and the divergence on the right contributes its range through
$\Rop$. Networks for $\psi$ and $q$ then give a flux that lies
in $H(\divvr;\Omega)$ for every value of the parameters, with $\divv\btau=q$.
The potential, $q$ and $\bh$ are the representing functions of $\btau$
(Section~\ref{subsec:realizations}), from which the field is assembled by fixed
differential and right-inverse operators associated with the de Rham complex.
The functions $q=\divv\btau$ and $\bh$ are determined by $\btau$;
the potential is determined up to the kernel of the curl,
and uniquely once a bounded linear lifting is fixed; \emph{stable}
refers to this bounded reconstruction and to the equivalence of
the representation norm with the graph norm (Section~\ref{subsec:hdiv_split}). In this
representation $\Rop q$ carries the divergence of $\btau$, $\bh$ its
topological part, and the potential the divergence-free remainder; in the piecewise regular interface
settings considered below, where $\Rop q+\bh$ is continuous across the
interface, the tangential jump of $\btau$ is carried entirely
by the potential. Section~\ref{subsec:hdivd} gives the potential in $\R^n$,
where the curl becomes the row-wise divergence of a skew-symmetric matrix;
this part is used by both representations. Section~\ref{sec:R} constructs $\Rop$, Section~\ref{subsec:hdiv_cohom} the cohomology part,
Section~\ref{subsec:hdiv_split} the stable splitting, and Section~\ref{subsec:hdiv_classes} defines the neural classes.

\subsection{Skew-symmetric potentials}\label{subsec:hdivd}

Let $\mathbb K_n:=\{A\in\R^{n\times n}:A^T=-A\}$. A field $A\in H^1(\Omega;\mathbb K_n)$ is determined by its $n(n-1)/2$ entries $a_{ij}:=A_{ij}$, $i<j$; these are
the scalar functions that networks approximate, and the norms of $A$ count each of them once:
\beq\label{skew_norm}
 |A|_1^2:=\sum_{i<j}|a_{ij}|_1^2,\qquad \|A\|_1^2:=\sum_{i<j}\|a_{ij}\|_1^2 .
\eeq
With $\mathcal E_{ij}:=\be_i\be_j^T-\be_j\be_i^T$, $A=\sum_{i<j}a_{ij}\mathcal E_{ij}$. The row-wise divergence of $A$,
\beq\label{Div_def}
 (\operatorname{Div}A)_i:=\sum_{j=1}^n\p_jA_{ij},\qquad i=1,\ldots,n ,
\eeq
is the dimension-independent curl-potential operator. The entry $a_{ij}$ contributes $\operatorname{Div}(a_{ij}\mathcal E_{ij})=(\p_ja_{ij})\be_i-(\p_ia_{ij})\be_j$, and in two and three dimensions $\operatorname{Div}$
reduces to the scalar and the vector curl: for $n=2$, $A=\psi\mathcal E_{12}$ gives $\operatorname{Div}A=(\p_2\psi,-\p_1\psi)^T=\curl\psi$, and for $n=3$, $A_{ij}=\sum_k\varepsilon_{ijk}\Psi_k$ gives $\operatorname{Div}A=\curll\bPsi$.

The potential part $\operatorname{Div}A$ is common to both representations
of $H(\divvr)$ used in this paper. The following
facts are standard; we record them for completeness.

\begin{lem}[skew potentials]\label{lem:skew}
Let $A\in H^1(\Omega;\mathbb K_n)$.
\begin{enumerate}
\item[(a)] The field $\operatorname{Div}A\in L^2(\Omega)^n$ satisfies
\beq\label{divDiv_zero}
 \divv(\operatorname{Div}A)=0,\qquad \|\operatorname{Div}A\|_0\le\sqrt{n-1}\,|A|_1 ,
\eeq
where the first identity holds in the sense of distributions.
\item[(b)] Let $S\subset\Omega$ be a smooth hypersurface with unit normal $\bn$, and let $A$ be piecewise $H^2$ on the two
sides of $S$ in a neighborhood of $S$. Then
on $S$, with one-sided traces $v^\pm$ and $\jump{v}:=v^+-v^-$, where $\bn$
points into the $+$ side (this convention is used throughout),
\beq\label{Div_jump}
 \jump{\grad A_{ij}}=\jump{\p_{\bn}A_{ij}}\,\bn,\qquad \jump{\operatorname{Div}A}=\jump{\p_{\bn}A}\,\bn,\qquad \bn\cdot\jump{\operatorname{Div}A}=0 .
\eeq
\item[(c)] For every $\mathbf j\in\R^n$ with $\mathbf j\cdot\bn=0$, $K:=\mathbf j\bn^T-\bn\mathbf j^T\in\mathbb K_n$ satisfies $K\bn=\mathbf j$; thus
$K'\mapsto K'\bn$ maps $\mathbb K_n$ onto $\bn^\perp:=\{\mathbf v\in\R^n:\mathbf v\cdot\bn=0\}$. Its kernel
consists of the skew-symmetric maps of $\bn^\perp$ (extended by
$K'\bn=0$) and has dimension $(n-1)(n-2)/2$; it is trivial for $n=2$.
\end{enumerate}
\end{lem}

\begin{proof}
(a) For $\phi\in C_c^\infty(\Omega)$, $\langle\divv\operatorname{Div}A,\phi\rangle=\sum_{i,j}\int_\Omega A_{ij}\,\p_i\p_j\phi\,dx=0$ by antisymmetry. Since $A_{ii}=0$, $(\operatorname{Div}A)_i$ has $n-1$ terms, and pointwise $|\operatorname{Div}A|^2\le(n-1)\sum_i\sum_{j\ne i}|\p_jA_{ij}|^2=(n-1)\sum_{i<j}(|\p_ja_{ij}|^2+|\p_ia_{ij}|^2)\le(n-1)\sum_{i<j}|\grad a_{ij}|^2$. (b) Each $A_{ij}\in H^1(\Omega)$ has a single-valued
trace on $S$, so its tangential derivatives agree from the two sides and the jump of its gradient is normal. Hence $\jump{(\operatorname{Div}A)_i}=\sum_j\jump{\p_{\bn}A_{ij}}n_j$, and $\bn\cdot\jump{\p_{\bn}A}\bn=0$ because $\jump{\p_{\bn}A}$ is
skew. (c) $K\bn=\mathbf j\,(\bn\cdot\bn)-\bn\,(\mathbf j\cdot\bn)=\mathbf j$. If $K'\in\mathbb K_n$ and $K'\bn=0$, then $\bn^TK'=-(K'\bn)^T=0$, so $K'$ maps $\bn^\perp$ into itself and vanishes on $\bn$; conversely, every skew-symmetric
map of $\bn^\perp$, extended by $0$ on $\bn$, is in $\mathbb K_n$. The skew-symmetric maps of the $(n-1)$-dimensional space $\bn^\perp$ form a space of dimension $(n-1)(n-2)/2$.
\end{proof}

For $n=2$ the bound in \eqref{divDiv_zero} is an equality ($|\curl\psi|=|\grad\psi|$), for $n=3$ it is $\|\curll\bPsi\|_0\le\sqrt2|\bPsi|_1$.

\begin{rem}[differential forms and notation]\label{rem:forms}
Identify $A$ with the $2$-form $\sum_{i<j}A_{ij}\,dx_i\wedge dx_j$. Then $d(\star A)=(\operatorname{Div}A)\lrcorner\,\mathrm{vol}$: the $(n-1)$-form $d(\star A)$ corresponds to $\operatorname{Div}A$ in the identification $H(\divvr;\Omega)\simeq H\Lambda^{n-1}(\Omega)$, $\bv\mapsto\bv\lrcorner\,\mathrm{vol}$. Equivalently, $\operatorname{Div}A$ is the vector field of the $1$-form $d^*A$, where $d^*$ is the formal $L^2$-adjoint of $d$ (the codifferential): for every compactly supported $1$-form $\alpha=\sum_i\alpha_i\,dx_i$, $\int_\Omega\sum_{i<j}(\p_i\alpha_j-\p_j\alpha_i)A_{ij}\,dx=\int_\Omega\boldsymbol\alpha\cdot\operatorname{Div}A\,dx$ after integration by parts. The same operator is called $\mathrm{curl}$ in \cite[Section~2.4]{ChenJiYuZhang2026}. We keep the notation $\operatorname{Div}$, which shows that the operator is the row-wise divergence of a skew-symmetric matrix and that $\divv\operatorname{Div}A=0$.
\end{rem}

\subsection{A right inverse of the divergence}\label{sec:R}

In the classes below a flux is written as $\btau=\btau_0+\Rop_nq$ with $\btau_0$ divergence free and $\Rop_n$
a fixed right inverse of the divergence, so that $\divv\btau=q$. We construct $\Rop_n$ explicitly as an integral of
$q$ along the segments from a fixed center; for piecewise continuous $q$ the field $\Rop_nq$ is continuous
under a mild condition on the segments, so that it carries no jump.

Let $\Omega\subset\R^n$, $n\ge2$, be a bounded Lipschitz domain. Fix a
center $x_0\in\R^n$; the operator below depends on it, and this dependence is
suppressed in the notation $\Rop_n$. The center enters the constants through $C_{x_0}:=\sup_{x\in\Omega}|x-x_0|$, the radius
of the smallest ball centered at $x_0$ that contains $\Omega$. In the interface geometries considered
below, $x_0$ is chosen so that the segment hypothesis of Lemma~\ref{lem:Rcont} holds; in particular, no
segment from $x_0$ lies along an interface or along a flat piece of
$\partial\Omega$. For a function $g$ on $\R^n$ let $G_g(x):=\int_0^1t^{n-1}g(x_0+t(x-x_0))\,dt$, $x\in\R^n$, and for
$q\in L^2(\Omega)$, with $\widetilde q$ its extension by zero to $\R^n$, define
\beq\label{Rd}
 (\Rop_nq)(x):=(x-x_0)\,G_{\widetilde q}(x)=(x-x_0)\int_0^1 t^{n-1}\,\widetilde q\bigl(x_0+t(x-x_0)\bigr)\,dt,\qquad x\in\R^n ;
\eeq
as an element of $H(\divvr;\Omega)$, $\Rop_nq$ is the restriction to
$\Omega$. The domain $\Omega$ need not be star-shaped with
respect to $x_0$: $\widetilde q$ is used along the whole
segment from $x_0$ to $x$, which may leave and
re-enter $\Omega$, and $x_0$ need not lie in $\Omega$.

\begin{lem}[radial right inverse]\label{lem:Rd}
For $q\in L^2(\Omega)$, $\Rop_nq\in H(\divvr;\Omega)$ and
\beq\label{Rd_bound}
 \divv\Rop_nq=q,\qquad \|\Rop_nq\|_0\le\frac{2C_{x_0}}{n}\|q\|_0,\qquad
 \|\Rop_nq\|_{H(\divvr)}\le\Bigl(1+\frac{4C_{x_0}^2}{n^2}\Bigr)^{1/2}\|q\|_0 .
\eeq
Moreover, \eqref{Rd} defines $\Rop_nq$ on all of $\R^n$, with $\Rop_nq\in L^2_{\rm loc}(\R^n)^n$ and $\divv\Rop_nq=\widetilde q$ in $\R^n$.
\end{lem}

\begin{proof}
\emph{Bound.} For $g\in L^2(\R^n)$ and $t\in(0,1]$ the substitution $y=x_0+t(x-x_0)$, $dx=t^{-n}dy$, gives $\|g(x_0+t(\cdot-x_0))\|_{0,\Omega}\le t^{-n/2}\|g\|_{0,\R^n}$, and Minkowski's integral inequality gives
\beq\label{G_bound}
 \|G_g\|_{0,\Omega}\le\int_0^1t^{n-1}\|g(x_0+t(\cdot-x_0))\|_{0,\Omega}\,dt\le\|g\|_{0,\R^n}\int_0^1t^{\,n-1-n/2}\,dt
 =\frac2n\|g\|_{0,\R^n}.
\eeq
Since $|x-x_0|\le C_{x_0}$ on $\Omega$, $\|(x-x_0)G_g\|_{0,\Omega}\le\frac{2C_{x_0}}{n}\|g\|_{0,\R^n}$; for $g=\widetilde
q$ this is the second bound, since $\|\widetilde q\|_{0,\R^n}=\|q\|_{0,\Omega}$.

\emph{Identity.} For $g\in C^1(\R^n)$, since $(x-x_0)\cdot\grad_x\bigl[g(x_0+t(x-x_0))\bigr]=t\,\frac{\mathrm d}{\mathrm dt}g(x_0+t(x-x_0))$, integration by parts in $t$ gives
\[
 (x-x_0)\cdot\grad G_g(x)=\int_0^1t^{n}\frac{\mathrm d}{\mathrm dt}g\bigl(x_0+t(x-x_0)\bigr)\,dt=g(x)-n\,G_g(x),
\]
hence $\divv\bigl((x-x_0)G_g\bigr)=n\,G_g+(x-x_0)\cdot\grad G_g=g$. For $q\in L^2(\Omega)$ take $g_k\in
C_c^\infty(\R^n)$ with $g_k\to\widetilde q$ in $L^2(\R^n)$. The bound,
applied to $g_k-\widetilde q$, gives $(x-x_0)G_{g_k}\to\Rop_nq$ in $L^2(\Omega)$, so for $\phi\in
C_c^\infty(\Omega)$ $-\int_\Omega\Rop_nq\cdot\grad\phi=\lim_{k\to\infty}\int_\Omega g_k\,\phi=\int_\Omega q\,\phi$, that is, $\divv\Rop_nq=q$ in
$\Omega$. The third bound follows from the first
two. The bound and the limit hold with
$\Omega$ replaced by any ball $B$ about $x_0$
(and $C_{x_0}$ by its radius), and $\phi\in C_c^\infty(B)$ gives
$-\int_B\Rop_nq\cdot\grad\phi=\int_B\widetilde q\,\phi$, that is, $\divv\Rop_nq=\widetilde q$ in $\R^n$.
\end{proof}

\begin{rem}[Poincar\'e homotopy]\label{rem:poincare}
Formula \eqref{Rd} is the Poincar\'e homotopy operator with a fixed center in top degree \cite{Cartan1967}, written for vector
fields: for the $n$-form $u=\widetilde q\,dx_1\wedge\cdots\wedge dx_n$ it is the vector proxy of $\int_0^1t^{n-1}(x-x_0)\lrcorner\,u(x_0+t(x-x_0))\,dt$, and $\divv\Rop_nq=q$ is the top-degree homotopy relation.
Costabel and McIntosh \cite{CostabelMcIntosh2010} average such fixed-center path integrals over the center with a
smooth weight \cite[(3.9)]{CostabelMcIntosh2010} to obtain operators with Sobolev mapping properties. Their \cite[(1.4)]{CostabelMcIntosh2010} is the familiar three-dimensional analogue for $\curll$.
\end{rem}

We now show that $\Rop_nq$ is continuous even when $q$ is only piecewise continuous, which is the situation in interface problems. Since $(\Rop_nq)(x)$ integrates $\widetilde q$ along the segment from $x_0$ to $x$, moving $x$ moves the segment, and the values of $\widetilde q$ on it change continuously except where the segment crosses $\partial\Omega$ or a surface across which $q$ jumps. If on every segment these crossings form a set of measure zero, they do not affect the integral, and $\Rop_nq$ is continuous.

\begin{lem}[continuity for piecewise continuous data]\label{lem:Rcont}
Let $S_q\subset\overline\Omega$ be closed, let $q\in L^\infty(\Omega)$ be continuous on
$\Omega\setminus S_q$, and let $S$ be a closed
set containing $\partial\Omega\cup S_q$, for instance $S=\partial\Omega\cup S_q$; then
$\widetilde q$ is continuous on $\R^n\setminus S$. Assume that
for every $x\in\overline\Omega\setminus\{x_0\}$ the open segment $\{x_0+t(x-x_0):0<t<1\}$ meets $S$
in a set of one-dimensional measure zero, for
instance in finitely many points. Then $(x-x_0)G_{\widetilde q}$
is continuous on $\overline\Omega$; in particular $\Rop_nq$ has no jumps.
\end{lem}

\begin{proof}
Let $x_k\to x$ in $\overline\Omega$. If $x=x_0$, then $|(x_k-x_0)G_{\widetilde q}(x_k)|\le\frac1n|x_k-x_0|\,\|q\|_{L^\infty}\to0$. If $x\ne x_0$, then $x_0+t(x-x_0)\notin S$ for almost every $t\in(0,1)$, and for such $t$, $\widetilde
q(x_0+t(x_k-x_0))\to\widetilde q(x_0+t(x-x_0))$ by the continuity of $\widetilde q$ on $\R^n\setminus S$; the integrand of $G_{\widetilde q}$ is bounded by $t^{n-1}\|q\|_{L^\infty}$, and dominated convergence gives $G_{\widetilde q}(x_k)\to G_{\widetilde q}(x)$.
\end{proof}

For a convex polygon or polyhedron with $x_0$ in its interior the condition on $\partial\Omega$ holds. When the divergence $q^*=\divv\bsigma^*$ of an exact flux is bounded and
piecewise continuous, as in the interface examples below, the lemma applies even if $q^*$ jumps across the interfaces; it also applies to a bounded
continuous network $q$. In both cases $\Rop_nq$ has no jump, so a tangential jump of a flux is carried by its potential (Section~\ref{subsec:hdiv_jumps}).

\subsection{Cohomology part}\label{subsec:hdiv_cohom}

We use the standard structure of the boundary of a bounded Lipschitz domain \cite[Section~2.1]{BGHRR:25}. Let $O_1,\ldots,O_b$ be the bounded components of $\R^n\setminus\overline\Omega$ (the holes
of $\Omega$ for $n=2$, the cavities for $n=3$); $\p\Omega$ is the disjoint union of the closed sets $\Gamma_0$, the boundary of the unbounded component, and $\Gamma_i:=\p O_i$, $i=1,\ldots,b$; $b=\dim\cH_n$
below is the number of these components. With the normal $\bn$ exterior to $\Omega$, which on $\Gamma_i$ points into $O_i$, the cavity fluxes
\[
 F_i(\bv):=\langle\bv\cdot\bn,1\rangle_{\Gamma_i},\qquad i=1,\ldots,b,
\]
are bounded on $H(\divvr;\Omega)$: the function equal to $1$ on $\Gamma_i$
and to $0$ on the other components lies in $H^{1/2}(\p\Omega)$, since the finitely many
compact boundary components are mutually separated, and the normal trace
is bounded from $H(\divvr;\Omega)$ to $H^{-1/2}(\p\Omega)$. Choose $z_i\in O_i$ and let
\beq\label{harmonic_fields}
 \bh_i(x):=\frac{x-z_i}{|x-z_i|^n},\qquad \cH_n:=\Span\{\bh_1,\ldots,\bh_b\},\qquad
 \Pi_{\cH_n}\bv:=-\frac1{|S^{n-1}|}\sum_{i=1}^bF_i(\bv)\,\bh_i
\eeq
($\cH_n=\{0\}$ if there are none). The $\bh_i$ are smooth
and divergence free on $\overline\Omega$, and $F_i(\bh_j)=-|S^{n-1}|\delta_{ij}$, since
$\bn$ is the inner normal of $O_i$; hence $F_j(\Pi_{\cH_n}\bv)=F_j(\bv)$:
$\Pi_{\cH_n}\bv$ is the element of $\cH_n$ with the cavity
fluxes of $\bv$, and the coefficients of the cohomology
component are determined explicitly by the cavity fluxes, $c_i=-F_i(\bv)/|S^{n-1}|$.
These coefficients represent the $(n-1)$st de Rham cohomology
of $\Omega$; in three dimensions the obstruction to a global
skew potential is therefore given by cavities, not by
handles (handles carry the first cohomology, which concerns curl-free
fields). No canonical choice of harmonic representatives is needed.

\begin{cor}[fluxes of the radial part]\label{cor:R_fluxes}
For $q\in L^2(\Omega)$, $F_i(\Rop_nq)=0$ for $i=1,\ldots,b$, and $\langle\Rop_nq\cdot\bn,1\rangle_{\Gamma_0}=\int_\Omega q\,dx$. In particular $\Pi_{\cH_n}\Rop_nq=0$.
\end{cor}

\begin{proof}
By Lemma~\ref{lem:Rd}, $\Rop_nq$ has a divergence in $L^2_{\rm loc}(\R^n)$, so its normal
traces from the two sides of $\Gamma_i$ coincide, and $\divv\Rop_nq=\widetilde q=0$
in $O_i$; since the outer normal of $O_i$ on $\Gamma_i$ is $-\bn$,
$F_i(\Rop_nq)=-\int_{O_i}\divv\Rop_nq\,dx=0$. The divergence theorem in $\Omega$ gives $\sum_{i=0}^b\langle\Rop_nq\cdot\bn,1\rangle_{\Gamma_i}=\int_\Omega q\,dx$.
\end{proof}

The radial part thus carries the source balance through the outer boundary and no cavity flux.

\subsection{Stable splitting and regular potentials}\label{subsec:hdiv_split}

The following splitting is classical in substance: for $n=2,3$ it
is the stream-function and vector-potential description of \cite[Lemmas~2.1.2 and~2.1.9]{BGHRR:25} together
with a right inverse of the divergence, and in general
it is the exactness of the de Rham complex
up to its cohomology \cite{CostabelMcIntosh2010}. We include a proof
because of the specific right inverse $\Rop_n$ and the linear lifting.

\begin{prop}[classical $H(\divvr)$ splitting]\label{prop:hdiv_split}
$P:=\Rop_n\divv$ is the bounded projection of $H(\divvr;\Omega)$ onto $\Rop_nL^2(\Omega)$ along $\ker\divv$, and
\beq\label{hdiv_split}
 H(\divvr;\Omega)=\bigl(\operatorname{Div}\,H^1(\Omega;\mathbb K_n)\oplus\cH_n\bigr)\oplus\Rop_nL^2(\Omega),
\eeq
where $\ker\divv=\operatorname{Div}\,H^1(\Omega;\mathbb K_n)\oplus\cH_n$ and
\beq\label{Div_range}
 \operatorname{Div}\,H^1(\Omega;\mathbb K_n)=\{\bv\in H(\divvr;\Omega):\ \divv\bv=0,\ F_i(\bv)=0,\ i=1,\ldots,b\}.
\eeq
Equivalently, every $\btau\in H(\divvr;\Omega)$ can be written
\beq\label{hdiv_d_representation}
 \btau=\operatorname{Div}A+\Pi_{\cH_n}\btau+\Rop_n\divv\btau ,
\eeq
that is, with $q=\divv\btau$ and $\bh=\Pi_{\cH_n}\btau$ (Corollary~\ref{cor:R_fluxes}), where $q$, $\bh$ and $\operatorname{Div}A=\btau-\Rop_nq-\bh$
are unique and the potential $A\in H^1(\Omega;\mathbb K_n)$ is unique up to
an element of the kernel of $\operatorname{Div}$. The potential can be fixed by a bounded
linear map: there is a bounded linear $\mathsf L$ from the space \eqref{Div_range}
to $H^1(\Omega;\mathbb K_n)$ with $\operatorname{Div}\,\mathsf L\bv=\bv$, and $A=\mathsf L(\btau-\Rop_nq-\bh)$ satisfies
\[
 \|A\|_1+\|\bh\|_0+\|q\|_0\le C\|\btau\|_{H(\divvr)} .
\]
\end{prop}

\begin{proof}
By Lemma~\ref{lem:Rd}, $P$ is bounded, $P^2=P$, $\ran P=\Rop_nL^2(\Omega)$ and $\ker P=\ker\divv$, so $H(\divvr;\Omega)=\ker\divv\oplus\Rop_nL^2(\Omega)$. It remains to show $\ker\divv=\operatorname{Div}\,H^1(\Omega;\mathbb K_n)\oplus\cH_n$ and \eqref{Div_range}, with $\|A\|_1+\|\bh\|_0\le C\|\bv\|_{H(\divvr)}$ for $\bv\in\ker\divv$. For
a general $\btau$ this is applied to $\bv:=\btau-\Rop_nq$, $q=\divv\btau$, for which $\Pi_{\cH_n}\bv=\Pi_{\cH_n}\btau$ by Corollary~\ref{cor:R_fluxes}; then $\|q\|_0\le\|\btau\|_{H(\divvr)}$ and $\|\bv\|_{H(\divvr)}\le C\|\btau\|_{H(\divvr)}$ by \eqref{Rd_bound}. Let $O_0:=B\setminus\overline{\Omega\cup O_1\cup\cdots\cup O_b}$ for an open ball $B\supset\overline\Omega$.

\emph{The sum is direct.} For $A\in H^1(\Omega;\mathbb K_n)$ let
$\widetilde A\in H^1(\R^n;\mathbb K_n)$ be a skew-symmetric extension:
each $a_{ij}$, $i<j$, is extended by a bounded extension
operator of the Lipschitz domain $\Omega$, and $\widetilde A_{ji}:=-\widetilde A_{ij}$.
By \eqref{divDiv_zero}, $\operatorname{Div}\widetilde A\in L^2(\R^n)^n$ is divergence free, so
its normal traces from the two sides of $\Gamma_i$
coincide, and $F_i(\operatorname{Div}A)=-\int_{O_i}\divv\operatorname{Div}\widetilde A\,dx=0$, $i=1,\ldots,b$. An $\bh\in\cH_n$ with zero cavity fluxes
vanishes, since $F_i(\bh_j)=-|S^{n-1}|\delta_{ij}$; hence $\operatorname{Div}\,H^1(\Omega;\mathbb K_n)\cap\cH_n=\{0\}$, and $\operatorname{Div}\,H^1(\Omega;\mathbb K_n)$
is contained in the right-hand side of \eqref{Div_range}.

\emph{Every divergence-free field is represented.} Let $\bv\in H(\divvr;\Omega)$ with $\divv\bv=0$, and let $\bh:=\Pi_{\cH_n}\bv$; then $\|\bh\|_0\le C\|\bv\|_{H(\divvr)}$ by the boundedness of the $F_i$. The field $\bw:=\bv-\bh$ is divergence free with zero flux through $\Gamma_1,\ldots,\Gamma_b$, and through $\Gamma_0$ as well, because the fluxes through all components sum to $\int_\Omega\divv\bw\,dx=0$. For $i=0,\ldots,b$ let $\phi_i\in H^1(O_i)$, unique up to a constant, solve the Neumann problem
\[
 \int_{O_i}\grad\phi_i\cdot\grad\chi\,dx=-\langle\bw\cdot\bn,\chi\rangle_{\Gamma_i}\qquad\text{for all }\chi\in H^1(O_i),
\]
that is, $\Delta\phi_i=0$ in $O_i$, $\grad\phi_i\cdot\bn=\bw\cdot\bn$ on $\Gamma_i$ (the outer
normal of $O_i$ there is $-\bn$) and, for $i=0$,
$\p_\nu\phi_0=0$ on $\p B$. The right-hand side vanishes
for $\chi=1$ because the flux of $\bw$ through $\Gamma_i$ is
zero, so $\phi_i$ exists, and $\|\grad\phi_i\|_{0,O_i}\le C\|\bw\cdot\bn\|_{H^{-1/2}(\Gamma_i)}\le C\|\bw\|_{H(\divvr)}$. The field
$\widetilde\bw$ equal to $\bw$ in $\Omega$ and to $\grad\phi_i$ in
$O_i$ has the same normal trace from both sides
of every $\Gamma_i$, so $\widetilde\bw\in L^2(B)^n$ is divergence free
in $B$, with $\|\widetilde\bw\|_{0,B}\le C\|\bv\|_{H(\divvr)}$. The $(n-1)$-form $\widetilde\bw\lrcorner\,\mathrm{vol}$ is
closed in $B$, and $B$ is starlike with respect
to a ball, so \cite[Proposition~4.1(i)]{CostabelMcIntosh2010} gives an $(n-2)$-form $\mu\in H^1(B,\Lambda^{n-2})$ with $d\mu=\widetilde\bw\lrcorner\,\mathrm{vol}$
and $\|\mu\|_{1,B}\le C\|\widetilde\bw\|_{0,B}$. Let $\widetilde A\in H^1(B;\mathbb K_n)$
be the skew-symmetric matrix of the $2$-form with $\star\widetilde A=\mu$.
By Remark~\ref{rem:forms}, $(\operatorname{Div}\widetilde A)\lrcorner\,\mathrm{vol}=d(\star\widetilde A)=\widetilde\bw\lrcorner\,\mathrm{vol}$, so $\operatorname{Div}\widetilde A=\widetilde\bw$, and
$\|\widetilde A\|_{1,B}\le C\|\widetilde\bw\|_{0,B}$. Then $A:=\widetilde A|_\Omega$ satisfies $\operatorname{Div}A=\bv-\bh$ and $\|A\|_1\le C\|\bv\|_{H(\divvr)}$. If
$F_i(\bv)=0$ for all $i$, then $\bh=0$ and $\bv=\operatorname{Div}A$, which gives the
other inclusion in \eqref{Div_range}. For $n=2$ and $n=3$ this
is the classical argument for the stream function
and the vector potential \cite[Lemmas~2.1.2 and~2.1.9]{BGHRR:25}. Every step is linear in $\bv$:
the Neumann problems with $\phi_i$ normalized by zero mean, the extension
$\widetilde\bw$, and the integral operator of \cite[Proposition~4.1(i)]{CostabelMcIntosh2010}; restricted to the space
\eqref{Div_range}, where $\bh=0$, this is the map $\mathsf L$.
\end{proof}

The splitting is a bounded direct sum: $P_R:=\Rop_n\divv$ and $P_H:=\Pi_{\cH_n}$ are bounded projections ($P_H\bh_i=\bh_i$ since $F_j(\bh_i)=-|S^{n-1}|\delta_{ij}$), with $P_HP_R=0$ because $\Rop_nq$ has no cavity fluxes (Corollary~\ref{cor:R_fluxes}) and $P_RP_H=0$ because every $\bh\in\cH_n$ is divergence free. Hence $P_0:=I-P_R-P_H$ satisfies $P_0^2=P_0$, and the three projections have the ranges $\Rop_nL^2(\Omega)$, $\cH_n$ and $\operatorname{Div}\,H^1(\Omega;\mathbb K_n)$.

\begin{cor}[representation norm]\label{cor:representation_norm}
For $\btau=\operatorname{Div}A+\Rop_nq+\bh$ as in Proposition~\ref{prop:hdiv_split}, with any potential $A$,
\[
 c\,\|\btau\|_{H(\divvr)}\le\operatorname{dist}_{H^1}(A,\ker\operatorname{Div})+\|\bh\|_0+\|q\|_0\le C\|\btau\|_{H(\divvr)} ,
\]
and the same holds for $\btau-\btau_m$ with $A-A_m$, $\bh-\bh_m$, $q-q_m$ in place of $A$, $\bh$, $q$.
\end{cor}

\begin{proof}
For $K\in\ker\operatorname{Div}$, $\|\btau\|_{H(\divvr)}\le\|\operatorname{Div}(A-K)\|_0+\|\bh\|_0+\|\Rop_nq\|_{H(\divvr)}\le\sqrt{n-1}\,|A-K|_1+\|\bh\|_0+C\|q\|_0$ by \eqref{divDiv_zero} and \eqref{Rd_bound}. Conversely, $\|q\|_0\le\|\btau\|_{H(\divvr)}$, $\|\bh\|_0\le C\|\btau\|_{H(\divvr)}$ by the boundedness of the $F_i$, and $A-\mathsf L\,\operatorname{Div}A\in\ker\operatorname{Div}$, so $\operatorname{dist}_{H^1}(A,\ker\operatorname{Div})\le\|\mathsf L(\btau-\Rop_nq-\bh)\|_1\le C\|\btau\|_{H(\divvr)}$. The statement for $\btau-\btau_m$ follows by linearity.
\end{proof}

Thus, after quotienting out the gauge of the potential, the potential, the cohomology part and
the divergence give a norm-equivalent representation of $H(\divvr;\Omega)$: with $H^1(\Omega;\mathbb K_n)/\ker\operatorname{Div}$ equipped with its quotient norm $\operatorname{dist}_{H^1}(A,\ker\operatorname{Div})$,
\[
 \bigl(H^1(\Omega;\mathbb K_n)/\ker\operatorname{Div}\bigr)\times\cH_n\times L^2(\Omega)\to H(\divvr;\Omega),\qquad([A],\bh,q)\mapsto\operatorname{Div}A+\bh+\Rop_nq,
\]
is an isomorphism of Banach spaces: it is onto by Proposition~\ref{prop:hdiv_split}; it is injective, since $\operatorname{Div}A+\bh+\Rop_nq=0$ gives $q=0$ by taking the divergence, then $\bh=0$ by
applying $\Pi_{\cH_n}$, and then $\operatorname{Div}A=0$; and both it and its inverse are bounded by Corollary~\ref{cor:representation_norm}.

\begin{cor}[regular potentials]\label{cor:regular_potential}
Let $r\ge0$, $\bu\in H^r(\Omega)^n$ with $\divv\bu=0$ and $F_i(\bu)=0$, $i=1,\ldots,b$. Then $\bu=\operatorname{Div}A$ with $A\in H^{r+1}(\Omega;\mathbb K_n)$ and $\|A\|_{r+1}\le C_r\|\bu\|_r$.
\end{cor}

\begin{proof}
By \eqref{Div_range} and Proposition~\ref{prop:hdiv_split}, $\bu=\operatorname{Div}A_0$ with $A_0\in H^1(\Omega;\mathbb
K_n)$ and $\|A_0\|_1\le C\|\bu\|_{H(\divvr)}=C\|\bu\|_0\le C\|\bu\|_r$. By Remark~\ref{rem:forms}, $\bu\lrcorner\,\mathrm{vol}=d\mu_0$ with $\mu_0:=\star A_0\in
H^1(\Omega,\Lambda^{n-2})$. By \cite[Corollary~4.7(a)]{CostabelMcIntosh2010} with $s=r$ and $t=1$ there is
$\mu\in H^{r+1}(\Omega,\Lambda^{n-2})$ with $d\mu=\bu\lrcorner\,\mathrm{vol}$ and $\|\mu\|_{r+1}\le C_r(\|\bu\|_r+\|\mu_0\|_1)\le C_r\|\bu\|_r$.
The skew-symmetric matrix $A$ of the $2$-form with $\star
A=\mu$ satisfies $\operatorname{Div}A=\bu$ by Remark~\ref{rem:forms}, and $\|A\|_{r+1}\le C_r\|\bu\|_r$. The
splitting removes the topological obstruction, and \cite[Corollary~4.7(a)]{CostabelMcIntosh2010} raises the regularity.
\end{proof}

On domains starlike with respect to a ball this is the lifting of \cite[Proposition~4.1(i)]{CostabelMcIntosh2010} used in \cite{HeLiuTian2026}; on a general bounded Lipschitz domain it holds after the cohomology component has been removed.

For finite elements the representation by potentials holds at the discrete level: the divergence-free fields of the $H(\divvr)$ elements
are, up to the discrete cohomology, $\operatorname{Div}$ of skew-symmetric polynomial potentials on each simplex
\cite{AFW:2006,AFW:10}. The neural use of these potentials in the literature is discussed in Section~\ref{subsec:related}.

\subsection{Neural classes}\label{subsec:hdiv_classes}

Let $\{\cA_m\}_{m\ge1}$ and $\{\cQ_m\}_{m\ge1}$ be nested families of sets with
$\cA_m\subset H^1(\Omega;\mathbb K_n)$ and $\cQ_m\subset L^2(\Omega)$; the index $m$ may
stand for a tuple of complexity parameters of the two classes. In
the neural realization $A_m=\sum_{i<j}a_{ij,m}\mathcal E_{ij}$, where the $n(n-1)/2$ entries $a_{ij,m}$ are scalar networks
(independent networks, or one network with $n(n-1)/2$ outputs and shared hidden
layers), so that skew symmetry holds by construction, and $q_m$ is a scalar network. Let
\beq\label{Sigma_hdiv_d}
 \Sigma_{m,n}^{\rm dR}:=\bigl\{\operatorname{Div}A_m+\Rop_nq_m+\bh:\ A_m\in\cA_m,\ q_m\in\cQ_m,\ \bh\in\cH_n\bigr\}.
\eeq
The trainable functions are $A_m$ and $q_m$, which are
assembled into the field by the fixed operators $\operatorname{Div}$
and $\Rop_n$, together with the $b$ coefficients of $\bh$; the $H(\divvr)$
field itself is not taken as an unconstrained componentwise
network (the entries of $A_m$ are), and $\Rop_nq_m$ is
in general not itself a network. The representing
functions of a target are those of Proposition~\ref{prop:hdiv_split}.

\begin{thm}[neural $H(\divvr)$ classes in $\R^n$]\label{thm:hdivd}
Every $\btau_m=\operatorname{Div}A_m+\Rop_nq_m+\bh\in\Sigma_{m,n}^{\rm dR}$ lies in $H(\divvr;\Omega)$, and $\divv\btau_m=q_m$. Let $\btau\in H(\divvr;\Omega)$, $q=\divv\btau$, $\bh=\Pi_{\cH_n}\btau$, and let $A\in H^1(\Omega;\mathbb K_n)$ be any potential with $\operatorname{Div}A=\btau-\Rop_nq-\bh$. Then
\beq\label{hdiv_d_approx}
 \inf_{q_m\in\cQ_m}\|q-q_m\|_0\le\inf_{\btau_m\in\Sigma_{m,n}^{\rm dR}}\|\btau-\btau_m\|_{H(\divvr)}
 \le\sqrt{n-1}\inf_{A_m\in\cA_m}|A-A_m|_1+\Bigl(1+\frac{4C_{x_0}^2}{n^2}\Bigr)^{1/2}\inf_{q_m\in\cQ_m}\|q-q_m\|_0 .
\eeq
\end{thm}

\begin{proof}
Conformity and the divergence identity follow from \eqref{divDiv_zero}, Lemma~\ref{lem:Rd} and $\divv\bh=0$. The lower bound holds because $\|\btau-\btau_m\|_{H(\divvr)}\ge\|\divv(\btau-\btau_m)\|_0=\|q-q_m\|_0$ for every $\btau_m\in\Sigma_{m,n}^{\rm dR}$. For the upper bound let $A_m\in\cA_m$ and $q_m\in\cQ_m$ be arbitrary and $\btau_m:=\operatorname{Div}A_m+\Rop_nq_m+\bh$, with the same $\bh$ as $\btau$. Then $\btau-\btau_m=\operatorname{Div}(A-A_m)+\Rop_n(q-q_m)$. The first term is divergence free, so its $H(\divvr)$ norm is $\|\operatorname{Div}(A-A_m)\|_0\le\sqrt{n-1}|A-A_m|_1$ by \eqref{divDiv_zero}; the second is bounded by $(1+4C_{x_0}^2/n^2)^{1/2}\|q-q_m\|_0$ by \eqref{Rd_bound}. The triangle inequality and independent minimization over $A_m$ and $q_m$ give the upper bound.
\end{proof}

Since the bound holds for every potential, $\inf_{A_m}|A-A_m|_1$ may be replaced by its infimum over all potentials $A$ with $\operatorname{Div}A=\btau-\Rop_nq-\bh$.

\begin{cor}[density]\label{cor:hdiv_density}
If $\bigcup_m\cA_m$ is dense in $H^1(\Omega;\mathbb K_n)$ and $\bigcup_m\cQ_m$ is dense in $L^2(\Omega)$, then $\bigcup_m\Sigma_{m,n}^{\rm dR}$ is dense in $H(\divvr;\Omega)$, and $\inf_{\btau_m\in\Sigma_{m,n}^{\rm dR}}\|\btau-\btau_m\|_{H(\divvr)}\to0$ as $m\to\infty$ for every $\btau\in H(\divvr;\Omega)$.
\end{cor}

\begin{proof}
By Proposition~\ref{prop:hdiv_split} every $\btau$ has a representation with $A$, $\bh$, $q$; the
two infima in \eqref{hdiv_d_approx} are nonincreasing in $m$ because the families are nested,
and they tend to zero by the density assumptions.
\end{proof}

The radial part $\Rop_nq_m$ requires only one-dimensional integrals of
$q_m$ along segments and no derivative of $q_m$;
its evaluation by quadrature is described in Section~\ref{sec:framework}.

\begin{rem}[relation to Raviart--Thomas elements]\label{rem:RT0}
The splitting mirrors the local polynomial decomposition of the Raviart--Thomas spaces; on a simplex the topological component $\cH_n$ is absent. Let $T\subset\R^n$ be a simplex and $x_T\in T$ a center chosen locally in $T$, independently of the global center $x_0$
of $\Rop_n$; let $\Rop_{n,T}$ denote \eqref{Rd} on $T$ with center $x_T$. Let $\mathbb P_k$ be the polynomials of total degree at most $k$, $\widetilde{\mathbb P}_k$ those homogeneous of degree $k$ in $x-x_T$, and $\mathbb K_n(\mathbb P_{k+1})$ the skew-symmetric matrices with entries in $\mathbb P_{k+1}$. Then $RT_k(T)=\mathbb P_k^n\oplus(x-x_T)\widetilde{\mathbb P}_k$, and
\[
 RT_k(T)=\operatorname{Div}\,\mathbb K_n(\mathbb P_{k+1})\oplus(x-x_T)\mathbb P_k\qquad\text{for every }k\ge0
\]
(the Koszul decomposition, \cite{AFW:2006}): an element $\mathbf p+(x-x_T)\widetilde p$, $\widetilde p\in\widetilde{\mathbb P}_k$, has divergence $\divv\mathbf p+(n+k)\widetilde p$, so
its divergence-free elements lie in $\mathbb P_k^n$, and the divergence-free fields in $\mathbb P_k^n$ are $\operatorname{Div}\,\mathbb K_n(\mathbb P_{k+1})$
by the exactness of the polynomial de Rham complex \cite{AFW:2006}; the second part carries the divergence. For polynomials $\Rop_{n,T}$ produces exactly the second part: $T$
is convex and contains $x_T$, so the segments from $x_T$ to the points of $T$ stay
in $T$, and for $q_j$ homogeneous of degree $j$ in $x-x_T$, $\Rop_{n,T}q_j(x)=(x-x_T)\,q_j(x)\int_0^1t^{n-1+j}\,dt=(x-x_T)\,q_j(x)/(n+j)$. Hence $\operatorname{Div}A+\Rop_{n,T}q$ with $A\in\mathbb K_n(\mathbb P_{k+1})$ and $q\in\mathbb P_k$ is exactly $RT_k(T)$, and the neural
classes replace the polynomials $A$ and $q$ by networks. On a triangulation the corresponding global statement is made in two dimensions only.
\end{rem}
\subsection{Two and three dimensions}\label{subsec:hdiv23}

For $n=2$ every $A\in\mathbb K_2$ is $A=\psi\mathcal E_{12}$, $\mathcal E_{12}=\begin{pmatrix}0&1\\-1&0\end{pmatrix}$, with $\psi=a_{12}$, and
\[
 \operatorname{Div}(\psi\mathcal E_{12})=(\p_2\psi,-\p_1\psi)^T=\curl\psi,\qquad |\psi\mathcal E_{12}|_1=|\psi|_1 .
\]
If $\bn=(n_1,n_2)^T$ and $\bt=(-n_2,n_1)^T$, then $\mathcal E_{12}\bn=-\bt$ and
\beq\label{curl_nt}
 \curl\psi\cdot\bn=\p_{\bt}\psi,\qquad \curl\psi\cdot\bt=-\p_{\bn}\psi .
\eeq
For $n=3$ let $A_{ij}=\sum_k\varepsilon_{ijk}\Psi_k$, that is $a_{12}=\Psi_3$, $a_{13}=-\Psi_2$, $a_{23}=\Psi_1$. Then
\[
 (\operatorname{Div}A)_i=\sum_{j,k}\varepsilon_{ijk}\p_j\Psi_k=(\curll\bPsi)_i,\qquad |A|_1=|\bPsi|_1,\qquad
 \jump{\p_{\bn}A}\,\bn=\bn\times\jump{\p_{\bn}\bPsi},
\]
and for $N_{ij}=\sum_k\varepsilon_{ijk}N_k$, $N\bv=\bv\times\mathbf N$. The fields \eqref{harmonic_fields} are
$\bh_i(x)=(x-z_i)/|x-z_i|^n$ with $z_i$ in the holes ($n=2$) or cavities ($n=3$)
of $\Omega$, with fluxes $-2\pi\delta_{ij}$ and $-4\pi\delta_{ij}$; for $n=2$
we write $\cH:=\cH_2$. Let $\{\cP_m\}_{m\ge1}$, $\{\cN_m\}_{m\ge1}$ and $\{\cQ_m\}_{m\ge1}$ be
nested families with $\cP_m\subset H^1(\Omega)$, $\cN_m\subset H^1(\Omega)^3$ and $\cQ_m\subset
L^2(\Omega)$, for example network classes; the classes \eqref{Sigma_hdiv_d} become
\beq\label{Sigma_hdiv_23}
\begin{aligned}
 \Sigma_{m,2}^{\rm dR}&=\bigl\{\curl\psi_m+\Rop_2q_m+\bh:\ \psi_m\in\cP_m,\ q_m\in\cQ_m,\ \bh\in\cH\bigr\},\\
 \Sigma_{m,3}^{\rm dR}&=\bigl\{\curll\bPsi_m+\Rop_3q_m+\bh:\ \bPsi_m\in\cN_m,\ q_m\in\cQ_m,\ \bh\in\cH_3\bigr\}.
\end{aligned}
\eeq
For $n=2$ this is the classical stream-function form of the two-dimensional
de Rham decomposition, with explicit representatives of the finite-dimensional cohomology.

\begin{cor}[neural $H(\divvr)$ classes for $n=2,3$]\label{cor:hdiv23}
Let $n\in\{2,3\}$. Every $\btau\in H(\divvr;\Omega)$ can be written
\beq\label{hdiv_23_representation}
 \btau=\curl\psi+\Rop_2q+\bh\quad(n=2),\qquad \btau=\curll\bPsi+\Rop_3q+\bh\quad(n=3),\qquad q=\divv\btau,\quad \bh=\Pi_{\cH_n}\btau,
\eeq
where $q$ and $\bh$ are unique ($\bh$ is determined by the hole or cavity fluxes of $\btau$, Corollary~\ref{cor:R_fluxes}), $\psi\in H^1(\Omega)$ is unique
up to a constant and $\bPsi\in H^1(\Omega)^3$ up to an element of the kernel of $\curll$, and they can be chosen with $\|\psi\|_1$, respectively $\|\bPsi\|_1$,
plus $\|\bh\|_0+\|q\|_0$ bounded by $C\|\btau\|_{H(\divvr)}$. Every element of $\Sigma_{m,n}^{\rm dR}$ lies in $H(\divvr;\Omega)$, with divergence $q_m$, and for every such $\psi$ or $\bPsi$
\beq\label{derham_density_bound}
\begin{aligned}
 &\inf_{q_m\in\cQ_m}\|q-q_m\|_0\le\inf_{\btau_m\in\Sigma_{m,n}^{\rm dR}}\|\btau-\btau_m\|_{H(\divvr)}\\
 &\qquad\le\begin{cases}
 \inf_{\psi_m\in\cP_m}|\psi-\psi_m|_1+\sqrt{1+C_{x_0}^2}\,\inf_{q_m\in\cQ_m}\|q-q_m\|_0, & n=2,\\[2pt]
 \sqrt2\inf_{\bPsi_m\in\cN_m}|\bPsi-\bPsi_m|_1+\bigl(1+\tfrac49C_{x_0}^2\bigr)^{1/2}\inf_{q_m\in\cQ_m}\|q-q_m\|_0, & n=3;
 \end{cases}
\end{aligned}
\eeq
for $n=2$ the first term on the right does not depend on the choice of $\psi$.
\end{cor}

\begin{proof}
Proposition~\ref{prop:hdiv_split} and Theorem~\ref{thm:hdivd} with $A=\psi\mathcal E_{12}$ ($n=2$) and $A_{ij}=\sum_k\varepsilon_{ijk}\Psi_k$ ($n=3$):
$\operatorname{Div}A$ is $\curl\psi$, respectively $\curll\bPsi$, so the kernel of $\operatorname{Div}$ corresponds to that of the curl (the constants for
$n=2$, $\Omega$ being connected), $|A|_1=|\psi|_1$, respectively $|\bPsi|_1$, and $\sqrt{n-1}$ and $(1+4C_{x_0}^2/n^2)^{1/2}$ give the constants; the
seminorm $|\cdot|_1$ is invariant under constants.
\end{proof}

In three dimensions an implementation of this form has three scalar potential outputs for $\bPsi$
(three scalar networks, or one network with three outputs and shared hidden layers), one for $q$, and the $b$ cavity
coefficients.

\begin{rem}[Raviart--Thomas elements on a triangulation]\label{rem:RT0_2d}
On a conforming triangulation $\mathcal T_h$ of a simply connected $\Omega$, $RT_0(\mathcal T_h)\cap\ker\divv=\curl S_1(\mathcal T_h)$,
$S_1(\mathcal T_h)$ the continuous piecewise linear functions, and every function in $S_1(\mathcal T_h)$ is a ReLU network
\cite{HLXZ:20}; hence $\curl S_1(\mathcal T_h)\subset\Sigma_{m,2}^{\rm dR}$ for a network size depending only on $\mathcal T_h$, whereas a
componentwise network with continuous activations is a continuous vector field, so for no network size does it contain a field of
$\curl S_1(\mathcal T_h)$ with a tangential jump; it can only approximate such a field.  For a general $\btau_h\in RT_0(\mathcal T_h)$, $\Rop_2q_h$ with $q_h=\divv\btau_h$ piecewise constant need not be
piecewise polynomial, so exact finite realization is not asserted; Corollary~\ref{cor:hdiv23} gives
$\btau_h=\curl\psi_h+\Rop_2q_h$ with $\psi_h\in H^1(\Omega)$, and by \eqref{derham_density_bound} fitting $\btau_h$ is a scalar
approximation problem for $\psi_h$ and $q_h$.  In three dimensions, on a topologically trivial domain, the divergence-free fields of
$RT_0(\mathcal T_h)$ are the curls of lowest-order N\'ed\'elec fields, which in general are not continuous; on a general domain the
corresponding discrete harmonic fields must be added.
\end{rem}

\section{Native neural approximation in \texorpdfstring{$H(\divvr)$}{H(div)} II: regular decomposition}\label{sec:hdiv_reg}

\subsection{Regular decomposition}\label{subsec:hdiv_regular}

The right inverse of Section~\ref{sec:hdiv} makes the divergence an explicit unknown.  A second stable representation leaves the
divergence to an $H^1$ field, $\btau=\operatorname{Div}A+\bzz$ with $\bzz\in H^1(\Omega)^n$; it comes from a regular decomposition.  Let $H_0(\divvr;\Omega)$ denote the subspace of $H(\divvr;\Omega)$ with $\btau\cdot\bn=0$ on $\partial\Omega$.

\begin{prop}[regular decomposition of $H(\divvr)$]\label{prop:hdiv_regular}
Every $\btau\in H(\divvr;\Omega)$ can be written $\btau=\operatorname{Div}A+\bzz$ with $A\in H^1(\Omega;\mathbb K_n)$ and $\bzz\in H^1(\Omega)^n$ depending linearly
on $\btau$ and $\|A\|_1+\|\bzz\|_1\le C\|\btau\|_{H(\divvr)}$; for $\btau\in H_0(\divvr;\Omega)$ the same holds with $A\in H^1_0(\Omega;\mathbb K_n)$ and
$\bzz\in H^1_0(\Omega)^n$.  Then $\divv\btau=\divv\bzz$, and no cohomology part appears: the harmonic fields are absorbed into $\bzz$.
\end{prop}

\begin{proof}
By Remark~\ref{rem:forms}, $\btau$ corresponds to the $(n-1)$-form $\omega=\btau\lrcorner\,\mathrm{vol}$, with $\dd\omega=(\divv\btau)\,\mathrm{vol}$, and an
$(n-2)$-form $\mu$ to the skew-symmetric matrix $A$ with $\star A=\mu$, with $\dd\mu=(\operatorname{Div}A)\lrcorner\,\mathrm{vol}$.  By
\cite[Theorem~4.6]{CostabelMcIntosh2010} there are operators $R_\ell$ (regularized Poincar\'e operators, not to be confused with $\Rop_n$),
bounded from $H^s\Lambda^\ell(\Omega)$ to $H^{s+1}\Lambda^{\ell-1}(\Omega)$ for every $s$, and $K_\ell$, bounded from $H^s\Lambda^\ell(\Omega)$ to
$C^\infty\Lambda^\ell(\overline\Omega)$, with $\dd R_\ell u+R_{\ell+1}\dd u=u-K_\ell u$.  For $u=\omega$ take $\mu=R_{n-1}\omega$, which is in $H^1$, and
$\bzz\lrcorner\,\mathrm{vol}=R_n\dd\omega+K_{n-1}\omega$, which is in $H^1$ since $\dd\omega\in L^2$.  For $H_0(\divvr;\Omega)$ the Bogovski\u{\i}-type
operators $T_\ell$, $L_\ell$ of the same theorem, applied to the extension of $\omega$ by zero, give the same identity with supports in
$\overline\Omega$.  Taking the divergence of $\btau=\operatorname{Div}A+\bzz$ gives $\divv\btau=\divv\bzz$ by \eqref{divDiv_zero}.
\end{proof}

Similar ideas can be generalized to the de Rham spaces $H\Lambda^k(\Omega)$, $1\le k\le n$, through the same theorem of
\cite{CostabelMcIntosh2010}; we do not pursue this here.

\subsection{The second neural class}\label{subsec:hdiv_reg_classes}

The second class takes the $H^1$ field of Proposition~\ref{prop:hdiv_regular} as a componentwise network.  Let $\{\cZ_m\}_{m\ge1}$ be nested
families of sets $\cZ_m\subset H^1(\Omega)^n$, in the neural realization vectors of $n$ scalar networks (or one network with $n$ outputs),
and
\beq\label{Sigma_hdiv_reg}
 \Sigma_{m,n}^{\rm reg}:=\bigl\{\operatorname{Div}A_m+\bzz_m:\ A_m\in\cA_m,\ \bzz_m\in\cZ_m\bigr\}.
\eeq
For $H_0(\divvr;\Omega)$ both the potential and the $H^1$ field are taken in the corresponding $H^1_0$ classes, $\cA_m\subset H^1_0(\Omega;\mathbb K_n)$
and $\cZ_m\subset H^1_0(\Omega)^n$, for instance through a boundary factor, in agreement with Proposition~\ref{prop:hdiv_regular}; then
$\btau_m\cdot\bn=0$ on $\partial\Omega$.

\begin{thm}[second neural $H(\divvr)$ class]\label{thm:hdivd_reg}
Every $\btau_m=\operatorname{Div}A_m+\bzz_m\in\Sigma_{m,n}^{\rm reg}$ lies in $H(\divvr;\Omega)$, and $\divv\btau_m=\divv\bzz_m$.  For $\btau\in H(\divvr;\Omega)$ and any
$A$, $\bzz$ as in Proposition~\ref{prop:hdiv_regular},
\beq\label{hdiv_reg_approx}
 \inf_{\btau_m\in\Sigma_{m,n}^{\rm reg}}\|\btau-\btau_m\|_{H(\divvr)}\le\sqrt{n-1}\inf_{A_m\in\cA_m}|A-A_m|_1+\sqrt n\inf_{\bzz_m\in\cZ_m}\|\bzz-\bzz_m\|_1 .
\eeq
\end{thm}

\begin{proof}
$\divv\operatorname{Div}A_m=0$ by \eqref{divDiv_zero}, and $\bzz_m\in H^1(\Omega)^n\subset H(\divvr;\Omega)$.  For $\btau_m=\operatorname{Div}A_m+\bzz_m$,
$\btau-\btau_m=\operatorname{Div}(A-A_m)+(\bzz-\bzz_m)$; the first term is divergence free, with $\|\operatorname{Div}(A-A_m)\|_0\le\sqrt{n-1}|A-A_m|_1$
by \eqref{divDiv_zero}, and for $\bw\in H^1(\Omega)^n$, $\|\divv\bw\|_0\le\sqrt n|\bw|_1$, so $\|\bw\|_{H(\divvr)}\le\sqrt n\|\bw\|_1$.
\end{proof}

In two and three dimensions the second class reads $\curl\psi_m+\bzz_m$ for $n=2$ and $\curll\bPsi_m+\bzz_m$ for $n=3$, with the estimate
\eqref{hdiv_reg_approx} and $|\psi-\psi_m|_1$, respectively $\sqrt2|\bPsi-\bPsi_m|_1$, in place of $\sqrt{n-1}|A-A_m|_1$.

\subsection{The two classes}\label{subsec:hdiv_two}

\label{rem:two_classes}The two classes differ in where the divergence sits.  In $\Sigma_{m,n}^{\rm dR}$ the divergence $q_m$ is an independent $L^2$ unknown: an
equation for the divergence can be built into the class, and a divergence that jumps across an interface is integrated by
$\Rop_n$ into a continuous field under the segment condition of Lemma~\ref{lem:Rcont}.  In $\Sigma_{m,n}^{\rm reg}$ the fields are evaluated pointwise, without the line
integrals along segments that $\Rop_nq_m$ requires, and no harmonic fields are needed; the divergence $\divv\bzz_m$ is not an independent
unknown, and a divergence that jumps across an interface requires a kink of $\bzz_m$ there.  In both classes every interface jump of $\btau_m$ is carried by the potential: $\Rop_nq_m$ has no jumps for a bounded, piecewise continuous $q_m$
under the segment condition of Lemma~\ref{lem:Rcont}, and $\bzz_m\in H^1$ has no interface jump, its trace being single valued.  The second class is therefore the
componentwise approach, with $\bzz_m$ a componentwise network of the kind used for fluxes in \cite{CCLL:20,LC:26p2nn}, enriched by $\operatorname{Div}A_m$.  The enrichment is divergence free, so it leaves
$\divv\btau_m$ and conformity unchanged, and it adds exactly the discontinuities that $H(\divvr)$ admits, the tangential jumps, which a
componentwise $H^1$ field cannot carry.  Both classes are used in the examples.

\section{Native neural approximation in \texorpdfstring{$H(\divvr)$}{H(div)} III: finite-width jumps and interface-adapted potentials}\label{sec:hdiv2}

Sections~\ref{sec:hdiv} and~\ref{sec:hdiv_reg} give every $\btau\in H(\divvr;\Omega)$ the representations $\btau=\operatorname{Div}A+\Rop_nq+\bh$ and $\btau=\operatorname{Div}A+\bzz$, and bound the best approximation error of the neural classes in $H(\divvr)$ by the approximation errors of $A$ and $q$, or of $A$ and $\bzz$.  With smooth
activations the potential component $\operatorname{Div}A_m$ is smooth, so a field with a tangential jump can only be approximated; to
represent such a jump at finite width the potential must have kinks, as a ReLU potential does.  Throughout
this section we write
\beq\label{hdiv_common_form}
 \btau=\operatorname{Div}A+\bw ,
\eeq
with the remaining part $\bw=\Rop_nq+\bh$ in the first class and $\bw=\bzz$ in the second.  The section concerns the potential part
$\operatorname{Div}A$, which the two classes share, and its results apply to both in the same form.  It shows that kinks of the potential carry the admissible tangential jumps of
the flux, builds a known or a learned interface into the potential, shows that, under an additional hypothesis on the jump
amplitude, extracting the interface kink restores the regularity of the remaining potential, and defines the three levels of
interface information used in the numerical examples.  Only a jump of the divergence concerns the remaining part, and it is
carried differently in the two classes (Section~\ref{subsec:div_jumps}).

\subsection{Tangential jumps carried by kinks of the potential}\label{subsec:hdiv_jumps}

Across a smooth interface where a piecewise regular field in $H(\divvr;\Omega)$ has one-sided traces, its normal component is
continuous and only its tangential component may jump.  In both classes the remaining part $\bw$ has no jump across an interior interface in the situations of interest, so the
tangential jump is carried by kinks of the potential $A$.

\begin{prop}[jump carried by the potential]\label{prop:jump}
Let $\btau=\operatorname{Div}A+\bw$ as in \eqref{hdiv_common_form} with $A\in H^1(\Omega;\mathbb K_n)$, let $\Gamma\subset\Omega$ be a smooth hypersurface with unit normal $\bn$,
let $A$ be piecewise $H^2$ on the two sides of $\Gamma$ in a neighborhood of $\Gamma$, and let $\bw$ have no jump across $\Gamma$, its one-sided traces coinciding.  (For $\bw=\Rop_nq+\bh$ this holds by Lemma~\ref{lem:Rcont} if $q$ is
bounded and piecewise continuous and the condition on the segments of that lemma is satisfied, $\bh$ being smooth; for $\bw=\bzz\in H^1(\Omega)^n$
the trace on $\Gamma$ is single valued.)  Then, with one-sided traces,
\[
 \jump{\btau\cdot\bn}=0,\qquad \jump{\btau}=\jump{\p_{\bn}A}\,\bn .
\]
Thus the normal continuity required by $H(\divvr)$ is automatic, and the tangential jump of $\btau$ is carried entirely
by the kink of $A$. Conversely, pointwise on $\Gamma$, every admissible jump $\mathbf j$ ($\mathbf j\cdot\bn=0$) is produced by the
kink $\jump{\p_{\bn}A}=\mathbf j\bn^T-\bn\mathbf j^T$; two kinks that produce the same jump differ by a skew-symmetric $K$ with $K\bn=0$, which vanishes for $n=2$ and ranges over a space
of dimension $(n-1)(n-2)/2$ for $n\ge3$. The statement holds for every potential of $\btau$ that is
piecewise $H^2$ near $\Gamma$; since $\jump{\btau}$ does not depend on the potential, the kinks of two such potentials differ by such a $K$.
\end{prop}

\begin{proof}
Since $\bw$ has no jump across $\Gamma$, $\jump{\btau}=\jump{\operatorname{Div}A}$, and Lemma~\ref{lem:skew}(b)
gives $\jump{\operatorname{Div}A}=\jump{\p_{\bn}A}\bn$ and $\bn\cdot\jump{\operatorname{Div}A}=0$. The converse and the nonuniqueness are Lemma~\ref{lem:skew}(c).
\end{proof}

The construction of a global potential with a
prescribed jump field along a whole interface is
the role of the interface modules below. At
finite width, ReLU networks give the following instance.

\begin{cor}[continuous piecewise linear potentials]\label{lem:cpwl}
Let the entries $a_{ij}$, $i<j$, of $A$ be continuous and piecewise affine on a finite partition of $\Omega$ into polyhedral cells; for finite ReLU networks this is the common refinement of the activation
partitions of the entries. Then $\operatorname{Div}A$ is piecewise constant, belongs to $H(\divvr;\Omega)$ with $\divv\operatorname{Div}A=0$, and across the relative interior of every interior codimension-one face, with unit normal $\bn$,
\[
 \bn\cdot\jump{\operatorname{Div}A}=0,\qquad \jump{\operatorname{Div}A}=\jump{\p_{\bn}A}\,\bn .
\]
\end{cor}

\begin{proof}
$A\in W^{1,\infty}(\Omega;\mathbb K_n)$ (continuous and piecewise affine on a finite polyhedral partition), hence
$A\in H^1(\Omega;\mathbb K_n)$, and Lemma~\ref{lem:skew}(a) applies; on the relative interior
of a face, $A$ is affine on each side, and Lemma~\ref{lem:skew}(b) applies.
\end{proof}

A generic ReLU potential therefore has the correct
jump geometry automatically: its tangential jumps lie on
the flat faces of its activation partition, which
is determined implicitly by the parameters, and a curved
interface is matched by such faces only approximately.
Sections~\ref{subsec:interface} and~\ref{subsec:learned_interface} make this geometry explicit when an
interface is known or is itself to be learned.

\subsection{Fixed-interface potential modules}\label{subsec:interface}

The classes \eqref{Sigma_hdiv_d} and \eqref{Sigma_hdiv_reg} need no interface enrichment
for conformity or for the estimates above. Without an
explicit interface module, however, an ordinary finite ReLU
potential can place exact kinks only on faces of its
activation partition (Corollary~\ref{lem:cpwl}), and a finite smooth network
has no exact kink. Moreover, if the divergence-free part
has a nonzero tangential jump on a subset
of positive $(n-1)$-dimensional measure of a smooth interface, then
no potential of it belongs to $H^{3/2}(\Omega)$: $A\in
H^{3/2}$ would give $\operatorname{Div}A\in H^{1/2}(\Omega)^n$, whereas a piecewise regular
function with a nonzero jump across a Lipschitz hypersurface on a
set of positive surface measure is not in $H^{1/2}$
(after a bi-Lipschitz flattening, which preserves $H^{1/2}$, and integration over the
tangential variables, the Sobolev--Slobodeckij seminorm contains $\iint_{s<0<t}(t-s)^{-2}\,ds\,dt=\infty$ times the
squared jump). The modules below place a kink
of the $H^1$ potential, and with it the tangential
jump of the flux, on a given hypersurface
(this subsection) or a learned one (Section~\ref{subsec:learned_interface}); they do not
affect $H(\divvr)$ conformity. Both are instances of one mechanism.

\begin{lem}[multiplicative kink]\label{lem:mult_kink}
Let $\rho\in W^{1,\infty}(\Omega)$, $N\in H^1(\Omega;\mathbb K_n)$ and $B:=\rho N$. Then $B\in H^1(\Omega;\mathbb K_n)$, and
$\operatorname{Div}(A_0+B)+\bw\in H(\divvr;\Omega)$ for every $A_0\in H^1(\Omega;\mathbb K_n)$ and every remaining part $\bw$ of either
class. Let further $x^*\in\Omega$, let $\Gamma$ be a hypersurface of class $C^1$
near $x^*$ with $x^*\in\Gamma$ and unit normal $\bn$,
let $\rho$ vanish on $\Gamma$ and be continuously differentiable on each side
of $\Gamma$ near $x^*$, up to $\Gamma$, with one-sided gradients $\grad\rho^\pm$, and let $N$ be
Lipschitz continuous near $x^*$. Then, in the sense of one-sided essential limits at $x^*$,
\[
 \jump{\grad B}=N(x^*)\otimes\jump{\grad\rho},\qquad
 \jump{\operatorname{Div}B}=N(x^*)\jump{\grad\rho}=\jump{\p_{\bn}\rho}\,N(x^*)\bn,\qquad \bn\cdot\jump{\operatorname{Div}B}=0,
\]
with $\jump{\grad\rho}=\jump{\p_{\bn}\rho}\bn$ evaluated at $x^*$, and the same identities
hold for $A_0+B$ if $A_0$ is continuously differentiable near $x^*$.
\end{lem}

\begin{proof}
Since $\rho\in W^{1,\infty}(\Omega)$ and $N\in H^1$, $B\in H^1$ with $\grad(\rho N_{ij})=N_{ij}\grad\rho+\rho\grad N_{ij}$ almost everywhere, and conformity follows from \eqref{divDiv_zero} and Lemma~\ref{lem:Rd}. Near $x^*$,
$N$ is continuous and $|\grad N|\le L$ almost everywhere ($L$ a Lipschitz constant), and $\rho$ is continuous with $\rho(x^*)=0$, so
$\operatorname{ess\,sup}|\rho\grad N|\le L\sup|\rho|\to0$ on shrinking neighborhoods of $x^*$; the first term has the one-sided limits $N_{ij}(x^*)\grad\rho^\pm(x^*)$. Since $\rho=0$ on $\Gamma$, its tangential derivatives vanish from both
sides, so $\jump{\grad\rho}=\jump{\p_{\bn}\rho}\bn$, and $\bn\cdot N\bn=0$ because $N$ is skew. Where $A_0$ is continuously differentiable, its gradient adds no jump.
\end{proof}

Only the Lipschitz continuity of $N$ near $x^*$ enters,
so activation interfaces of a ReLU amplitude network
$N$ may cross $\Gamma$ without changing the jump, whereas
a kink of $A_0$ on $\Gamma$ adds its
own jump. Only $N(x^*)\bn$ enters the jump. For
$n\ge3$ the map $N\mapsto N\bn$ has a nontrivial
kernel ($K\in\mathbb K_n$ with $K\bn=0$, Proposition~\ref{prop:jump}), so a
prescribed jump does not determine $N$ on $\Gamma$.
This is a property of the jump map
only: it does not mean that $N$ and
$N+K$ give the same flux in the bulk (in
general they do not), and the remaining degrees
of freedom are available to fit the bulk
residual. Independently, the potential has the usual gauge
nonuniqueness: the functional depends on it only through $\operatorname{Div}A$.

First let $\Gamma=\{x\in\Omega:d_\Gamma(x)=0\}$ be a known interface, where $d_\Gamma\in W^{1,\infty}(\Omega)$ is continuously differentiable near $\Gamma$ with $\grad d_\Gamma\ne0$ on $\Gamma$;
by the implicit function theorem $\Gamma$ is a $C^1$ hypersurface, and on $\Gamma$ let $\bn:=\grad d_\Gamma/|\grad d_\Gamma|$, which points into the $+$ side $\{d_\Gamma>0\}$.

\begin{prop}[fixed interface, trainable amplitude]\label{prop:interface_module}
Let $A=A_0+|d_\Gamma|N$ with $A_0,N\in H^1(\Omega;\mathbb K_n)$. Lemma~\ref{lem:mult_kink} applies with $\rho=|d_\Gamma|$ and $\jump{\grad\rho}=2\grad d_\Gamma$: $A\in H^1(\Omega;\mathbb
K_n)$, $\operatorname{Div}A+\bw\in H(\divvr;\Omega)$ for every remaining part $\bw$ as in Lemma~\ref{lem:mult_kink}, and at every
$x^*\in\Gamma$ near which $N$ is Lipschitz continuous and $A_0$ is continuously differentiable,
\[
 \jump{\p_{\bn}A}=2|\grad d_\Gamma|\,N,\qquad \jump{\operatorname{Div}A}=2N\grad d_\Gamma=2|\grad d_\Gamma|\,N\bn .
\]
Every Lipschitz tangential jump field is attained: let $d_\Gamma$ be in $C^{1,1}$ near $\Gamma$, let $\mathbf j:\Gamma\to\R^n$ be Lipschitz with $\mathbf j\cdot\bn=0$, let $U$ be a neighborhood of $\Gamma$ on which
$|\grad d_\Gamma|$ is bounded away from zero, $\widetilde{\mathbf j}$ a Lipschitz extension of $\mathbf j$ to $U$ (componentwise), $\widehat\bn:=\grad d_\Gamma/|\grad d_\Gamma|$ on $U$, and $\chi$ a Lipschitz cutoff with $\chi=1$ on $\Gamma$ and support in $U$. Then
\beq\label{module_jump_field}
 N:=\frac{\chi}{2|\grad d_\Gamma|}\bigl(\widetilde{\mathbf j}\,\widehat\bn^T-\widehat\bn\,\widetilde{\mathbf j}^T\bigr)
\eeq
(extended by $0$ outside $U$) is skew-symmetric and Lipschitz, and $\jump{\operatorname{Div}(|d_\Gamma|N)}=\mathbf j$ at every point of $\Gamma$. The geometry is fixed and the jump amplitude is trainable.
\end{prop}

\begin{proof}
On the two sides of $\Gamma$, $|d_\Gamma|=\pm d_\Gamma$ is continuously differentiable up to $\Gamma$ with gradients $\pm\grad d_\Gamma$, so $\jump{\grad\rho}=2\grad d_\Gamma$ and $\jump{\p_{\bn}\rho}=2|\grad d_\Gamma|$, and Lemma~\ref{lem:mult_kink} gives the identities. For \eqref{module_jump_field}: $\widetilde{\mathbf j}$, $\widehat\bn$, $1/|\grad d_\Gamma|$ and $\chi$ are bounded and
Lipschitz on $U$, so $N$ is Lipschitz and the identities hold at every point of $\Gamma$; there $\chi=1$ and $\widehat\bn=\bn$, so $2N\grad d_\Gamma=(\mathbf j\bn^T-\bn\mathbf j^T)\bn=\mathbf j-\bn(\mathbf j\cdot\bn)=\mathbf j$.
\end{proof}

The amplitude required for a jump $\mathbf j$
is $|N|\sim|\mathbf j|/|\grad d_\Gamma|$, uniformly controlled because $|\grad d_\Gamma|$ is bounded
away from zero on $\Gamma$. A signed distance, with $|\grad d_\Gamma|=1$ on $\Gamma$,
normalizes the geometric scale: the jump is $2N\bn$, and the effective amplitude satisfies $N\bn=\frac12\mathbf j$.

Proposition~\ref{prop:interface_module} concerns the function-space module: every Lipschitz tangential jump field
is attained exactly by a Lipschitz amplitude $N$. In
the classes $N$ is a network, and the jump of the finite
neural module approximates $\mathbf j$ as the network approximates $N$ in $H^1$.

\begin{cor}[approximation of the jump amplitude]\label{cor:jump_amplitude}
Let $U$ be a bounded Lipschitz domain with $\Gamma\subset\partial U$ (a
one-sided neighborhood of $\Gamma$), let $|\grad d_\Gamma|\le G$ on $\Gamma$, and for $M\in H^1(U;\mathbb K_n)$ let
$\mathcal J_\Gamma(M):=2|\grad d_\Gamma|\,M\bn\in L^2(\Gamma)^n$, with the trace of $M$. Then
\[
 \|\mathcal J_\Gamma(M)-\mathcal J_\Gamma(M')\|_{L^2(\Gamma)}\le 2\sqrt2\,G\,C_{\rm tr}\,\|M-M'\|_{H^1(U)}
\]
for $M,M'\in H^1(U;\mathbb K_n)$, with $C_{\rm tr}$ the constant of the trace inequality $\|w\|_{L^2(\Gamma)}\le C_{\rm tr}\|w\|_{H^1(U)}$. In particular, if $N$ is the amplitude \eqref{module_jump_field} of a jump $\mathbf j$ and $N_m$ are Lipschitz networks, then $\jump{\operatorname{Div}(|d_\Gamma|N_m)}=\mathcal J_\Gamma(N_m)$ and $\|\jump{\operatorname{Div}(|d_\Gamma|N_m)}-\mathbf j\|_{L^2(\Gamma)}\le2\sqrt2\,G\,C_{\rm tr}\|N-N_m\|_{H^1(U)}$.
\end{cor}

\begin{proof}
$\mathcal J_\Gamma$ is linear and $|M\bn|\le(\sum_{i,j}M_{ij}^2)^{1/2}=\sqrt2\,(\sum_{i<j}M_{ij}^2)^{1/2}$ pointwise; apply the trace inequality to each
entry. The identity for $N_m$ is Proposition~\ref{prop:interface_module} with $A_0=0$.
\end{proof}

\begin{rem}[Kellogg checkerboard]\label{rem:kellogg_ansatz}
Several fixed interfaces are treated by one module each, and where two of them
cross a product term is added. For the interfaces $\{x_1=0\}$ and
$\{x_2=0\}$ of the Kellogg checkerboard problem \cite{Kellogg:74,Zhang:kellogg} the two-dimensional potential is
\beq\label{kellogg_psi}
 \psi_\theta=N_0+|x_1|N_1+|x_2|N_2+|x_1||x_2|N_3 ,
\eeq
and the same form $v_\theta=V_0+|x_1|V_1+|x_2|V_2+|x_1||x_2|V_3$ is used for the scalar solution, which is continuous with coefficient-induced derivative jumps. On $\{x_1=0\}$, $\jump{\p_{x_1}\psi_\theta}=2N_1(0,x_2)+2|x_2|N_3(0,x_2)$ (Lemma~\ref{lem:mult_kink} with $\rho=|x_1|$ and the amplitude $N_1+|x_2|N_3$), and $\jump{\curl\psi_\theta\cdot\bt}$ is minus this (Corollary~\ref{cor:jump2d}); similarly on $\{x_2=0\}$. The two break lines are present for every parameter value. The singular behavior at the cross point, $u\sim r^\alpha\varphi(\vartheta)$ with $0<\alpha<1$, is left to the trainable networks.
\end{rem}

\subsection{Regularity recovery and enriched approximation}\label{subsec:regularity_recovery}

The module also recovers regularity: after the kink
is extracted, the remaining potential is globally smoother.

\begin{prop}[regularity recovery]\label{prop:regularity_recovery}
Let $\Gamma$ divide $\Omega$ into Lipschitz subdomains $\Omega_\pm$, let $d_\Gamma\in W^{1,\infty}(\Omega)$ with $d_\Gamma|_{U_\Gamma}\in C^{1,1}(U_\Gamma)$ on a neighborhood $U_\Gamma$ of
$\Gamma$, $\Gamma\subset\{d_\Gamma=0\}$, $\pm d_\Gamma>0$ on $\Omega_\pm\cap U_\Gamma$, and $|\grad d_\Gamma|$ bounded away from zero on $\Gamma$, and let $\bv\in
H(\divvr;\Omega)$ satisfy $\divv\bv=0$, $F_i(\bv)=0$, $i=1,\ldots,b$, and $\bv|_{\Omega_\pm}\in H^1(\Omega_\pm)^n$, with jump $\mathbf
j:=\jump{\bv}$ on $\Gamma$. If $N\in H^2(\Omega;\mathbb K_n)$ is supported in $U_\Gamma$ and $2|\grad d_\Gamma|N\bn=\mathbf j$ on $\Gamma$, then
\[
 \bv=\operatorname{Div}\bigl(A_0+|d_\Gamma|N\bigr),\qquad A_0\in H^2(\Omega;\mathbb K_n),\qquad \|A_0\|_2\le C\|\bv-\operatorname{Div}(|d_\Gamma|N)\|_1 .
\]
\end{prop}

\begin{proof}
Let $\bw:=\operatorname{Div}(|d_\Gamma|N)$. Since $|d_\Gamma|N\in H^1(\Omega;\mathbb K_n)$, $\bw$ lies in the space \eqref{Div_range}: $\divv\bw=0$ and $F_i(\bw)=0$.
On $\Omega_\pm$, $|d_\Gamma|N=\pm d_\Gamma N\in H^2$, so $\bw$ is in $H^1$ on both sides, and $\bw=N\grad|d_\Gamma|+|d_\Gamma|\,\operatorname{Div}N$ has the
traces $\pm N\grad d_\Gamma$ on $\Gamma$, since $|d_\Gamma|=0$ there; hence $\jump{\bw}=2N\grad d_\Gamma=2|\grad d_\Gamma|N\bn=\mathbf j$. The field $\bv_0:=\bv-\bw$ is divergence
free, has zero cavity fluxes, is in $H^1$ on both sides and has no jump across
$\Gamma$, so $\bv_0\in H^1(\Omega)^n$ (a function that is in $H^1$ on both sides of a Lipschitz interface and has equal
traces there is in $H^1$). Corollary~\ref{cor:regular_potential} with $r=1$ gives $\bv_0=\operatorname{Div}A_0$ with $A_0\in H^2(\Omega;\mathbb K_n)$ and $\|A_0\|_2\le C_1\|\bv_0\|_1$.
\end{proof}

For a flux $\btau=\operatorname{Div}A+\bw$, with $\bw=\Rop_nq+\bh$ or $\bw=\bzz$, the proposition is applied to its potential part $\bv=\operatorname{Div}A=\btau-\bw$, which is divergence free and has zero cavity fluxes, whenever that part is piecewise $H^1$; no regularity of $\bw$ is needed. The existence of
$N\in H^2$ is a hypothesis on the jump: piecewise $H^1$ regularity of $\bv$ gives only $\mathbf j\in H^{1/2}(\Gamma)^n$, whereas the trace of $2|\grad d_\Gamma|N\bn$ with $N\in H^2$ is smoother. It holds if $N_\Gamma:=(\mathbf j\bn^T-\bn\mathbf j^T)/(2|\grad d_\Gamma|)$
on $\Gamma$, for which $2|\grad d_\Gamma|N_\Gamma\bn=\mathbf j$, has an $H^2$ extension supported in $U_\Gamma$. Two facts are thus separate: a Lipschitz jump is realized exactly by a Lipschitz amplitude (Proposition~\ref{prop:interface_module}), and an $H^2$ amplitude yields an $H^2$ remainder. With the fixed-interface classes, the approximation problem then splits into approximation problems for smooth functions.

\begin{cor}[fixed-interface best approximation]\label{cor:enriched_hdiv}
Let $\btau=\operatorname{Div}(A_0+|d_\Gamma|N)+\bw$ with $A_0,N\in H^1(\Omega;\mathbb K_n)$, $d_\Gamma\in W^{1,\infty}(\Omega)$ and a remaining part $\bw$ of either
class, and let the class consist of $\operatorname{Div}(A_{0,m}+|d_\Gamma|N_m)+\bw_m$ with $A_{0,m}\in\cA_m$,
$N_m\in\cA_m'$, $\{\cA_m'\}$ a second family as in Section~\ref{subsec:hdiv_classes}, and $\bw_m=\Rop_nq_m+\bh$ with $q_m\in\cQ_m$, respectively $\bw_m=\bzz_m\in\cZ_m$. Then
\[
 \inf_{\btau_m}\|\btau-\btau_m\|_{H(\divvr)}\le\sqrt{n-1}\Bigl(\inf_{A_{0,m}\in\cA_m}|A_0-A_{0,m}|_1+C_d\inf_{N_m\in\cA_m'}\|N-N_m\|_1\Bigr)+E_{\bw},
\]
with $C_d:=\|d_\Gamma\|_{L^\infty(\Omega)}+\|\grad d_\Gamma\|_{L^\infty(\Omega)}$ and $E_{\bw}=(1+4C_{x_0}^2/n^2)^{1/2}\inf_{q_m\in\cQ_m}\|q-q_m\|_0$ in the first class, $E_{\bw}=\sqrt n\inf_{\bzz_m\in\cZ_m}\|\bzz-\bzz_m\|_1$ in the second.
\end{cor}

\begin{proof}
The upper bounds in the proofs of Theorems~\ref{thm:hdivd}
and~\ref{thm:hdivd_reg} hold for every $A_m$, in particular for
$A_m=A_{0,m}+|d_\Gamma|N_m$. Since $\grad(|d_\Gamma|M)=\operatorname{sign}(d_\Gamma)\,M\otimes\grad d_\Gamma+|d_\Gamma|\grad M$ almost everywhere, $\bigl||d_\Gamma|M\bigr|_1\le
C_d\|M\|_1$, which applied to $M=N-N_m$ gives the estimate.
\end{proof}

Under the $H^2$-amplitude hypothesis of Proposition~\ref{prop:regularity_recovery}, the prescribed amplitude $N$ is in $H^2$, and the recovered remainder $A_0$ is in $H^2$ as well.
The interface singularity is thus isolated in the explicit kink factor, and the remaining potential approximation problems are for
smooth functions; this is the intended purpose of the module. The remaining part, $q$ or $\bzz$, is not affected.

\subsection{Learned-interface potentials}\label{subsec:learned_interface}

The same mechanism works without prescribing the interface. With a trainable scalar function $g_\eta$ let
\beq\label{learned_interface_psi}
 A_{\theta,\eta}=A_{0,\theta}+\operatorname{ReLU}(g_\eta)\,N_\theta ;
\eeq
the zero level set of $g_\eta$ carries the geometry, while the jump is determined by the product $N_\theta\grad g_\eta$.

\begin{prop}[trainable interface]\label{prop:learned_interface_hdiv}
Let $g_\eta\in W^{1,\infty}(\Omega)$, $A_{0,\theta},N_\theta\in H^1(\Omega;\mathbb K_n)$ and $B_{\theta,\eta}:=\operatorname{ReLU}(g_\eta)N_\theta$.
Then $A_{\theta,\eta}\in H^1(\Omega;\mathbb K_n)$ and $\operatorname{Div}A_{\theta,\eta}+\bw\in H(\divvr;\Omega)$ for
every $(\theta,\eta)$ and every remaining part $\bw$ as
in Lemma~\ref{lem:mult_kink}, also where $\Gamma_\eta:=\{g_\eta=0\}$ has singular points,
branches, or points with $\grad g_\eta=0$; moving the learned
interface never destroys $H(\divvr)$ conformity. Let $x^*\in\Gamma_\eta$ be
a regular point: $g_\eta$ is continuously differentiable near
$x^*$ (for a ReLU network $g_\eta$: $x^*$ away from
its own activation interfaces) with $\grad g_\eta(x^*)\ne0$, and
$N_\theta$ is Lipschitz continuous near $x^*$. With the
$+$ side $\{g_\eta>0\}$ and $\bn_\eta=\grad g_\eta/|\grad g_\eta|$, in
the sense of one-sided essential limits at $x^*$,
\beq\label{learned_interface_Div_jump}
 \jump{\grad B_{\theta,\eta}}=N_\theta\otimes\grad g_\eta,\qquad
 \jump{\operatorname{Div}B_{\theta,\eta}}=N_\theta\grad g_\eta=|\grad g_\eta|\,N_\theta\bn_\eta,\qquad
 \bn_\eta\cdot\jump{\operatorname{Div}B_{\theta,\eta}}=0,
\eeq
and the same identities hold for $A_{\theta,\eta}$ if $A_{0,\theta}$ is continuously differentiable near $x^*$. At each regular point the jump map $N\mapsto N\grad g_\eta(x^*)$ is onto $\bn_\eta^\perp$: for
$\mathbf j\cdot\bn_\eta=0$, $N=|\grad g_\eta(x^*)|^{-1}(\mathbf j\bn_\eta^T-\bn_\eta\mathbf j^T)$ gives $N\grad g_\eta(x^*)=\mathbf j$. This is a pointwise statement; it does not assert that a fixed finite network realizes a prescribed jump field exactly along $\Gamma_\eta$.
\end{prop}

\begin{proof}
$\operatorname{ReLU}(g_\eta)\in W^{1,\infty}(\Omega)$, so the first part of Lemma~\ref{lem:mult_kink} gives conformity for every $(\theta,\eta)$; it uses neither the zero set nor its regularity. Near a regular point, $\Gamma_\eta$ is a $C^1$ hypersurface, $\operatorname{ReLU}(g_\eta)$ equals $g_\eta$ on $\{g_\eta>0\}$ and $0$ on $\{g_\eta<0\}$, both continuously differentiable up to $\Gamma_\eta$, so $\jump{\grad\rho}=\grad g_\eta(x^*)$, and Lemma~\ref{lem:mult_kink} gives the identities.
\end{proof}

The amplitude required at a regular point is $|N_\theta(x^*)|\sim|\mathbf j|/|\grad g_\eta(x^*)|$; uniform control of it on a part of $\Gamma_\eta$ requires a lower
bound $|\grad g_\eta|\ge c_0>0$ there, which the signed distance of the fixed module provides with $c_0=1$. The one-sided module $\operatorname{ReLU}(g_\eta)N_\theta$ is a convenient choice; $|g_\eta|N_\theta$ gives the two-sided analogue,
with jump $2N_\theta\grad g_\eta$ as in Proposition~\ref{prop:interface_module}. In both modules the jump is the amplitude network times the jump of the gradient of the multiplier.

The parameterization has the positive scaling symmetry
\[
 (g_\eta,N_\theta)\mapsto(\lambda g_\eta,N_\theta/\lambda),\quad\lambda>0,\qquad \operatorname{ReLU}(\lambda g_\eta)\,N_\theta/\lambda=\operatorname{ReLU}(g_\eta)\,N_\theta ,
\]
which leaves the realized potential, the zero set and the jump product $N_\theta\grad g_\eta$ unchanged. There may
be further nonuniqueness; for instance, where $N_\theta=0$ the zero set of $g_\eta$ is
not seen at all. Any unregularized functional that depends on $(g_\eta,N_\theta)$ only through the realized field is invariant
under this scaling, and a normalization of $g_\eta$ can remove this global scaling degeneracy.

For a fixed interface the results above give the exact realization of every Lipschitz admissible jump, the recovery of regularity, and an enriched best-approximation estimate. The learned-interface results establish
conformity for every parameter value and the local jump mechanism; no theorem on the global recovery of the interface or on an improved approximation rate is asserted for Level~II.

As a realizable function class, a class with a module contains the class without it ($N_\theta=0$, the
same $A_{0,\theta}$), so the conformity and approximation upper bounds above, in particular \eqref{hdiv_d_approx} and \eqref{hdiv_reg_approx}, remain valid for it; this does not compare the parameter budgets of
the two architectures. The purpose of a module is to improve the finite-parameter representation
of an interface kink, without forcing the flux into $H^1(\Omega)^n$, whose components have single-valued traces on an
interior interface and therefore no jump there: the potential is in $H^1$, and its kink carries the jump.

\subsection{Two and three dimensions}\label{subsec:hdiv23_jumps}

With $A=\psi\mathcal E_{12}$ for $n=2$ and $A_{ij}=\sum_k\varepsilon_{ijk}\Psi_k$ for $n=3$
(Section~\ref{subsec:hdiv23}), the results of this section read as follows.

\begin{cor}[two-dimensional jumps]\label{cor:jump2d}
\begin{enumerate}
\item[(a)] If $\psi$ is continuous and piecewise affine on a finite polygonal partition (for instance a finite ReLU
network), then $\curl\psi$ is piecewise constant, belongs to $H(\divvr;\Omega)$, and across
the relative interior of every interior edge $\jump{\curl\psi\cdot\bn}=0$ and $\jump{\curl\psi\cdot\bt}=-\jump{\p_{\bn}\psi}$.
\item[(b)] Under the hypotheses of Proposition~\ref{prop:jump}, with $A=\psi\mathcal E_{12}$: $\jump{\p_{\bt}\psi}=0$,
$\jump{\p_{\bn}\psi}=-\jump{\btau\cdot\bt}$ and $\jump{\btau\cdot\bn}=0$; the kink of $\psi$ is determined by the tangential jump of $\btau$.
\item[(c)] At points of $\Gamma$ near which $N_1$ is Lipschitz continuous and $N_0$ is continuously differentiable,
$\psi=N_0+|d_\Gamma|N_1$ has $\jump{\p_{\bn}\psi}=2|\grad d_\Gamma|N_1$ and $\jump{\curl\psi\cdot\bt}=-2|\grad d_\Gamma|N_1$ ($2N_1$ and $-2N_1$ for a signed distance). For
\beq\label{learned_interface_psi_2d}
 \psi_{\theta,\eta}=N_{0,\theta}+\operatorname{ReLU}(g_\eta)\,N_{1,\theta},
\eeq
at a regular point of $\Gamma_\eta$ near which
$N_{1,\theta}$ is Lipschitz continuous and $N_{0,\theta}$ is continuously differentiable,
with $\bt_\eta=(-n_{\eta,2},n_{\eta,1})^T$, $\jump{\grad\psi_{\theta,\eta}}=N_{1,\theta}\grad g_\eta$, $\jump{\curl\psi_{\theta,\eta}\cdot\bn_\eta}=0$ and $\jump{\curl\psi_{\theta,\eta}\cdot\bt_\eta}=-N_{1,\theta}|\grad g_\eta|$.
\end{enumerate}
\end{cor}

\begin{proof}
Corollary~\ref{lem:cpwl}, Proposition~\ref{prop:jump} and Propositions~\ref{prop:interface_module} and~\ref{prop:learned_interface_hdiv} with $A=\psi\mathcal E_{12}$, $N=N_1\mathcal E_{12}$: $\jump{\p_{\bn}A}=\jump{\p_{\bn}\psi}\mathcal E_{12}$ and $\mathcal E_{12}\bn=-\bt$.
\end{proof}

\begin{cor}[three-dimensional jumps]\label{cor:jump3d}
\begin{enumerate}
\item[(a)] If $\bPsi\in H^1(\Omega)^3$ is piecewise $H^2$ on the two sides of a smooth interface with unit normal $\bn$ (or continuous and
piecewise affine, across the relative interior of every interior face), then $\jump{\curll\bPsi}=\bn\times\jump{\p_{\bn}\bPsi}$.
\item[(b)] For $\bPsi=\bPsi_0+|d_\Gamma|\mathbf N$ with $d_\Gamma$ as in Section~\ref{subsec:interface}, at points of $\Gamma$ near which $\mathbf N$ is Lipschitz
continuous and $\bPsi_0$ is continuously differentiable, $\jump{\curll\bPsi}=2\grad d_\Gamma\times\mathbf N$.
\item[(c)] For $\bPsi_{\theta,\eta}=\bPsi_{0,\theta}+\operatorname{ReLU}(g_\eta)\mathbf N_\theta$, at regular points of $\Gamma_\eta$ near which $\mathbf N_\theta$ is Lipschitz
continuous and $\bPsi_{0,\theta}$ is continuously differentiable, $\jump{\curll\bPsi_{\theta,\eta}}=\grad g_\eta\times\mathbf N_\theta$.
\end{enumerate}
All these jumps are tangential.
\end{cor}

\begin{proof}
With $A_{ij}=\sum_k\varepsilon_{ijk}\Psi_k$ and $N_{ij}=\sum_k\varepsilon_{ijk}N_k$, $\jump{\p_{\bn}A}\bn=\bn\times\jump{\p_{\bn}\bPsi}$ and $N\bv=\bv\times\mathbf N$; (a) is Lemma~\ref{lem:skew}(b) and Corollary~\ref{lem:cpwl}, (b) and (c) are Lemma~\ref{lem:mult_kink} with $\jump{\grad\rho}=2\grad d_\Gamma$ and $\jump{\grad\rho}=\grad g_\eta$ (Propositions~\ref{prop:interface_module} and~\ref{prop:learned_interface_hdiv}).
\end{proof}

\subsection{A divergence that jumps}\label{subsec:div_jumps}

The tangential jumps above are carried by the potential in both
classes. A jump of the divergence concerns the remaining part $\bw$. If the
divergence $q^*=\divv\btau^*$ of the target is piecewise continuous and jumps across $\Gamma$,
the two classes carry this jump as follows. In the first class,
$q_m$ itself may jump: a piecewise $q_m$ for a known interface, or
$q_m=Q_0+\chi Q_1$ with $\chi$ the indicator function of $\{g_\eta>0\}$ for a
learned one; $\Rop_nq_m$ remains continuous under the segment condition of Lemma~\ref{lem:Rcont}, and conformity
is not affected, since only $q_m\in L^2$ is required.
In the second class, $\bzz_m\in H^1$ has no interface jump, and the jump of its divergence comes from a kink.

\begin{lem}[divergence of a kinked field]\label{lem:div_kink}
Let $\rho\in W^{1,\infty}(\Omega)$, $\bzz_0,\mathbf{Z}_1\in H^1(\Omega)^n$ and $\bzz:=\bzz_0+\rho\mathbf{Z}_1$. Then $\bzz\in H^1(\Omega)^n$. Let
further $x^*\in\Omega$, $\Gamma$, $\bn$ and $\rho$ be as in Lemma~\ref{lem:mult_kink}, let $\mathbf{Z}_1$
be Lipschitz continuous and $\bzz_0$ continuously differentiable near $x^*$. Then,
in the sense of one-sided essential limits at $x^*$, $\jump{\bzz}=0$ and
\[
 \jump{\divv\bzz}=\mathbf{Z}_1(x^*)\cdot\jump{\grad\rho}=\jump{\p_{\bn}\rho}\,\mathbf{Z}_1(x^*)\cdot\bn .
\]
For $\rho=|d_\Gamma|$ this is $2\mathbf{Z}_1\cdot\grad d_\Gamma$, and for $\rho=\operatorname{ReLU}(g_\eta)$ it is $\mathbf{Z}_1\cdot\grad g_\eta$. If $d_\Gamma\in C^{1,1}$ near $\Gamma$, every Lipschitz jump $j$ on $\Gamma$ is attained with the Lipschitz amplitude $\mathbf{Z}_1=\chi\,\tilde j\,\grad d_\Gamma/(2|\grad d_\Gamma|^2)$, $\tilde j$ a Lipschitz extension of $j$ and $\chi$ a cutoff as in Proposition~\ref{prop:interface_module}.
\end{lem}

\begin{proof}
As in Lemma~\ref{lem:mult_kink}, applied to the components of $\bzz$: $\grad(\rho\mathbf{Z}_1)=\mathbf{Z}_1\otimes\grad\rho+\rho\grad\mathbf{Z}_1\in L^2$, and near $x^*$ the term $\rho\grad\mathbf{Z}_1$ tends to zero, so $\jump{\grad\bzz}=\mathbf{Z}_1(x^*)\otimes\jump{\grad\rho}$, whose trace is the jump of $\divv\bzz$; $\bzz$ itself is continuous. For the last statement, $2\mathbf{Z}_1\cdot\grad d_\Gamma=\chi\tilde j=j$ on $\Gamma$.
\end{proof}

In $\Sigma_{m,n}^{\rm reg}$ the module $\bzz_m=\bzz_{0,m}+|d_\Gamma|\mathbf{Z}_{1,m}$ for a known interface, or $\bzz_m=\mathbf{Z}_{0,m}+\operatorname{ReLU}(g_\eta)\mathbf{Z}_{1,m}$ for a learned one, therefore carries the jump of the divergence with the same factor as the potential; in a generic network, a
kink of $\bzz_m$ at the zero set of a unit of a shared last layer does the same. When the divergence of the target is continuous across $\Gamma$, no module in $\bzz_m$ is needed.

\subsection{Three implementation levels}\label{subsubsec:hdiv_levels}

The three levels below describe how the interface geometry enters the skew potential carrying the flux jump, in the same way for both classes. If the divergence also jumps, the same interface
information is used in the remaining part, as described in Section~\ref{subsec:div_jumps}. With the $n(n-1)/2$ entries $a_{ij}$, $i<j$, of the potential as scalar network outputs, the flux is
\beq\label{hdiv_impl_common}
 \btau_{\theta,\eta}=\operatorname{Div}A_{\theta,\eta}+\bw_\theta,\qquad \bw_\theta=\Rop_nq_\theta+\textstyle\sum_rc_r\bh_r\quad\text{or}\quad\bw_\theta=\bzz_\theta ,
\eeq
and the levels are the following choices of $A_{\theta,\eta}$, with $A_{0,\theta}$ and $N_\theta$ skew.\par\noindent \emph{Level~I} (known interface): $A=A_{0,\theta}+|d_\Gamma|N_\theta$ with a fixed level-set function $d_\Gamma$; the geometry is
given; the amplitude network is $N_\theta$, and the module contributes the flux jump $\jump{\operatorname{Div}(|d_\Gamma|N_\theta)}=2N_\theta\grad d_\Gamma$ ($2N_\theta\bn$ for a signed distance), which is the jump of $\operatorname{Div}A$ where $A_{0,\theta}$ has
no kink on $\Gamma$. The $n(n-1)/2$ entries of $A_{0,\theta}$ and of $N_\theta$ are scalar outputs, for instance of one network with a linear output
layer; the module doubles the number of output coefficients of the potential and adds no trainable parameter for the geometry. The function $d_\Gamma$ is
data, evaluated in closed form, for instance $d_\Gamma(x)=x_1-s$ for the plane $\{x_1=s\}$ or $d_\Gamma(x)=|x-x_C|-R$ for a sphere (smooth away from $x_C$); both are signed distances,
but any level-set function as in Proposition~\ref{prop:interface_module} can be used. The module contains an explicit multiplication and need not itself be representable by an ordinary finite ReLU
network; when both factors are continuous and piecewise linear, their product is in general piecewise quadratic; in a ReLU-based implementation $|d_\Gamma|N_\theta=2\operatorname{ReLU}(d_\Gamma)N_\theta-d_\Gamma N_\theta$, and the last term has no additional kink across $\Gamma$. The
derivative is evaluated piecewise from the explicit formula $\grad(|d_\Gamma|N_\theta)=\operatorname{sign}(d_\Gamma)\,N_\theta\otimes\grad d_\Gamma+|d_\Gamma|\,\grad N_\theta$ (almost everywhere), so no numerical differentiation across the kink and no derivative of $\operatorname{sign}(d_\Gamma)$ is
required; since the residual integrand is in general only piecewise smooth and may jump across $\Gamma$, the quadrature is split along
$\Gamma$, which is possible because $\Gamma$ is known. With a smooth activation the identities of Proposition~\ref{prop:interface_module} hold at every point of $\Gamma$; with ReLU layers they hold away from
the points where activation interfaces of $A_{0,\theta}$ meet $\Gamma$.\par\noindent \emph{Level~II} (learned interface): $A=A_{0,\theta}+\operatorname{ReLU}(g_\eta)N_\theta$ with a trainable scalar network $g_\eta$; the geometry is learned through $g_\eta$, whose parameters $\eta$ are the
only trainable parameters of the geometry; the amplitude network is $N_\theta$, and the module contributes the flux jump $\jump{\operatorname{Div}(\operatorname{ReLU}(g_\eta)N_\theta)}=N_\theta\grad g_\eta$ at regular points of
$\Gamma_\eta$, which is the jump of $\operatorname{Div}A$ where $A_{0,\theta}$ has no kink on $\Gamma_\eta$. By Proposition~\ref{prop:learned_interface_hdiv} the field is in $H(\divvr)$ for every $\eta$.
The derivative is evaluated piecewise from $\grad(\operatorname{ReLU}(g_\eta)N_\theta)=\chi_{\{g_\eta>0\}}(N_\theta\otimes\grad g_\eta+g_\eta\grad N_\theta)$ (almost everywhere); for a ReLU network $g_\eta$ the jump formula holds at points of $\Gamma_\eta$ away from the activation interfaces
of $g_\eta$. A finite ReLU network $g_\eta$ is continuous and piecewise linear, so its zero set $\Gamma_\eta$ is piecewise flat on its activation partition; a smooth
geometry network can produce a genuinely curved zero set. Since the kink moves with $\eta$, the numerical integration must account for the moving zero set, either by cutting the quadrature cells along the
current zero set of $g_\eta$ (the derivative with respect to $\eta$ then contains the motion of the cut) or by a cellwise weak treatment of the derivatives; a fixed pointwise
rule with automatic differentiation alone does not in general capture the contribution associated with the motion of the kink.\par\noindent \emph{Level~III}
(generic): $A=A_\theta$, an ordinary skew network without interface variable; the geometry is implicit in the activation partition. With ReLU layers the zero sets
of the hidden preactivations provide the possible finite-width kink sets (a zero set produces no kink in the output if, for instance, its downstream coefficients vanish); with smooth activations
an exact jump is absent at finite width and can only be approximated by a smooth transition layer.\par If the target divergence is
continuous across the interface, the remaining part $\bw_\theta$, given by the networks $q_\theta$ and the harmonic coefficients, or $\bzz_\theta$, is unchanged across the three levels; if the divergence jumps, it
is treated as in Section~\ref{subsec:div_jumps}; in two and three dimensions $\psi$ and $\bPsi$ (Section~\ref{subsec:hdiv23}) replace $A$. The fields are evaluated as described in Section~\ref{sec:framework}.

\section{Native neural approximation in \texorpdfstring{$H(\curll)$}{H(curl)}}\label{sec:hcurl}

\subsection{Motivation: the de Rham sequence}\label{subsec:hcurl_derham}

In two and three dimensions the de Rham sequence at $H(\curll)$ is
\[
 H^1(\Omega)\xrightarrow{\ \grad\ }H(\curll;\Omega)\xrightarrow{\ \curll\ }L^2(\Omega)\longrightarrow0\quad(n=2),
\]
\[
 H^1(\Omega)\xrightarrow{\ \grad\ }H(\curll;\Omega)\xrightarrow{\ \curll\ }H(\divvr;\Omega)\quad(n=3),
\]
with the scalar curl $\curll\bv=\p_1v_2-\p_2v_1$ for $n=2$. The intrinsic continuity
of $H(\curll)$ is tangential continuity: across an interface the tangential trace is single
valued, while the normal component may jump. In two dimensions
the picture is that of Sections~\ref{sec:hdiv} and~\ref{sec:hdiv_reg} rotated, and both representations are available:
$\curll$ is onto $L^2(\Omega)$ and has a bounded right inverse
$\mathcal S$, so $r:=\curll\bE$ is an independent $L^2$ function, $\bE-\mathcal Sr$ is curl free and hence
a gradient up to finitely many fields determined by the topology, and
$\bE=\grad\phi+\mathcal Sr+\bh$; or, by the regular decomposition, $\bE=\grad\phi+\bzz$ with $\phi\in H^1(\Omega)$ and $\bzz\in H^1(\Omega)^2$. In
three dimensions $\curll$ is not onto: its range is divergence free,
so an unconstrained network cannot represent the curl, and only the regular decomposition
$\bE=\grad\phi+\bzz$ with $\bzz\in H^1(\Omega)^3$ is used. In all cases the gradient on the
left carries the normal jumps through the kinks of $\phi$, the remaining part
has no interface jump in the situations of interest ($\mathcal Sr+\bh$ is
continuous for a bounded piecewise continuous $r$ under the segment condition of Lemma~\ref{lem:Rcont},
rotated, and $\bzz\in H^1$ has a single-valued trace), and networks for
the representing functions give a field in $H(\curll;\Omega)$ for every value of the parameters.

\subsection{Two dimensions}\label{subsec:hcurl2d}

For $\bv=(v_1,v_2)^T$ let $J(v_1,v_2)^T:=(-v_2,v_1)^T$ and $K:=-J$. With the scalar curl of Section~\ref{subsec:hcurl_derham},
\beq\label{hcurl_rotation}
 \curll\bv=\divv(K\bv),\qquad (K\bv)\cdot\bn=\bv\cdot\bt,\qquad K\grad\phi=\curl\phi ,
\eeq
so $K$ ($=\mathcal E_{12}$ of Section~\ref{subsec:hdiv23}) is an isometry
from $H(\curll;\Omega)$ onto $H(\divvr;\Omega)$ that maps tangential traces to normal traces.
With $\mathcal S:=K^{-1}\Rop_2=J\Rop_2$, Lemma~\ref{lem:Rd} gives $\curll(\mathcal Sr)=r$ and $\|\mathcal Sr\|_0\le C_{x_0}\|r\|_0$. As for $H(\divvr)$,
there are two representations: the rotation of Corollary~\ref{cor:hdiv23}
and the rotation of Proposition~\ref{prop:hdiv_regular} with $n=2$ give
\beq\label{hcurl_representation}
 \begin{aligned}
 &\bE=\grad\phi+\mathcal Sr+\bh_{\curll},\quad r=\curll\bE,\quad \bh_{\curll}\in J\cH,\qquad\text{or}\\
 &\bE=\grad\phi+\bzz,\quad \bzz\in H^1(\Omega)^2,\quad \curll\bzz=\curll\bE .
 \end{aligned}
\eeq
In the first, $r$ and $\bh_{\curll}$ are unique, the
potential is unique up to a constant, and
\beq\label{hcurl_2d_stability}
 |\phi|_1+\|\bh_{\curll}\|_0+\|r\|_0\le C\|\bE\|_{H(\curll)} ,
\eeq
the rotation of the bound in Corollary~\ref{cor:hdiv23}. The cohomology part is determined by the circulations of $\bE$ around the holes. By \eqref{hcurl_rotation} the circulation $\langle\bv\cdot\bt,1\rangle_{\Gamma_j}$ is the flux of $K\bv$ through $\Gamma_j$: for $J\bh_i$ it is the flux of $KJ\bh_i=\bh_i$, that is $-2\pi\delta_{ij}$; for $\grad\phi$ it is zero, $\phi$ being single valued on the closed curve $\Gamma_j$; and for $\mathcal Sr$ it
is the flux of $\Rop_2r$, which is zero by Corollary~\ref{cor:R_fluxes}. So the coefficients of $\bh_{\curll}$ are the circulations of $\bE$ divided by $-2\pi$. In the second, $\phi$ and $\bzz$ depend linearly on $\bE$ with $\|\phi\|_1+\|\bzz\|_1\le C\|\bE\|_{H(\curll)}$ (Corollary~\ref{cor:hcurl2d_reg}), and no cohomology part appears. Let $\{\cP_m\}$, $\{\cQ_m\}$ and $\{\cZ_m\}$ be nested families with $\cP_m\subset H^1(\Omega)$, $\cQ_m\subset L^2(\Omega)$ and $\cZ_m\subset H^1(\Omega)^2$, and let
\beq\label{X_hcurl_2d}
 \begin{aligned}
 &\cE_m^{\rm dR}:=\bigl\{\grad\phi_m+\mathcal Sr_m+\bh_{\curll}:\ \phi_m\in\cP_m,\ r_m\in\cQ_m,\ \bh_{\curll}\in J\cH\bigr\},\\
 &\cE_m^{\rm reg}:=\bigl\{\grad\phi_m+\bzz_m:\ \phi_m\in\cP_m,\ \bzz_m\in\cZ_m\bigr\}.
 \end{aligned}
\eeq

\begin{cor}[first two-dimensional neural $H(\curll)$ class]\label{cor:hcurl_density}
For every value of the trainable parameters, every element of $\cE_m^{\rm dR}$ lies in $H(\curll;\Omega)$ and $\curll(\grad\phi_m+\mathcal Sr_m+\bh_{\curll})=r_m$. For $\bE$ as in \eqref{hcurl_representation}
\beq\label{hcurl_2d_approx}
 \inf_{r_m\in\cQ_m}\|r-r_m\|_0\le\inf_{\bE_m\in\cE_m^{\rm dR}}\|\bE-\bE_m\|_{H(\curll)}
 \le\inf_{\phi_m\in\cP_m}|\phi-\phi_m|_1+\sqrt{1+C_{x_0}^2}\,\inf_{r_m\in\cQ_m}\|r-r_m\|_0 .
\eeq
\end{cor}

\begin{proof}
Apply the isometry $K$ and Corollary~\ref{cor:hdiv23}. Directly: $\bE-\bE_m=\grad(\phi-\phi_m)+\mathcal S(r-r_m)$,
the first term is curl free, and $\|\mathcal
S(r-r_m)\|_{H(\curll)}^2=\|\Rop_2(r-r_m)\|_0^2+\|r-r_m\|_0^2\le(1+C_{x_0}^2)\|r-r_m\|_0^2$. The lower bound holds since $\curll\bE_m=r_m$ gives $\|\bE-\bE_m\|_{H(\curll)}\ge\|r-r_m\|_0$.
\end{proof}

\begin{cor}[second two-dimensional neural $H(\curll)$ class]\label{cor:hcurl2d_reg}
Every $\bE\in H(\curll;\Omega)$ can be written $\bE=\grad\phi+\bzz$ with $\phi\in H^1(\Omega)$
and $\bzz\in H^1(\Omega)^2$ depending linearly on $\bE$, $\|\phi\|_1+\|\bzz\|_1\le C\|\bE\|_{H(\curll)}$ and $\curll\bE=\curll\bzz$.
For every value of the trainable parameters, every
element of $\cE_m^{\rm reg}$ lies in $H(\curll;\Omega)$ and $\curll(\grad\phi_m+\bzz_m)=\curll\bzz_m$, and
\beq\label{hcurl_2d_reg_approx}
 \inf_{\bE_m\in\cE_m^{\rm reg}}\|\bE-\bE_m\|_{H(\curll)}\le\inf_{\phi_m\in\cP_m}|\phi-\phi_m|_1+\sqrt2\inf_{\bzz_m\in\cZ_m}\|\bzz-\bzz_m\|_1 .
\eeq
\end{cor}

\begin{proof}
Apply Proposition~\ref{prop:hdiv_regular} with $n=2$ to $K\bE$: $K\bE=\curl\psi+\bw$ with $\psi\in
H^1(\Omega)$ and $\bw\in H^1(\Omega)^2$. Since $K^{-1}=J$ and $J\curl\psi=JK\grad\psi=\grad\psi$ by
\eqref{hcurl_rotation}, $\bE=\grad\psi+J\bw$; take $\phi=\psi$ and $\bzz=J\bw$, so that $\|\phi\|_1+\|\bzz\|_1\le C\|K\bE\|_{H(\divvr)}
=C\|\bE\|_{H(\curll)}$, and $\curll\bzz=\divv(KJ\bw)=\divv\bw=\divv K\bE=\curll\bE$. For the classes, $\curll\grad\phi_m=0$, and for $\bw\in H^1(\Omega)^2$, $|\curll\bw|^2\le2(|\p_1w_2|^2+|\p_2w_1|^2)$, so $\|\bw\|_{H(\curll)}\le\sqrt2\|\bw\|_1$.
\end{proof}

As for $H(\divvr)$, the second class is the componentwise approach enriched
by the curl-free field $\grad\phi_m$, which carries the interface jumps (Section~\ref{subsec:hcurl_derham}).

The correct finite-parameter continuity is already visible for piecewise affine potentials.

\begin{lem}[piecewise affine scalar potentials]\label{lem:hcurl_cpwl}
If $\phi$ is continuous and piecewise affine on a finite polygonal partition, then $\grad\phi$ is piecewise constant and belongs to $H(\curll;\Omega)$, and across every interior edge with unit tangent $\bt$ and normal $\bn$
\beq\label{hcurl_cpwl_jump}
 \jump{\grad\phi\cdot\bt}=0,\qquad \jump{\grad\phi\cdot\bn}\ \mbox{may be nonzero}.
\eeq
This applies to every finite scalar ReLU network.
\end{lem}

\begin{proof}
The trace of $\phi$ on the edge is
single valued, so its tangential derivative is the
same from both sides; the piecewise curl of
$\grad\phi$ vanishes, and the continuity of the tangential trace
excludes an edge contribution to the distributional curl.
\end{proof}

A finite ReLU scalar potential therefore generates jumps of exactly the admissible $H(\curll)$ type, in both classes; this is the rotation of Corollary~\ref{cor:jump2d}(a).

\subsection{Three dimensions}\label{subsec:hcurl3d}

In three dimensions the rotation is not available,
and the regular decomposition replaces it. We consider $\Gamma_D=\emptyset$ or $\Gamma_D=\partial\Omega$
only and let $H_{\Gamma_D}(\curll;\Omega):=\{\bu\in H(\curll;\Omega):\bu\times\bn=0\mbox{ on }\Gamma_D\}$, with
the trace in the usual weak sense, and
$H^1_{\Gamma_D}(\Omega):=\{v\in H^1(\Omega):v=0\mbox{ on }\Gamma_D\}$; thus $H_{\emptyset}(\curll;\Omega)=H(\curll;\Omega)$ and $H_{\partial\Omega}(\curll;\Omega)=H_0(\curll;\Omega)$. For
a bounded Lipschitz domain in $\R^3$ the regular decomposition
\beq\label{regular_decomposition}
 \bE=\grad\phi+\bzz,\qquad \phi\in H^1_{\Gamma_D}(\Omega),\quad \bzz\in H^1_{\Gamma_D}(\Omega)^3,
\eeq
holds for every $\bE\in H_{\Gamma_D}(\curll;\Omega)$, with $\phi$ and $\bzz$ depending linearly and boundedly on $\bE$,
\beq\label{regular_decomposition_stability}
 \|\phi\|_1+\|\bzz\|_1\le C\|\bE\|_{H(\curll)} ,
\eeq
so that, once the decomposition operator is fixed, $\bE\mapsto(\phi,\bzz)$ is
a bounded linear map \cite{HiptmairPechstein2019}, \cite[Lemma~2.4]{Hip:02}. For $\Gamma_D=\emptyset$ this follows from the
regularized Poincar\'e operators $R_\ell$ and the smoothing operators $K_\ell$ of \cite[Theorem~4.6]{CostabelMcIntosh2010}, which satisfy
$\grad R_1\bE+R_2\curll\bE=\bE-K_1\bE$ with $R_1,R_2:L^2\to H^1$ bounded and $K_1\bE$ smooth: $\phi=R_1\bE$ and $\bzz=R_2\curll\bE+K_1\bE$.
For $\Gamma_D=\partial\Omega$ the Bogovski\u{\i}-type operators $T_\ell$, $L_\ell$ of the same theorem, applied
to the extension of $\bE$ by zero, give $\phi\in
H^1_0(\Omega)$ and $\bzz\in H^1_0(\Omega)^3$ in the same way. There is no topological
restriction and no separate harmonic part: a topological component is absorbed
into the $H^1$ field $\bzz$. The neural realization is
\beq\label{structured3d}
 \bE_\theta=\grad\phi_\theta+\bzz_\theta
\eeq
with an ordinary scalar network $\phi_\theta$ and an ordinary componentwise vector network $\bzz_\theta$. With the boundary conditions imposed on $\phi_\theta$ and
$\bzz_\theta$, $\bE_\theta\in H_{\Gamma_D}(\curll;\Omega)$ and $\curll\bE_\theta=\curll\bzz_\theta$ for every parameter value; the represented field need not lie in $H^1(\Omega)^3$, since the gradient part may have normal jumps.

\begin{lem}[intrinsic $H(\curll)$ continuity]\label{lem:hcurl_intrinsic}
If $\phi$ is continuous and piecewise affine on a polyhedral partition of
$\Omega\subset\R^3$, then $\grad\phi$ is piecewise constant, $\grad\phi\in H(\curll;\Omega)$, and
across every interior face $F$ with unit normal $\bn_F$
\beq\label{hcurl_intrinsic_jump}
 \bn_F\times\jump{\grad\phi}=0,\qquad \jump{\grad\phi\cdot\bn_F}\ \mbox{may be nonzero}.
\eeq
Consequently $\grad\phi+\bzz$ with $\bzz\in H^1(\Omega)^3$ has the same tangential continuity and may retain the normal jump of $\grad\phi$.
\end{lem}

\begin{proof}
The restriction of $\phi$ to $F$ is single valued,
so all tangential derivatives agree from both sides, which is $\bn_F\times\jump{\grad\phi}=0$;
the distributional curl therefore has no surface part on $F$.
\end{proof}

Let $\{\cP_m\}_{m\ge1}$ and $\{\cZ_m\}_{m\ge1}$ be nested families with
$\cP_m\subset H^1_{\Gamma_D}(\Omega)$ and $\cZ_m\subset H^1_{\Gamma_D}(\Omega)^3$, and let $\cE_m:=\{\grad\phi_m+\bzz_m:\ \phi_m\in\cP_m,\ \bzz_m\in\cZ_m\}$.

\begin{prop}[three-dimensional neural $H(\curll)$ class]\label{prop:hcurl3d_density}
Every realization in $\cE_m$ lies in $H_{\Gamma_D}(\curll;\Omega)$, and for a regular decomposition $\bE=\grad\phi+\bzz$
\beq\label{hcurl_3d_approx}
 \inf_{\bE_m\in\cE_m}\|\bE-\bE_m\|_{H(\curll)}\le\inf_{\phi_m\in\cP_m}|\phi-\phi_m|_1+\sqrt2\inf_{\bzz_m\in\cZ_m}\|\bzz-\bzz_m\|_1 .
\eeq
\end{prop}

\begin{proof}
For $\bE_m=\grad\phi_m+\bzz_m$ we have $\curll\bE_m=\curll\bzz_m$, hence $\|\bE-\bE_m\|_{H(\curll)}\le|\phi-\phi_m|_1+\|\bzz-\bzz_m\|_{H(\curll)}$, and $\|\bv\|_{H(\curll)}^2=\|\bv\|_0^2+\|\curll\bv\|_0^2\le\|\bv\|_0^2+2|\bv|_1^2\le2\|\bv\|_1^2$ since $|\curll\bv|^2\le2|\grad\bv|^2$ pointwise; the regular decomposition is surjective, so every $\bE$ has such a representation.
\end{proof}

The representing functions inherit the regularity of the field.

\begin{rem}[smoother decompositions]\label{rem:hcurl_regular}
If $\Gamma_D=\emptyset$ and $\bE,\curll\bE\in H^r(\Omega)^3$, $r\ge0$, the regular decomposition can be
chosen with $\phi\in H^{r+1}(\Omega)$, $\bzz\in H^{r+1}(\Omega)^3$ and $\|\phi\|_{r+1}+\|\bzz\|_{r+1}\le C_r(\|\bE\|_r+\|\curll\bE\|_r)$: a field $\bzz_0\in
H^{r+1}(\Omega)^3$ with $\curll\bzz_0=\curll\bE$ exists by \cite[Corollary~4.7(a)]{CostabelMcIntosh2010}, and the curl-free remainder $\bE-\bzz_0$ is
$\grad\phi$ plus a smooth field of a finite-dimensional space
\cite[Theorem~4.9(b)]{CostabelMcIntosh2010}, which is added to $\bzz_0$. This is used in Proposition~\ref{prop:hcurl_recovery}.
\end{rem}

\begin{rem}[regular decomposition, not Helmholtz decomposition]\label{rem:rd_not_helmholtz}
\eqref{regular_decomposition} is not a Helmholtz decomposition: $\bzz$ is neither divergence free nor orthogonal to gradients. This matters for the neural parameterization, since on a nonconvex domain the divergence-free part of a Helmholtz
decomposition may lie in $H(\curll)\cap H(\divvr)$ without belonging to $H^1(\Omega)^3$ \cite{Cos:91}; the regular decomposition is used because both of its parts are in $H^1$.
\end{rem}

If the classes for the representing functions have dense
unions in the corresponding $H^1$ spaces (with their boundary conditions),
and in $L^2$ for the curl $r$ of the two-dimensional class, then the
native $H(\curll)$ classes are dense in the graph space, by \eqref{hcurl_2d_approx} and \eqref{hcurl_3d_approx}.

\subsection{Normal jumps and interface modules}\label{subsec:hcurl_jumps}

The mechanism of Lemmas~\ref{lem:hcurl_cpwl} and~\ref{lem:hcurl_intrinsic} holds near a smooth interface as well, also with the interface modules of Sections~\ref{subsec:interface} and~\ref{subsec:learned_interface} applied
to the scalar potential. With $P_\tau:=I-\bn\otimes\bn$ the tangential projection, $P_\tau\jump{\bE}=0$ is the continuity of the tangential trace; for $n=3$ it is $\bn\times\jump{\bE}=0$.

\begin{prop}[normal jumps carried by the scalar potential]\label{prop:hcurl_jump}
Let $\Omega\subset\R^n$, $n=2,3$, let $\Gamma$ be a smooth interface with unit normal $\bn$, and let $\bE=\grad\phi+\bw$ near $\Gamma$, where $\phi\in H^1$ is piecewise $H^2$ on the two sides of $\Gamma$, and $\bw\in H(\curll)$ has no jump across $\Gamma$ (its one-sided traces coincide; $\bw\in H^1$ is sufficient). Then, with one-sided traces,
\[
 P_\tau\jump{\bE}=0,\qquad \jump{\bE\cdot\bn}=\jump{\p_{\bn}\phi} ,
\]
and a prescribed normal jump $j$ is carried by the kink $\jump{\p_{\bn}\phi}=j$, which
it determines. For the modules the same formulas hold in the sense
of one-sided essential limits, under the weaker local assumptions of Lemma~\ref{lem:mult_kink}: for $\phi=\phi_0+|d_\Gamma|N$
with $\phi_0,N\in H^1$ and $d_\Gamma$ as in Section~\ref{subsec:interface}, $\jump{\p_{\bn}\phi}=2|\grad d_\Gamma|N$
at the points of $\Gamma$ near which $N$ is Lipschitz continuous
and $\phi_0$ is continuously differentiable; for $\phi=\phi_0+\operatorname{ReLU}(g_\eta)N$, $\jump{\grad\phi}=N\grad
g_\eta$, a normal vector, at the regular points of $\Gamma_\eta$ (as in
Proposition~\ref{prop:learned_interface_hdiv}) near which $N$ is Lipschitz continuous and $\phi_0$ is
continuously differentiable. In both cases $\grad\phi+\bw\in H(\curll)$ for every value of the parameters.
\end{prop}

\begin{proof}
$\phi$ has a single-valued trace on $\Gamma$, so $\jump{\grad\phi}=\jump{\p_{\bn}\phi}\bn$, and
$\jump{\bE}=\jump{\grad\phi}$. The modules are treated as in Propositions~\ref{prop:interface_module} and~\ref{prop:learned_interface_hdiv} with the scalar $N$
in place of the skew matrix. Conformity holds because $\phi\in H^1$ and $\curll\grad\phi=0$.
\end{proof}

In two dimensions \eqref{hcurl_representation} has this form with $\bw=\mathcal Sr+\bh_{\curll}\in H(\curll)$,
which has no jump across $\Gamma$ when $r$ is
bounded and piecewise continuous and the condition on
the segments of Lemma~\ref{lem:Rcont} holds (rotated); in the regular decomposition $\bw=\bzz\in H^1$.

If $\bE$ has a nonzero normal jump on a subset of positive $(n-1)$-dimensional measure of a smooth interface, as can occur across a coefficient interface in a curl--curl problem, then $\phi\notin H^{3/2}(\Omega)$
in every decomposition $\bE=\grad\phi+\bw$ with $\bw$ as in Proposition~\ref{prop:hcurl_jump}: $\bw$ has no jump, so $\grad\phi=\bE-\bw$ carries the jump, which excludes $\grad\phi\in H^{1/2}(\Omega)^n$ by the $H^{1/2}$ jump obstruction of Section~\ref{subsec:interface}. The modules separate the singular part from smoother functions.

If $d_\Gamma\in C^{1,1}$ near $\Gamma$, every prescribed normal jump is attained by the fixed module: for a Lipschitz
function $j$ on $\Gamma$, the amplitude $N:=\chi\widetilde j/(2|\grad d_\Gamma|)$, with a
Lipschitz extension $\widetilde j$ of $j$ and a cutoff $\chi$ as in Proposition~\ref{prop:interface_module},
gives $\jump{\p_{\bn}(|d_\Gamma|N)}=j$ on $\Gamma$. The jump depends continuously on the amplitude.

\begin{cor}[approximation of the normal-jump amplitude]\label{cor:hcurl_jump_amplitude}
With $U$, $G$ and $C_{\rm tr}$ as in Corollary~\ref{cor:jump_amplitude}, $\|2|\grad d_\Gamma|(M-M')\|_{L^2(\Gamma)}\le2G\,C_{\rm tr}\|M-M'\|_{H^1(U)}$ for
$M,M'\in H^1(U)$. In particular, if $N$ is the amplitude of a jump $j$ and $N_m$ are Lipschitz networks, then $\|\jump{\p_{\bn}(|d_\Gamma|N_m)}-j\|_{L^2(\Gamma)}\le2G\,C_{\rm tr}\|N-N_m\|_{H^1(U)}$.
\end{cor}

\begin{proof}
Apply the trace inequality to $M-M'$; the identity $\jump{\p_{\bn}(|d_\Gamma|N_m)}=2|\grad d_\Gamma|N_m$ is Proposition~\ref{prop:hcurl_jump} with $\phi_0=0$.
\end{proof}

The module also extracts the whole normal jump at
the level of the field, as for $H(\divvr)$ (Proposition~\ref{prop:regularity_recovery}).

\begin{prop}[regularity recovery]\label{prop:hcurl_recovery}
Let $n=2,3$, and let $\Gamma$, $\Omega_\pm$, $d_\Gamma$ and $U_\Gamma$ be as in Proposition~\ref{prop:regularity_recovery}. Let $\bE\in H(\curll;\Omega)$ with $\bE|_{\Omega_\pm}\in H^1(\Omega_\pm)^n$ and normal jump $j:=\jump{\bE\cdot\bn}$ on $\Gamma$, and let $N\in H^2(\Omega)$ be supported in $U_\Gamma$ with $2|\grad d_\Gamma|N=j$ on $\Gamma$. Then
\[
 \bE=\grad\bigl(|d_\Gamma|N\bigr)+\bzz,\qquad \bzz\in H^1(\Omega)^n .
\]
This is the form of the second representation;
for $n=2$, in the first representation of \eqref{hcurl_representation}, $\bE=\grad(\phi_0+|d_\Gamma|N)+\mathcal
Sr+\bh_{\curll}$ with $\phi_0\in H^1(\Omega)$, $r=\curll\bE$ and $\bh_{\curll}$ the cohomology part of $\bE$. If
$n=3$ and $\curll\bE\in H^1(\Omega)^3$, then $\bE=\grad(\phi_0+|d_\Gamma|N)+\bzz'$ with $\phi_0\in H^2(\Omega)$ and $\bzz'\in H^2(\Omega)^3$.
\end{prop}

\begin{proof}
Since $\bE\in H(\curll;\Omega)$, its tangential trace is single valued, $P_\tau\jump{\bE}=0$, so $\jump{\bE}=j\bn$. On $\Omega_\pm$, $|d_\Gamma|N=\pm d_\Gamma N\in H^2$, so $\grad(|d_\Gamma|N)=N\grad|d_\Gamma|+|d_\Gamma|\grad N$ is in $H^1$ on both sides, with traces $\pm N\grad d_\Gamma$ on $\Gamma$; its jump is $2N\grad d_\Gamma=2|\grad d_\Gamma|N\bn=j\bn$. Hence $\bzz:=\bE-\grad(|d_\Gamma|N)$ is in $H^1$ on both sides and has no jump across $\Gamma$, so $\bzz\in H^1(\Omega)^n$ by the same gluing argument. For $n=2$, \eqref{hcurl_representation} applied to $\bzz$ gives $\bzz=\grad\phi_0+\mathcal Sr+\bh_{\curll}$ with $r=\curll\bzz=\curll\bE$, and $\bh_{\curll}$ has the circulations of $\bzz$, which are those of $\bE$, since the gradient of an $H^1$ function has zero circulations. For $n=3$, $\curll\bzz=\curll\bE\in H^1(\Omega)^3$, and Remark~\ref{rem:hcurl_regular} with $r=1$ gives $\bzz=\grad\phi_0+\bzz'$ with $\phi_0\in H^2(\Omega)$ and $\bzz'\in H^2(\Omega)^3$.
\end{proof}

The scalar kink extracts the entire admissible jump of an $H(\curll)$ field, and the remainder
is an ordinary $H^1$ field; this is the counterpart of the extraction of the
tangential jump of an $H(\divvr)$ field by the kink of
the skew potential. As there, the existence of $N\in H^2$ is a
hypothesis on the jump: it holds if $N_\Gamma:=j/(2|\grad d_\Gamma|)$
on $\Gamma$ has an $H^2$ extension supported in $U_\Gamma$.

\begin{cor}[fixed-interface best approximation]\label{cor:enriched_hcurl}
Let $d_\Gamma\in W^{1,\infty}(\Omega)$, let $\{\cP_m'\}$ be a second family of scalar
classes, and let $C_d$ be as in Corollary~\ref{cor:enriched_hdiv}.
\begin{enumerate}
\item[(a)] For $n=2$, $\bE=\grad(\phi_0+|d_\Gamma|N)+\mathcal Sr+\bh_{\curll}$ with $\phi_0,N\in H^1(\Omega)$, $r\in L^2(\Omega)$, $\bh_{\curll}\in J\cH$, and the class of
the fields $\grad(\phi_{0,m}+|d_\Gamma|N_m)+\mathcal Sr_m+\bh_{\curll}$ with $\phi_{0,m}\in\cP_m$, $N_m\in\cP_m'$, $r_m\in\cQ_m$,
\[
 \inf_{\bE_m}\|\bE-\bE_m\|_{H(\curll)}\le\inf_{\phi_{0,m}\in\cP_m}|\phi_0-\phi_{0,m}|_1+C_d\inf_{N_m\in\cP_m'}\|N-N_m\|_1+\sqrt{1+C_{x_0}^2}\,\inf_{r_m\in\cQ_m}\|r-r_m\|_0 .
\]
\item[(b)] For $n=3$ and $\Gamma_D=\emptyset$, $\bE=\grad(\phi_0+|d_\Gamma|N)+\bzz$ with $\phi_0,N\in H^1(\Omega)$, $\bzz\in H^1(\Omega)^3$, and the class of the fields
$\grad(\phi_{0,m}+|d_\Gamma|N_m)+\bzz_m$ with $\phi_{0,m}\in\cP_m$, $N_m\in\cP_m'$, $\bzz_m\in\cZ_m$,
\[
 \inf_{\bE_m}\|\bE-\bE_m\|_{H(\curll)}\le\inf_{\phi_{0,m}\in\cP_m}|\phi_0-\phi_{0,m}|_1+C_d\inf_{N_m\in\cP_m'}\|N-N_m\|_1+\sqrt2\inf_{\bzz_m\in\cZ_m}\|\bzz-\bzz_m\|_1 .
\]
\end{enumerate}
\end{cor}

\begin{proof}
In both cases $\bE-\bE_m$ is $\grad(\phi_0-\phi_{0,m})+\grad(|d_\Gamma|(N-N_m))$ plus $\mathcal S(r-r_m)$ for $n=2$ and $\bzz-\bzz_m$ for $n=3$. The gradients are curl free, and $|\,|d_\Gamma|M|_1\le C_d\|M\|_1$ as in Corollary~\ref{cor:enriched_hdiv}; $\|\mathcal S(r-r_m)\|_{H(\curll)}\le(1+C_{x_0}^2)^{1/2}\|r-r_m\|_0$ as in the proof of Corollary~\ref{cor:hcurl_density}, and $\|\bv\|_{H(\curll)}\le\sqrt2\|\bv\|_1$ as in the proof of Proposition~\ref{prop:hcurl3d_density}.
\end{proof}

With $\Gamma_D=\partial\Omega$ the module is placed inside the boundary-conforming parameterization of the scalar potential, as described at the end of this subsection. For $n=3$, if $\curll\bE\in H^1$ and $N\in H^2$, Proposition~\ref{prop:hcurl_recovery} gives
$\phi_0$, $N$ and $\bzz'$ in $H^2$, and the nonsmooth graph-space field is reduced to approximation problems for smooth functions.

\begin{rem}[the kink is intrinsic]\label{rem:kink_intrinsic}
The regular decomposition is not unique, but its singular part
is. If $\bE=\grad\phi_1+\bzz_1=\grad\phi_2+\bzz_2$ with $\bzz_1,\bzz_2\in H^1(\Omega)^n$, then $\grad(\phi_1-\phi_2)=\bzz_2-\bzz_1\in H^1(\Omega)^n$, so $\phi_1-\phi_2\in
H^2(\Omega)$; where one-sided traces exist, $\jump{\p_{\bn}\phi_1}=\jump{\p_{\bn}\phi_2}=\jump{\bE\cdot\bn}$. The kink is
determined by the field, not by the decomposition operator.
\end{rem}

The three implementation levels of Section~\ref{subsubsec:hdiv_levels} apply to the
scalar potential: $\phi_\theta=\phi_{0,\theta}+|d_\Gamma|N_\theta$ with the interface known, $\phi_\theta=\phi_{0,\theta}+\operatorname{ReLU}(g_\eta)N_\theta$
with the interface learned, or a generic ReLU
network $\phi_\theta$, whose normal jumps lie on its
kink sets (Lemmas~\ref{lem:hcurl_cpwl} and~\ref{lem:hcurl_intrinsic}). With the essential condition
on $\Gamma_D=\partial\Omega$ the module is placed inside the
boundary-conforming parameterization of the scalar potential, for instance $\phi_\theta=B(\phi_{0,\theta}+|d_\Gamma|N_\theta)$ with
a boundary factor $B\in W^{1,\infty}(\Omega)$ that vanishes on
$\Gamma_D$ and is positive in $\Omega$ (Section~\ref{subsec:curl_numerics}), so
that $\phi_\theta\in H^1_{\Gamma_D}(\Omega)$ for every parameter value; the
kink is then $2B|\grad d_\Gamma|N_\theta$, and likewise with $\operatorname{ReLU}(g_\eta)$.

\subsection{Relation to finite elements and componentwise networks}\label{subsec:componentwise}

\begin{rem}[relation to N\'ed\'elec elements]\label{rem:nedelec}
In two dimensions $K^{-1}=J$ maps the Raviart--Thomas spaces
onto the N\'ed\'elec spaces of the first kind and curls
to gradients; Remark~\ref{rem:RT0} therefore gives, on a triangle
$T$ with a local center $x_T$ and $\mathcal S_T:=J\Rop_{2,T}$,
\[
 N_k(T)=J\,RT_k(T)=\grad\mathbb P_{k+1}\oplus\mathcal S_T\mathbb P_k ,
\]
so $\grad\phi+\mathcal S_Tr$ with $\phi\in\mathbb P_{k+1}$ and $r\in\mathbb P_k$ is exactly $N_k(T)$, and the class $\cE_m^{\rm dR}$ of \eqref{X_hcurl_2d} replaces the polynomials $\phi$ and $r$ by networks. On a conforming triangulation $\mathcal T_h$ of a simply connected
$\Omega$, the curl-free fields of the lowest-order N\'ed\'elec space are $\grad S_1(\mathcal T_h)$, and every function in $S_1(\mathcal T_h)$ is a ReLU network \cite{HLXZ:20}; this is the rotation of Remark~\ref{rem:RT0_2d}. In three dimensions
$N_k(T)=\mathbb P_k^3\oplus(x-x_T)\times\widetilde{\mathbb P}_k^3=\grad\mathbb P_{k+1}\oplus(x-x_T)\times\mathbb P_k^3$ \cite{AFW:2006}, where the second part is the Poincar\'e operator applied to the curl. The class \eqref{structured3d} replaces it by the $H^1$ field $\bzz$, since the curl cannot be prescribed freely (Section~\ref{subsec:hcurl_derham}), and no correspondence with N\'ed\'elec spaces is claimed.
\end{rem}

Since $H^1(\Omega)^n\subset H(\curll;\Omega)$, a componentwise $H^1$ network is automatically $H(\curll)$-conforming. It imposes more continuity than
$H(\curll)$ requires, however: an $H^1$ field has single-valued traces of all components
across interior interfaces (and a finite network with a continuous activation is componentwise
continuous), whereas the graph space requires only a single-valued tangential trace; so
no realization has a normal jump across an interior interface, and a
normal jump of the target can only be approximated by a continuous transition
layer. The structured class $\grad\phi_\theta+\bzz_\theta$ does not impose this continuity: with a
ReLU potential, $\grad\phi_\theta$ has a normal jump already at finite width while keeping the correct tangential
trace. Componentwise continuous networks are a continuous-space analogue of vector Lagrange elements $(C^0)^n$, and the structured
classes one of N\'ed\'elec elements, exactly so on a triangle in two dimensions (Remark~\ref{rem:nedelec}).

\section{The FOSLS-deRhaNN method}\label{sec:framework}

Given a stable FOSLS formulation in the natural spaces of the
weak problem, the deRhaNN classes above provide conforming
neural trial classes. We recall the resulting standard best-approximation estimate.

\subsection{FOSLS-deRhaNN approximation}\label{subsec:fosls_approx}

Let $X=X_1\times\cdots\times X_J$ be a product of the natural spaces of a problem
($H^1$, $H(\divvr)$, $H(\curll)$, $L^2$, with homogeneous essential boundary conditions
where required), $Y$ a product of $L^2$ spaces
(and of a trace space such as $H^{1/2}(\partial\Omega)$
when a boundary condition is imposed through the
functional, Section~\ref{subsec:weak_dirichlet}), $L:X\to Y$ bounded and linear, and $F\in Y$.
The first-order system is $Lz=F$, and its least-squares functional is
\[
 \cJ(z;F):=\|Lz-F\|_Y^2 .
\]

\begin{assum}[FOSLS stability]\label{assum:fosls}
There are $0<c_L\le C_L$ with $c_L\|z\|_X^2\le\|Lz\|_Y^2\le C_L\|z\|_X^2$ for
all $z\in X$, and $Lz^*=F$ has a solution $z^*\in X$.
\end{assum}

Then $z^*$ is the unique minimizer of $\cJ(\cdot;F)$ over $X$, $\cJ(z^*;F)=0$, and $\cJ(z;F)=\|L(z-z^*)\|_Y^2$ for every $z\in X$. For the second-order elliptic equation, Assumption~\ref{assum:fosls} holds under unique solvability alone, for right-hand sides in $L^2$ and in $H^{-1}$ \cite{Zhang:23}.

\begin{defn}[FOSLS-deRhaNN method]\label{def:nativefosls}
For $j=1,\dots,J$ let $\{\cX_{j,m}\}_m$ be nested classes for $X_j$: ordinary networks for scalar unknowns in $H^1$ or $L^2$, and the native classes of Sections~\ref{sec:hdiv}--\ref{sec:hcurl} for vector unknowns (for $H(\divvr)$ either the class of Section~\ref{sec:hdiv} or that of Section~\ref{sec:hdiv_reg}, for $H(\curll)$ in two dimensions either class of Section~\ref{subsec:hcurl2d}). Every realization lies in $X_j$ for every parameter value. Let $\cX_m:=\cX_{1,m}\times\cdots\times\cX_{J,m}$. A FOSLS-deRhaNN approximation is any $z_m\in\cX_m$ with $\cJ(z_m;F)\le\inf_{\cX_m}\cJ(\cdot;F)+\varepsilon_m^2$, where $\varepsilon_m\ge0$ is
a suboptimality defect (from nonlinear optimization, stopping, regularized linear algebra or another numerical effect, which are not distinguished here). We call the resulting method \emph{FOSLS-deRhaNN}, first-order system least squares with de Rham neural networks, in which the graph-space unknowns are represented by the deRhaNN classes developed above and the loss is the FOSLS functional posed in the natural spaces of the weak formulation.
\end{defn}

\begin{prop}[best-approximation estimate]\label{prop:best_approx}
Under Assumption~\ref{assum:fosls}, $\cX_m\subset X$ for every value of the trainable parameters, and
\[
 c_L\|z^*-z_m\|_X^2\le\cJ(z_m;F)\le C_L\inf_{x\in\cX_m}\|z^*-x\|_X^2+\varepsilon_m^2 .
\]
\end{prop}

\begin{proof}
For every $x\in\cX_m$, $\cJ(z_m;F)\le\cJ(x;F)+\varepsilon_m^2=\|L(x-z^*)\|_Y^2+\varepsilon_m^2\le C_L\|x-z^*\|_X^2+\varepsilon_m^2$, and $\cJ(z_m;F)=\|L(z_m-z^*)\|_Y^2\ge c_L\|z_m-z^*\|_X^2$.
\end{proof}

Hence the best-approximation results for the native $H(\divvr)$ and $H(\curll)$ classes
apply directly to the corresponding variables of any stable FOSLS formulation; in particular
$z_m\to z^*$ if $\varepsilon_m\to0$ and the classes have dense unions,
which holds when the scalar network classes for the representing functions are dense in $H^1$
(with their boundary conditions) and in $L^2$ (Corollary~\ref{cor:hdiv_density} and Section~\ref{subsec:hcurl3d}).

\subsection{Remarks on the trial classes}\label{subsec:fosls_remarks}

\begin{rem}[shared hidden layers]\label{rem:shared}
The approximation estimates of Sections~\ref{sec:hdiv}--\ref{sec:hcurl} minimize over each representing function independently, that is, over a product of scalar classes. In the computations the representing functions are outputs of
one network with shared hidden layers, which at a fixed width is in general a proper subset of that product. Conversely, $s$ networks of the
same depth with widths $m_1,\ldots,m_s$ form one network with $s$ outputs and width $m_1+\cdots+m_s$: the first hidden layers are stacked, the deeper ones are block diagonal, and the output layer selects
the blocks. The estimates over the product therefore hold for the shared architecture with the widths added, not at the same width.
\end{rem}

\begin{rem}[variants]\label{rem:variants}
\emph{Boundary conditions.} Inhomogeneous essential conditions are imposed exactly in the classes (affine sets) or weakly through
the functional (Section~\ref{subsec:weak_dirichlet}). \emph{Built-in equations.} If $L=(L_1,L_2)$ and the trial set is restricted to $\{L_2z=F_2\}$, which $z^*$ satisfies,
then $L_2(z-z^*)=0$, so $\cJ=\|L_1z-F_1\|^2$ on this set and the stability of Assumption~\ref{assum:fosls} restricts to it; Proposition~\ref{prop:best_approx}
holds unchanged. Any $z_{2,p}$ with $L_2z_{2,p}=F_2$ gives the affine representation $z_{2,p}+\ker L_2$, with native classes
for $\ker L_2$; a bounded right inverse $R_2$ of $L_2$ is a
convenient stable construction of such a lift, $z_{2,p}=R_2F_2$, and is not needed for the stability. \emph{Density.}
Conformity alone does not suffice: if the closure of the union of the classes is
a proper subspace of $X$, the best approximation error stays bounded away from zero. With Maxwell-type
boundary conditions on nonsmooth domains, an $H^1$-conforming single-field class may fail to be
dense in the corresponding intersection graph space \cite{Cos:91}; this is avoided by the first-order formulations used here.
\end{rem}

\begin{rem}[one variable per continuity space]\label{rem:uniform}
As in least-squares finite element methods, the first-order system is chosen so that
every variable lies in one continuity space and each continuity space has its own
variable: $u\in H^1$ and $\bsigma\in H(\divvr)$ for second-order equations, $\bE\in H(\curll)$ and $\bD\in
H(\divvr)$ for div--curl systems. The native classes are taken for each variable separately.
What remains problem-specific is the equivalence of Assumption~\ref{assum:fosls} for the chosen
system, as for least-squares finite elements. An exterior-derivative
equation $d\omega=g$ can be built into the trial class with
$g$ an unconstrained trainable function only in top degree: in intermediate degree right
inverses exist, but $g$ must satisfy $dg=0$, which an arbitrary network does not satisfy automatically.
\end{rem}

\begin{rem}[a divergence equation in the system: two levels]\label{rem:div_levels}
Suppose the first-order system contains an equation $\divv\bsigma+B(z')=f$ for the divergence of its $H(\divvr)$ unknown, with $B$ acting on
the other unknowns $z'$; in Section~\ref{sec:method} $B(v)=\bb\cdot\grad v+cv$, in Section~\ref{sec:transport} $B(v)=\gamma v$. \emph{(i) Simplification.} In the
first class $\btau=\operatorname{Div}A+\Rop_nq+\bh$ the divergence is one of the representing functions, $\divv\btau=q$ exactly, so the divergence
residual is $q+B(z')-f$: it contains $q$ by its values and no derivative of the potential, which
enters only the other residuals. In the second class the divergence residual is $\divv\bzz+B(z')-f$: it contains the first derivatives of
$\bzz$ and, again, no derivative of the potential. This is used in the formulations of Sections~\ref{sec:method},
\ref{sec:transport} and~\ref{sec:conslaw}, which contain a divergence equation, and in their numerical examples, except the generic level of Section~\ref{subsec:tr_numerics} with
the first class. The curl--curl system of Section~\ref{sec:curlcurl} has no divergence equation. \emph{(ii) Elimination.} With the first
class (Section~\ref{sec:hdiv}), for some problems the divergence $q$ can be eliminated by $q:=f-B(z')$, which builds the divergence equation into the
class with the lift $\Rop_n(f-B(z'))$ (the built-in equations of Remark~\ref{rem:variants}); the divergence residual then vanishes identically. This removes one network and one
residual, at the price of tying $q$ to $z'$ (for the elliptic equation through $\grad v$, Section~\ref{subsec:balanced}). It is
an option, not the main formulation; it is used for the generic level of transport with the first class (Section~\ref{subsec:tr_numerics}).
\end{rem}

Sections~\ref{sec:method} and~\ref{sec:curlcurl} are instances of Proposition~\ref{prop:best_approx}; Section~\ref{sec:transport} uses an error identity in a norm that is not equivalent to that of $X$.

\subsection{Computation}\label{subsec:computation}

The following describes the implementation used in the numerical examples. All theoretical statements
refer to the exactly integrated least-squares functional; in the computations the integrals
are evaluated numerically, and no convergence analysis of the quadrature or optimization procedures is
attempted in this paper. The computed approximation does depend on how the functional is integrated: for free-knot linear splines,
exact integration of the $L^2$ fit leads to Gibbs oscillations at discontinuities \cite{Zhang:22gibbs};
adaptive quadrature for neural network functionals is used, for instance, in \cite{LCR:23}.

\emph{Output coefficients and nonlinear parameters.} For the linear problems considered below, once the hidden and interface parameters are fixed, the output coefficients enter the residuals
linearly and are computed by a linear least-squares (Galerkin) solve; the small regularization added there is a numerical stabilization of the linear algebra and
not part of the functional. The representation maps have kernels that the functional does not see: $A\mapsto A+K$ with $K\in\ker\operatorname{Div}$ (for instance a
constant skew matrix; in two dimensions $\psi\mapsto\psi+C$), $\phi\mapsto\phi+C$ in the regular decomposition when constants are admissible, and more generally the nonuniqueness of the regular decomposition. The
analysis concerns the fields, that is, the quotient by these kernels, whereas the linear least-squares problem for the output coefficients is singular in the corresponding directions when the outputs
have biases; the regularization selects a solution there, and its effect on the functional is part of the suboptimality defect $\varepsilon_m$. The remaining nonlinear parameters, including
those of the level-set network, are trained on the resulting reduced functional by Levenberg--Marquardt, restarted when it stalls, or by L-BFGS where noted; for the nonlinear conservation law the treatment is stated in Section~\ref{subsec:cl_numerics}.

\emph{Integration.} We use two implementations for derivative residuals: direct integration on a partition cut along explicitly tractable kink sets (a given or learned level set, or the zero sets
of the units of a network), and a cellwise moment treatment when resolving the activation partition explicitly is inconvenient. The choice used in each numerical experiment is stated below. In the moment treatment, every residual $R=D+V$
that contains a derivative of the network, with $D$ its derivative part and $V$ its value part, enters training through its moments against $\mathbb P_1$
on the cells of a training partition, together with the part of $V$ orthogonal to $\mathbb P_1$, integrated pointwise; residuals
that contain only values are integrated pointwise. The moments of $D$ are computed from values of the network by integration by parts on each cell, as in the discrete divergence operator of \cite{CCL:23_divergence},
whenever this can be done without differentiating a nonsmooth coefficient, the cells being split along coefficient interfaces where necessary. On a cell $K$ the training functional $\cJ_h$ thus contains $\|\Pi_KR\|_K^2+\|(I-\Pi_K)V\|_K^2$ instead of
$\|R\|_K^2$; the first term is at most $\|R\|_K^2$, but the difference $\|(I-\Pi_K)D\|_K^2+2((I-\Pi_K)D,(I-\Pi_K)V)_K$ has no sign in general, so $\cJ_h$ is not a lower bound for $\cJ$. The computed iterate
is checked by $\cJ_{\rm check}$, the functional integrated on a finer rule, cut along the physical interface and, where feasible, piecewise linearly on subcells along the zero sets of
the ReLU units of the network; the rule used in each experiment is stated below. It is an independent finer evaluation of the functional, not a bound. Training cells on
which $\cJ_{\rm check}-\cJ_h$ is large are refined, and units are added to the ReLU layer where $\cJ_{\rm check}$ is large, by the adaptive network enhancement of \cite{LCC:22ane1,LC:22ane2,CL:24adaptiveNN}. The separation of output coefficients and hidden parameters follows \cite{CDLLX:25sggn,Cai_shallowRitz1D:24}.

\emph{Evaluation of the $H(\divvr)$ representations.} The fields of both classes are evaluated
term by term. In the potential part, common to both, the entry $a_{ij}$ contributes $(\p_ja_{ij})\be_i-(\p_ia_{ij})\be_j$ to $\operatorname{Div}A$, so the cost grows with the
dimension through the $n(n-1)/2$ entries, each of which requires only the two partial derivatives $\p_ia_{ij}$ and $\p_ja_{ij}$. In the
first class the radial part $\Rop_nq_m$ is evaluated by its one-dimensional integral. Since $q$ is extended by
zero in \eqref{Rd}, the integral runs over the parts of $I_x:=\{t\in(0,1):x_0+t(x-x_0)\in\Omega\}$ (the segment is clipped to $\Omega$;
$I_x=(0,1)$ if $\Omega$ is starlike with respect to $x_0$). For the computational geometries considered below these intersections have finitely many connected components,
which are enumerated explicitly (for a general Lipschitz domain a line may
meet $\Omega$ in infinitely many components, which does not affect the definition of $\Rop_n$), and on each
part $(t_a,t_b)$ an $N_t$-point Gauss rule with nodes $t_\ell$ and weights $\omega_\ell$ on $(t_a,t_b)$ gives the contribution
\beq\label{Rd_quadrature}
 (x-x_0)\textstyle\sum_{\ell=1}^{N_t}\omega_\ell\,t_\ell^{n-1}q_m(x_0+t_\ell(x-x_0)),
\eeq
so no derivative of $q_m$ is needed. The
accuracy is that of the Gauss rule for
$t\mapsto t^{n-1}q_m(x_0+t(x-x_0))$ on each part: it is high
for a smooth network, while kinks of $q_m$
along the segment (ReLU layers) or a piecewise
$q_m$ across a known interface limit it unless
the rule is split there as well. If the
network $q_m$ is evaluated outside $\Omega$ instead of
being extended by zero, the right-inverse identity still
holds in $\Omega$, while the bound \eqref{Rd_bound}, and
with it \eqref{hdiv_d_approx}, involves $\|\widetilde q-q_m\|_{L^2(\Omega_{x_0}^*)}$, with $\widetilde
q$ the zero extension of $q$ and $\Omega_{x_0}^*:=\{x_0+t(x-x_0):x\in\Omega,\
0\le t\le1\}$ the union of the segments from
$x_0$ to $\Omega$, instead of $\|q-q_m\|_{0,\Omega}$. The quadrature
is an evaluation of the exact operator $\Rop_n$: conformity
and $\divv\Rop_nq_m=q_m$ are properties of $\Rop_n$, and the
computed field differs from $\Rop_nq_m$ by the quadrature
error; a Gauss sum regarded as an operator
in its own right is not an exact
right inverse of the divergence. In the second
class there is no radial part: $\bzz_m$ is
evaluated pointwise, and $\divv\bzz_m$ from the first derivatives
of its components, like the entries of $A_m$;
conformity and $\divv\btau_m=\divv\bzz_m$ hold for the computed field exactly.

\section{Second-order elliptic equations}\label{sec:method}

\subsection{Problem and first-order system}\label{sec:problem}

\begin{assum}\label{assum:main}
\begin{itemize}
\item[(i)] $\Omega\subset\R^2$ is a bounded Lipschitz polygon.
$\Omega$ is partitioned into polygonal subdomains $\{\Omega_i\}_{i=1}^m$ (the coefficient subdomains).
\item[(ii)] $\mathbf A\in L^\infty(\Omega)^{2\times2}$ is symmetric with $\lambda_0|\xi|^2\le\xi^T\mathbf A\xi\le\lambda_1|\xi|^2$; $\bb,\bbet\in L^\infty(\Omega)^2$; $c\in L^\infty(\Omega)$.
No derivatives of the coefficients are used.
\item[(iii)] The data are those of the weak formulation: $f_1\in L^2(\Omega)$, $\bff_2\in L^2(\Omega)^2$,
and $g_D\in H^{1/2}(\partial\Omega)$.
\item[(iv)] The bilinear form
$a(v,w):=(\mathbf A\grad v-\bbet v,\grad w)+(\bb\cdot\grad v+cv,w)$ induces an isomorphism $H^1_0(\Omega)\to H^{-1}(\Omega)$. Equivalently, the homogeneous
problem has only the trivial solution: $a$ satisfies a G\aa rding inequality, and the lower-order terms are compact from $H^1_0(\Omega)$
to $H^{-1}(\Omega)$ ($v\mapsto\bb\cdot\grad v+cv$ is bounded into $L^2(\Omega)$, which embeds compactly into $H^{-1}(\Omega)$, and
$v\mapsto\bbet v$ is compact into $L^2(\Omega)^2$ by Rellich's theorem), so the operator is Fredholm of index zero.
\end{itemize}
\end{assum}

The problem is posed at the level of the weak formulation: find $u\in H^1(\Omega)$ with $u=g_D$ on $\partial\Omega$ and
\beq\label{pde}
 a(u,w)=(f_1,w)-(\bff_2,\grad w)\qquad\forall\,w\in H^1_0(\Omega),
\eeq
formally $-\divv(\mathbf A\grad u)+\bb\cdot\grad u+\divv(\bbet u)+cu=f_1+\divv\bff_2$ in $\Omega$. Every $F\in
H^{-1}(\Omega)$ is of the form $w\mapsto(f_1,w)-(\bff_2,\grad w)$: take $z\in H^1_0(\Omega)$
with $(z,w)+(\grad z,\grad w)=F(w)$ for all $w$, and $f_1=z$,
$\bff_2=-\grad z$. Hence Assumption~\ref{assum:main}(iii) is the general data
class for which the weak problem is well posed, and
$u$ has only the regularity of the weak solution.
Examples for which no strong form exists: $\bff_2$ discontinuous,
for instance the gradient of a continuous piecewise
linear function \cite{Zhang:23}; and sources on curves, for instance,
for the line $\{x_1=0\}$, $F(w)=\int_{\{x_1=0\}}g\,w\,ds$ is of this
form with $f_1=0$ and $\bff_2=\chi_{\{x_1>0\}}\,g(x_2)\,\be_1$. Point sources are
not in $H^{-1}$ for $n\ge2$ and are not covered.

With the flux $\bsigma:=-\mathbf A\grad u+\bbet u-\bff_2$, testing \eqref{pde} with $w\in C_c^\infty(\Omega)$
gives $\divv\bsigma+\bb\cdot\grad u+cu=f_1\in L^2(\Omega)$, so $\bsigma\in H(\divvr;\Omega)$. The first-order system is
\beq\label{fos}
\mathbf A^{-1/2}(\bsigma-\bbet u+\bff_2)+\mathbf A^{1/2}\grad u=0,\qquad \divv\bsigma+\bb\cdot\grad u+cu=f_1 .
\eeq
The $H(\divvr)$ unknown is $\bsigma$; the physical flux $-\mathbf
A\grad u+\bbet u=\bsigma+\bff_2$ is in general not in $H(\divvr;\Omega)$
when $\bff_2$ jumps \cite{Zhang:23}. For $F:=(f_1,\bff_2)$ and $(v,\btau)\in H^1(\Omega)\times\Hdiv$ define
\beq\label{LS}
\LS(v,\btau;F):=\|\mathbf A^{-1/2}(\btau-\bbet v+\bff_2)+\mathbf A^{1/2}\grad v\|_0^2+\|\divv\btau+\bb\cdot\grad v+cv-f_1\|_0^2,\qquad \tri(v,\btau)\tri^2:=\LS(v,\btau;0).
\eeq
Let $H^1_g(\Omega):=\{v\in H^1(\Omega):v=g_D\mbox{ on }\partial\Omega\}$. Since \eqref{fos} is linear, $\LS(v,\btau;F)=\tri(u-v,\bsigma-\btau)\tri^2$ for all $(v,\btau)$.

\begin{lem}[norm equivalence \cite{Zhang:23}; cf.\ \cite{CLMM:94,BG:09}]\label{lem:equiv}
Under Assumption~\ref{assum:main} there are constants $0<C_\ell\le C_u$ such that $$ C_\ell\big(\|v\|_1^2+\hdn{\btau}^2\big)\le\tri(v,\btau)\tri^2\le C_u\big(\|v\|_1^2+\hdn{\btau}^2\big)\qquad\forall\,(v,\btau)\in H^1_0(\Omega)\times H(\divvr;\Omega). $$
\end{lem}

\begin{rem}[data at the level of the weak formulation]\label{rem:Hminus1}
Lemma~\ref{lem:equiv} concerns the homogeneous operator only, so the error identity and the norm equivalence hold for all data in Assumption~\ref{assum:main}(iii), with no regularity beyond
that of the weak formulation \cite{Zhang:23}. In the native classes of Section~\ref{sec:hdiv} only $f_1$ enters the right inverse (through $q=f_1-\bb\cdot\grad
v-cv$ in the balanced formulation below), in those of Section~\ref{sec:hdiv_reg} it enters only the divergence residual, and in both $\bff_2$ enters only the constitutive residual, pointwise
and never differentiated. The representation $F=f_1+\divv\bff_2$ of a datum in $H^{-1}(\Omega)$ is not unique: the weak solution $u$ does not depend on it, but the flux $\bsigma=-\mathbf A\grad u+\bbet
u-\bff_2$ does, and with it the regularity, the interface structure and the best approximation of the flux. The representation
is therefore part of the first-order formulation; when the data come with a natural one, such as a source on an interface
or a curve, or a given physical flux, that one is used. A strong-form residual $-\divv(\mathbf A\grad v)+\cdots-f_1-\divv\bff_2$ is not defined
for such data. The Dirichlet data are imposed either in the trial set, $v\in H^1_g(\Omega)$ (Theorem~\ref{thm:qo}), or weakly through the functional (Section~\ref{subsec:weak_dirichlet}).
\end{rem}

The method is Proposition~\ref{prop:best_approx} with $X=H^1_0(\Omega)\times H(\divvr;\Omega)$ (shifted by a lifting of $g_D$), or with $X=H^1(\Omega)\times H(\divvr;\Omega)$ when the Dirichlet condition is imposed weakly (Section~\ref{subsec:weak_dirichlet}), $L$ given by \eqref{fos}, $c_L=C_\ell$ and $C_L=C_u$ from Lemma~\ref{lem:equiv}, and the native classes of Sections~\ref{subsec:hdiv23} and~\ref{subsec:hdiv_reg_classes}. Assumption~\ref{assum:main}(i) is stated in two dimensions; Lemma~\ref{lem:equiv} holds, with the same proof \cite{Zhang:23}, for a bounded Lipschitz domain $\Omega\subset\R^n$, and the native classes are then those of Sections~\ref{subsec:hdiv_classes} and~\ref{subsec:hdiv_reg_classes}.

\subsection{Structured-flux formulation}\label{subsec:unconstrained}

Let $\cU\subset H^1(\Omega)$ be a scalar neural network class, possibly augmented
by the interface modules of Sections~\ref{subsec:interface} and~\ref{subsec:learned_interface}. Let
\[
 \cQ:=\left\{\sum_{i=1}^m\one_{\Omega_i}q_i:\ q_i\in\cQ_i\right\}\subset L^2(\Omega),
\]
where each $\cQ_i$ is a scalar network class on the coefficient subdomain $\Omega_i$; no continuity is imposed across coefficient interfaces. Define
\beq\label{Sigma}
\Sigma_\theta:=\{\curl\psi+\Rop_2 q+\bh:\ \psi\in\cU,\ q\in\cQ,\ \bh\in\cH\}\quad\text{or}\quad\Sigma_\theta:=\{\curl\psi+\bzz:\ \psi\in\cU,\ \bzz\in\cZ\},
\eeq
with $\cZ\subset H^1(\Omega)^2$ a class of vector networks, with
the modules of Section~\ref{subsec:div_jumps} if the divergence jumps
across a coefficient interface. By Corollary~\ref{cor:hdiv23} and Theorem~\ref{thm:hdivd_reg}, every element
of $\Sigma_\theta$ is $H(\divvr)$-conforming for every parameter value, with $\divv(\curl\psi+\Rop_2
q+\bh)=q$, respectively $\divv(\curl\psi+\bzz)=\divv\bzz$. Hence, for the first class,
\beq\label{loss}
\begin{split}
\LS(v,\curl\psi+\Rop_2 q+\bh;F)
={}&\|\mathbf A^{-1/2}(\curl\psi+\Rop_2 q+\bh-\bbet v+\bff_2)+\mathbf A^{1/2}\grad v\|_0^2\\
&+\|q+\bb\cdot\grad v+cv-f_1\|_0^2 ,
\end{split}
\eeq
and, for the second, the same with $\bzz$ in place of $\Rop_2q+\bh$ and $\divv\bzz$ in
place of $q$. Only $\grad v$, $\grad\psi$ and, in the second class, the
first derivatives of $\bzz$ are differentiated; neither $q$ nor the discontinuous coefficients are differentiated.

Given an admissible set $\cU_g\subset\cU$, an $\varepsilon$-minimizer satisfies
\beq\label{method}
\LS(u_\theta,\bsigma_\theta;F)\le
\inf_{(v,\btau)\in\cU_g\times\Sigma_\theta}
\LS(v,\btau;F)+\varepsilon^2 .
\eeq
The infimum need not be attained; this is the
correct statement for a nonclosed neural approximation set.

\begin{thm}[exact boundary conditions]\label{thm:qo}
Suppose $\cU_g\subset H^1_g(\Omega)$. Then
\[
\LS(u_\theta,\bsigma_\theta;F)
=\tri(u-u_\theta,\bsigma-\bsigma_\theta)\tri^2
\le
\inf_{(v,\btau)\in\cU_g\times\Sigma_\theta}
\tri(u-v,\bsigma-\btau)\tri^2+\varepsilon^2,
\]
and consequently
\[
\|u-u_\theta\|_1^2+\hdn{\bsigma-\bsigma_\theta}^2
\le\frac{C_u}{C_\ell}
\inf_{(v,\btau)\in\cU_g\times\Sigma_\theta}
\big(\|u-v\|_1^2+\hdn{\bsigma-\btau}^2\big)
+\frac{\varepsilon^2}{C_\ell}.
\]
\end{thm}

\begin{proof}
Proposition~\ref{prop:best_approx} with $c_L=C_\ell$, $C_L=C_u$.
\end{proof}

\begin{cor}[reduction to scalar approximation]\label{cor:red}
With $(\psi^*,q^*,\bh^*)$ from Corollary~\ref{cor:hdiv23},
\[
\hdn{\bsigma-(\curl\psi+\Rop_2 q+\bh)}
\le |\psi^*-\psi|_1+\sqrt{1+C_{x_0}^2}\,\|q^*-q\|_0+\|\bh^*-\bh\|_0 .
\]
Since the full harmonic space $\cH$ is in the class, $\bh=\bh^*$ is admissible and the last term contributes no approximation error,
only finitely many coefficients. Since $q\in\cQ$ is piecewise, jumps of $q^*$ across coefficient
interfaces carry no approximation penalty beyond the separate subdomain $L^2$ errors. With $(\psi^*,\bzz^*)$ from Proposition~\ref{prop:hdiv_regular} with $n=2$,
\[
\hdn{\bsigma-(\curl\psi+\bzz)}\le|\psi^*-\psi|_1+\sqrt2\,\|\bzz^*-\bzz\|_1 ,
\]
and jumps of $q^*$ are carried by kinks of $\bzz$ (Lemma~\ref{lem:div_kink}).
\end{cor}

\begin{rem}[why $q$ is kept]\label{rem:keepq}
In the first class the divergence $q$ is an
independent $L^2$ function in the representation of the flux and is
approximated directly: its target $q^*=f_1-\bb\cdot\grad u-cu$ may jump across
coefficient interfaces, and the piecewise classes $\cQ$ impose no continuity
there. The two residuals in \eqref{loss} stay separate, so
the equilibrium residual $\|q+\bb\cdot\grad v+cv-f_1\|_0$ is available as its
own indicator, and no cancellation between $\Rop_2 f_1$
and $\Rop_2(\bb\cdot\grad v+cv)$ occurs. Eliminating $q$, as in
the balanced formulation, removes one network and one residual, at
the price of that cancellation and of tying the flux
error to $\|u-v\|_1$. In the second class the equilibrium
residual $\|\divv\bzz+\bb\cdot\grad v+cv-f_1\|_0$ stays separate in the same way.
\end{rem}

\subsection{Weakly imposed Dirichlet condition}\label{subsec:weak_dirichlet}

A network class does not satisfy $v=g_D$ on $\partial\Omega$ in general, and building the condition into the class requires a
lifting of $g_D$ and a factor vanishing on $\partial\Omega$. Instead, the Dirichlet residual is added to the functional: for $w_D>0$ and $(v,\btau)\in H^1(\Omega)\times H(\divvr;\Omega)$,
\beq\label{LS_D}
 \LS_D(v,\btau;F,g_D):=\LS(v,\btau;F)+w_D\|v-g_D\|_{1/2,\partial\Omega}^2,\qquad
 \|e\|_{1/2,\partial\Omega}^2:=\|e\|_{0,\partial\Omega}^2+\int_{\partial\Omega}\!\int_{\partial\Omega}\frac{|e(s)-e(t)|^2}{|s-t|^{n}}\,ds\,dt .
\eeq
The $H^{1/2}$ boundary residual is the natural form covered by the continuous stability theory:
with a bounded right inverse $E:H^{1/2}(\partial\Omega)\to H^1(\Omega)$ of the trace $\gamma$, write $v=v_0+E(\gamma v)$ with $v_0\in H^1_0(\Omega)$;
Lemma~\ref{lem:equiv} applied to $(v_0,\btau)$, together with $\tri(E(\gamma v),0)\tri\le C_u^{1/2}\|E(\gamma v)\|_1\le C_u^{1/2}C_E\|\gamma v\|_{1/2,\partial\Omega}$, gives $\|v\|_1^2+\hdn{\btau}^2\lesssim\tri(v,\btau)\tri^2+\|\gamma v\|_{1/2,\partial\Omega}^2$. Hence Proposition~\ref{prop:best_approx} applies with
$X=H^1(\Omega)\times H(\divvr;\Omega)$ and classes without boundary conditions. The Dirichlet residual may also
be measured in $L^2(\partial\Omega)$ with a large weight, $w_b\|v-g_D\|_{0,\partial\Omega}^2$; for least-squares functionals with weakly imposed boundary conditions, Dirichlet
and Neumann, see \cite{Zhang:weakBC}. The double integral in \eqref{LS_D} is computed by a product quadrature rule on $\partial\Omega$.

\subsection{Built-in variant: the PDE-balanced formulation}\label{subsec:balanced}

This is the second level of Remark~\ref{rem:div_levels}, with the first class: the
independent $q$-network can be eliminated. For any $v\in H^1(\Omega)$ set
\beq\label{qv}
 q[v]:=f_1-\bb\cdot\grad v-cv\in L^2(\Omega)
\eeq
and define
\beq\label{tau_bal}
 \btau[v,\psi,\bh]:=\curl\psi+\Rop_2 q[v]+\bh .
\eeq
Then, for every parameter value,
\beq\label{balance_identity}
 \divv\btau[v,\psi,\bh]+\bb\cdot\grad v+cv=f_1
\eeq
holds identically. Thus one of the two first-order equations is built
into the architecture and the least-squares loss reduces to the constitutive residual
\beq\label{LS_bal}
\LS_{\rm bal}(v,\psi,\bh;F)
:=\|\mathbf A^{-1/2}(\btau[v,\psi,\bh]-\bbet v+\bff_2)+\mathbf A^{1/2}\grad v\|_0^2.
\eeq
In other words, the balanced method is the full least-squares problem restricted to the exact equilibrium manifold $\{\divv\btau+\bb\cdot\grad v+cv=f_1\}$: one residual vanishes identically, and the remaining constitutive residual is still equivalent to the complete natural error. Only two scalar neural fields, $v$ and $\psi$, are required (plus the finitely many coefficients of $\bh$).

\begin{rem}[equations built into the trial class]\label{rem:builtin}
The construction is an instance of one rule. Write an equation of the system as $Dz_2+Bz_1=F$, where $D$ has a bounded right inverse $R_D$ ($DR_D=I$) and $B$ acts on the other unknowns, and parameterize $z_2=z_{2,0}+R_D(F-Bz_1)$ with $z_{2,0}$ in a conforming class for $\ker D$. The equation then holds for every parameter value, the functional is the full least-squares functional restricted to $\{Dz_2+Bz_1=F\}$, and its equivalence with the error is inherited (Remark~\ref{rem:variants}). Here $D=\divv$, $R_D=\Rop_2$, $z_2=\btau$, $Bv=\bb\cdot\grad v+cv$, $F=f_1$ and $\ker\divv=\curl H^1\oplus\cH$ (Corollary~\ref{cor:hdiv23}), which gives $\btau=\btau_0+\Rop_2(f_1-\bb\cdot\grad v-cv)$.
\end{rem}

\begin{thm}[error equivalence for the PDE-balanced formulation]\label{thm:balanced_equiv}
Let $(v,\psi,\bh)$ satisfy the exact boundary conditions $v=g_D$ on $\partial\Omega$. Then
\beq\label{balanced_equiv}
 C_\ell\big(\|u-v\|_1^2+\hdn{\bsigma-\btau[v,\psi,\bh]}^2\big)
 \le \LS_{\rm bal}(v,\psi,\bh;F)
 \le C_u\big(\|u-v\|_1^2+\hdn{\bsigma-\btau[v,\psi,\bh]}^2\big).
\eeq
\end{thm}

\begin{proof}
For the error pair $(u-v,\bsigma-\btau[v,\psi,\bh])$ the second residual in \eqref{LS} vanishes by
\eqref{balance_identity}. Hence its full least-squares norm is exactly \eqref{LS_bal}; apply Lemma~\ref{lem:equiv}.
\end{proof}

\begin{cor}[two scalar approximation problems]\label{cor:balanced_approx}
Let $q^*=f_1-\bb\cdot\grad u-cu$ and $\bsigma=\curl\psi^*+\Rop_2 q^*+\bh^*$. Then
\beq\label{balanced_approx}
\hdn{\bsigma-\btau[v,\psi,\bh]}
\le |\psi^*-\psi|_1
 +\sqrt{1+C_{x_0}^2}\,(\|\bb\|_\infty+\|c\|_\infty)\|u-v\|_1
 +\|\bh^*-\bh\|_0 .
\eeq
Thus, apart from the finite-dimensional harmonic component, approximation of the $H(\divvr)$ flux is reduced to $H^1$ approximation of the two scalar functions $u$ and $\psi^*$.
\end{cor}

\begin{proof}
Use Corollary~\ref{cor:red} with $q=q[v]$ and
\[
 \|q^*-q[v]\|_0
 \le \|\bb\|_\infty\|\grad(u-v)\|_0+\|c\|_\infty\|u-v\|_0.
\]
\end{proof}
\begin{rem}[what is and is not built into the network]\label{rem:no_emulation}
The fixed interface factors $|d_\Gamma|$, the Poincar\'e operator $\Rop_2$, and the harmonic basis $\cH$ are deterministic ingredients given in closed form. They are not themselves emulated by
specially constructed neural networks. The trainable objects are ordinary scalar networks approximating the unknown amplitudes. Thus exact $H(\divvr)$ conformity and the prescribed interface locations do not have to be learned by optimization.
\end{rem}

\subsection{Numerical examples}\label{subsec:ell_numerics}
The formulation is that of Section~\ref{subsec:unconstrained}, with $v$ and the representing functions of the flux as outputs of one network,
and both $H(\divvr)$ classes are used; the divergence equation is treated at the first
level of Remark~\ref{rem:div_levels}, with $q_\theta$ a network output in the first class and $\divv\bzz_\theta$
in the second. The Dirichlet condition is imposed weakly, through the boundary term of
Section~\ref{subsec:weak_dirichlet}. We report $\cJ_{\rm check}$ and the relative errors $|u-u_\theta|_1/|u|_1$ and $\|\bsigma-\bsigma_\theta\|_0/\|\bsigma\|_0$, and $\|\divv(\bsigma_\theta-\bsigma)\|_0/\|\bsigma\|_0$, the
divergence error normalized by $\|\bsigma\|_0$ so that it is on the scale of the flux error ($\hdn{\bsigma}^2=\|\bsigma\|_0^2+\|\divv\bsigma\|_0^2$); it is not the standard relative divergence error.

\noindent{\bf Example 1 (curved interface, nonsymmetric problem with a line source).}
$\Omega=(-1,1)^2$, $\Gamma=\{|x-C|=R\}\cap\Omega$ with $C=(-2,0)$ and $R=2.1$, $\mathbf A=\alpha I$ with $\alpha=1$ for $|x-C|<R$ and
$\alpha=3$ outside, $\bb=\frac12(-x_2,x_1+2)$, divergence free and tangent to the circles about $C$, $\bbet=0$ and $c=1$;
$r$ and $\vartheta$ are the distance to $C$ and the polar angle about $C$,
$\be_r$ the radial unit vector and $\chi_+$ the indicator function of $r>R$. With $U=\sin(\pi x_1/2+0.3)\cos(\pi x_2/2)$,
\[
 u=U+\chi_+(r-R)V(\vartheta),\qquad \bff_2=\chi_+\,g(\vartheta)\,\frac Rr\,\be_r,\qquad \bsigma=-\alpha\grad u-\bff_2,
\]
$V(\vartheta)=(\pi^2R/3)\,U(C+R\be_r(\vartheta))$, $g=-(\alpha_+-\alpha_-)\p_rU-\alpha_+V$ on $\Gamma$, $f_1=\divv\bsigma+\bb\cdot\grad u+cu$ and $g_D=u$. Since $g(\vartheta)R/r\,\be_r$ is divergence free for
$r>R$, $\divv\bff_2=g\,\delta_\Gamma$: the $H^{-1}$ part of the data is a line source
on $\Gamma$, and $\bff_2\in L^2$ is not in $H(\divvr)$. Then $\jump{u}=0$, $\jump{\p_ru}=V$, $\jump{\bsigma\cdot\bn}=0$
and $\jump{\bsigma\cdot\be_\vartheta}=-(\alpha_+-\alpha_-)\p_\vartheta U/R$ on $\Gamma$: $u$ has a kink and the flux a
tangential jump. $V$ is chosen so that $\jump{\alpha\Delta u}=0$, which with $\bb\cdot\bn=0$ on
$\Gamma$ makes $f_1$ and $q^*=\divv\bsigma$ continuous across $\Gamma$, so that the divergence, a continuous network output
$q_\theta$ or the divergence of a continuous field $\bzz_\theta$, has no
jump to represent, and no module is needed in either. The classes are generic, $\bsigma_\theta=\curl\psi_\theta+\Rop_2q_\theta$ with
$x_0=(0.2,0.1)$, or $\bsigma_\theta=\curl\psi_\theta+\bzz_\theta$, and $\psi_\theta$, $q_\theta$ or the two components of $\bzz_\theta$, and $v_\theta$ are outputs of
one network ($q_\theta$ a single continuous output, not the piecewise class $\cQ$ of Corollary~\ref{cor:red}) with two $\tanh$ layers
of width $30$ followed by three ReLU layers of width $20$; at the start the first
ReLU layer is fitted to lines in $\Omega$ and
the deeper ones are the identity (active in all of $\Omega$), so that
the initial kinks are those of the first layer. The boundary term is the $H^{1/2}(\partial\Omega)$ norm,
$\|e\|_{0,\partial\Omega}^2+\int_{\partial\Omega}\int_{\partial\Omega}|e(s)-e(t)|^2|s-t|^{-2}\,ds\,dt$, by the product midpoint rule with $256$ points per side,
in which the pairs with equal midpoints are omitted (for a Lipschitz $e$ the
diagonal cell pairs contribute $O(h)$ in total). The training cells are squares of side
$2/64$, refined twice within distance $0.1$ of $\Gamma$ and cut along $\Gamma$,
and the derivative residuals enter by the cellwise moment treatment of Section~\ref{subsec:computation}; training
runs in six phases of fifteen minutes, with refinement and four new units after
each phase. The check $\cJ_{\rm check}$ uses squares of side $2/128$, refined
near $\Gamma$ and cut along $\Gamma$ and all kink sets of the network.
\begin{center}
\small
\begin{tabular}{cccccccccc}
\toprule
class & width & cells & steps & time & $\cJ_h$ & $\cJ_{\rm check}$ & $|u-u_\theta|_1$ & $\|\bsigma-\bsigma_\theta\|_0$ & $\|\divv(\bsigma_\theta-\bsigma)\|_0$ \\
\midrule
$\Rop_2q$ & 40 & 24229 & 422 & 90 min & $6.19\cdot10^{-4}$ & $8.01\cdot10^{-4}$ & $0.28\%$ & $0.26\%$ & $0.11\%$ \\
$\bzz$ & 40 & 26719 & 274 & 90 min & $1.19\cdot10^{-3}$ & $1.37\cdot10^{-3}$ & $0.30\%$ & $0.27\%$ & $0.27\%$ \\
\bottomrule
\end{tabular}
\end{center}
At $90$ minutes the two classes give errors of $0.28\%$ and $0.30\%$ in $u$
and of $0.26\%$ and $0.27\%$ in $\bsigma$; the second run made
$274$ steps, the first $422$. With the first, $\cJ_{\rm check}$ decreased in every phase,
from $4.1\cdot10^{-3}$ after the first to $8.0\cdot10^{-4}$ after the last ($0.60\%$
and $0.26\%$ for $\bsigma$), and was still decreasing at the time limit. Figure~\ref{fig:ell_generic} shows
the first class. The error of $\bsigma$ is largest in a
narrow band along $\Gamma$ (Figure~\ref{fig:ell_generic}(a)). No single kink set of the network
lies on $\Gamma$; the tangential jump of $\bsigma$ is carried by
several kink sets of the first layer close to it (Figure~\ref{fig:ell_generic}(b)), and measured across
$\pm0.005$ it follows the exact jump with a relative error of $7.8\%$ (Figure~\ref{fig:ell_generic}(c)).
\begin{figure}[t]
\centering
\includegraphics[width=0.92\textwidth]{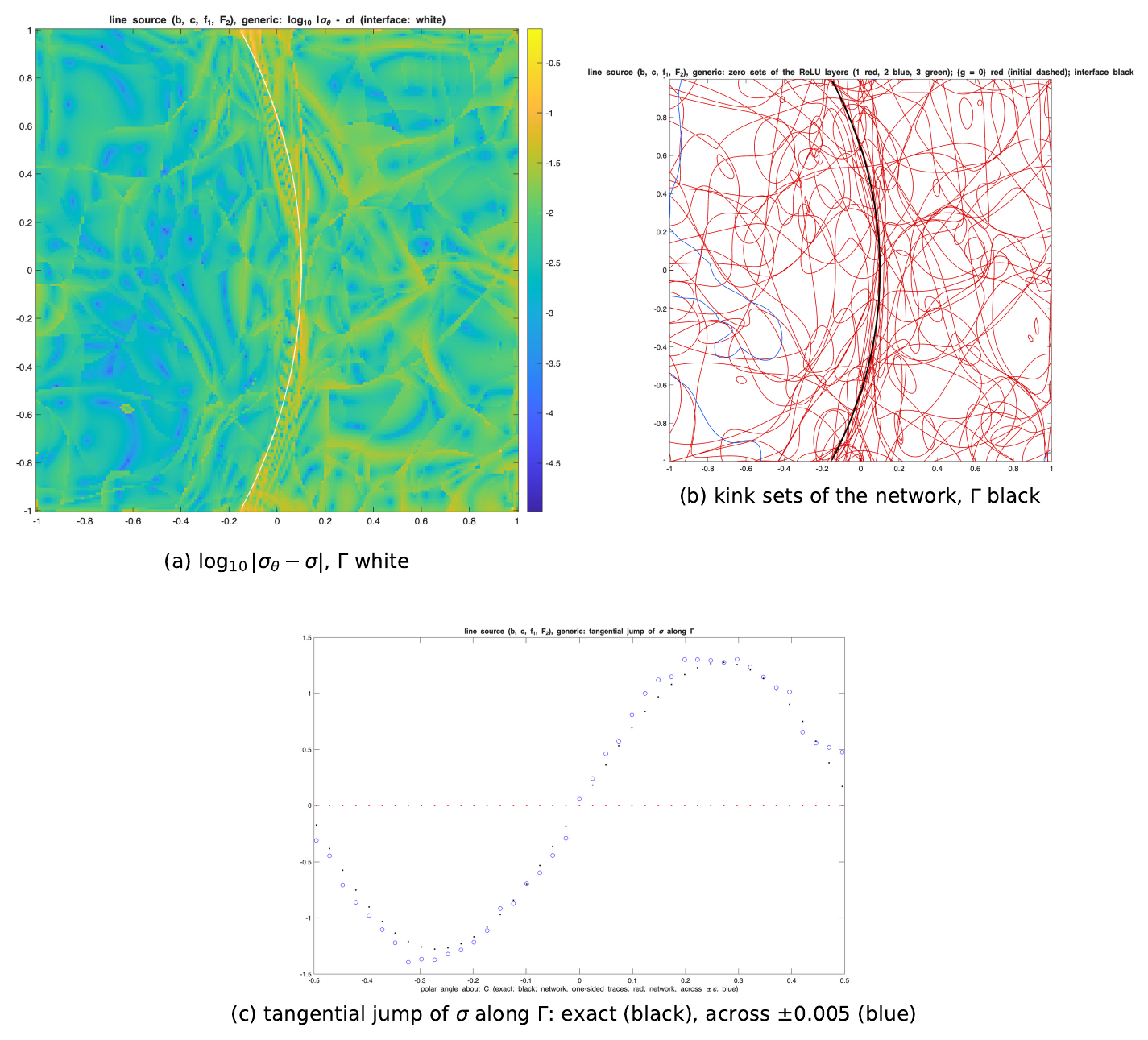}
\caption{Example~1, generic level, first class: (a) $\log_{10}|\bsigma_\theta-\bsigma|$; (b) the kink sets of the three ReLU layers and $\Gamma$; (c) the tangential jump of
$\bsigma$ along $\Gamma$ against the polar angle about $C$: exact, and of $\bsigma_\theta$ across $\pm0.005$.}\label{fig:ell_generic}
\end{figure}

\noindent{\bf Example 2 (hyperplane interface in $\R^n$, Levels I and II).}
$\Omega$ is the unit ball of $\R^n$, $n=4,8$, $\Gamma=\{x_1=s\}\cap\Omega$ with $s=0.1$, $\mathbf A=\alpha I$ with $\alpha=1$ for $x_1<s$ and $\alpha=3$ for $x_1>s$, $\bb=n^{-1/2}(1,\ldots,1)$, $\bbet=0$, $c=1$, so the problem is nonsymmetric. The solution and the flux are chosen independently:
\[
 u=U\ (x_1<s),\quad u=U+(x_1-s)V\ (x_1>s);\qquad \bsigma=\mathbf S\ (x_1<s),\quad \bsigma=\mathbf S+M\be_2\ (x_1>s),
\]
$U=n^{-1/2}\sum_r\cos(\pi x_r/2)$, $V=1+\frac12\sin\pi x_2$, $S_r=n^{-1/2}\sin(\pi x_{r+1}/2)+x_1\delta_{r1}$ (indices cyclic), $M=1+2\sin\pi x_3$,
and $\bff_2=-\alpha\grad u-\bsigma$, $f_1=\divv\bsigma+\bb\cdot\grad u+cu$, $g_D=u$. Then $\jump{u}=0$, $\jump{\p_1u}=V$,
$\jump{\bsigma\cdot\be_1}=0$ and $\jump{\bsigma\cdot\be_2}=M$: the kink of $u$ and the
tangential jump of $\bsigma$ are independent, and $\bff_2\in L^2$
is not in $H(\divvr)$, since $\jump{\bff_2\cdot\be_1}=-\jump{\alpha\p_1u}\ne0$. The classes are
$\bsigma_\theta=\operatorname{Div}A_\theta+\Rop_nq_\theta$ with center $0$, or $\bsigma_\theta=\operatorname{Div}A_\theta+\bzz_\theta$ with $\bzz_\theta$ an output of
$n$ components without module, since $\divv\bsigma=1$ on both sides; $A_\theta=A_{0,\theta}+mN_\theta$ skew,
$v_\theta=v_{0,\theta}+mv_{1,\theta}$, all outputs of one network with two
$\tanh$ layers of width $32$; $m=|x_1-s|$ (Level~I) or
$m=\operatorname{ReLU}(g_\eta)$ (Level~II) with $g_\eta(x)=\mathbf a\cdot x-c_0+\mathbf w^T\tanh(Wx+\mathbf b_0)+c_2$, $W$
with $10$ rows, initially the plane $x_1+0.15x_2+0.1x_3=s+0.15$. The integrals
are taken on rays from $0$ along $N$ scrambled
Sobol directions, split at $\Gamma$ (Level~I) or at the zeros
of $g_\eta$ (Level~II), with Gauss rules for the weight
$\rho^{n-1}$, so the functional is integrated directly; the boundary
term is the $L^2$ form of Section~\ref{subsec:weak_dirichlet}, $w_b\|v-g_D\|_{0,\partial\Omega}^2$
with $w_b=100$. Level~I is trained by Levenberg--Marquardt; Level~II
in two stages, L-BFGS on all parameters (the interface fixed
for the first $100$ iterations) until the zeros of $g_\eta$
settle, then Levenberg--Marquardt on the network with the
interface fixed. The check $\cJ_{\rm check}$ uses $4N$ directions.
\begin{center}
\small\setlength{\tabcolsep}{4pt}
\begin{tabular}{cccccccccc}
\toprule
$n$ & level & class & $N$ & iterations & $\cJ$ & $\cJ_{\rm check}$ & $|u-u_\theta|_1$ & $\|\bsigma-\bsigma_\theta\|_0$ & distance \\
\midrule
4 & I & $\Rop_nq$ & 1024 & 463 & $8.9\cdot10^{-9}$ & $1.5\cdot10^{-8}$ & $1.7\cdot10^{-5}$ & $4.0\cdot10^{-5}$ & -- \\
4 & I & $\bzz$ & 1024 & 543 & $4.5\cdot10^{-9}$ & $8.6\cdot10^{-9}$ & $1.3\cdot10^{-5}$ & $3.6\cdot10^{-5}$ & -- \\
8 & I & $\Rop_nq$ & 2048 & 99 & $8.4\cdot10^{-6}$ & $1.5\cdot10^{-5}$ & $5.9\cdot10^{-4}$ & $1.0\cdot10^{-3}$ & -- \\
8 & I & $\bzz$ & 2048 & 97 & $2.1\cdot10^{-6}$ & $4.8\cdot10^{-6}$ & $4.0\cdot10^{-4}$ & $3.7\cdot10^{-4}$ & -- \\
4 & II & $\Rop_nq$ & 1024 & 586 & $1.5\cdot10^{-3}$ & $2.7\cdot10^{-3}$ & $1.2\%$ & $1.5\%$ & $3.6\cdot10^{-4}$ \\
4 & II & $\bzz$ & 1024 & 875 & $9.7\cdot10^{-4}$ & $1.5\cdot10^{-3}$ & $1.0\%$ & $1.0\%$ & $6.3\cdot10^{-5}$ \\
8 & II & $\Rop_nq$ & 4096 & 548 & $5.9\cdot10^{-4}$ & $7.1\cdot10^{-4}$ & $0.82\%$ & $0.93\%$ & $1.8\cdot10^{-5}$ \\
8 & II & $\bzz$ & 4096 & 398 & $1.4\cdot10^{-3}$ & $1.7\cdot10^{-3}$ & $1.2\%$ & $1.5\%$ & $6.0\cdot10^{-5}$ \\
\bottomrule
\end{tabular}
\end{center}
``Distance'' is the mean of $|x_1^*-s|$ over the section $x_4=\cdots=x_n=0$
of $\Gamma$, where $g_\eta(x_1^*,x_2,x_3,0,\ldots,0)=0$ (for the first class the maximum
is $1.4\cdot10^{-3}$ for $n=4$ and $1.4\cdot10^{-4}$ for $n=8$; Figure~\ref{fig:highd_interface}).
With the known interface the module carries the kink exactly
on $\Gamma$; for the first class the errors are $10^{-5}$
for $n=4$ and about $10^{-3}$ for $n=8$, both runs
ended at their time limits of $15$ and $30$ minutes,
and for $n=8$ an iteration takes $18$ seconds.
For the first class with the learned interface, the first
stage moved the zero set of $g_\eta$ from the
initial tilted plane toward $\Gamma$; the interface was then
fixed, and Levenberg--Marquardt was applied to the remaining
network parameters: for $n=4$, $500$ L-BFGS iterations and $86$ Levenberg--Marquardt steps; for
$n=8$, $486$ L-BFGS iterations, at the limit of $84$
minutes for the first stage, and $62$ Levenberg--Marquardt steps, in
$120$ minutes in all. For the first class
at $n=8$, the learned interface lies within $1.4\cdot10^{-4}$
of $\Gamma$ on the section, and the errors are
below $1\%$. With the same network and time, the second class
gives, with the known interface, errors of $u$ and $\bsigma$
smaller by factors $1.3$ and $1.1$ for $n=4$
and $1.5$ and $2.7$ for $n=8$. With the learned
interface it gives smaller errors and a smaller mean
distance for $n=4$ ($1.0\%$ against $1.2\%$ and $1.5\%$, $6.3\cdot10^{-5}$ against $3.6\cdot10^{-4}$, maximum
$3.8\cdot10^{-4}$), and larger ones for $n=8$ ($1.2\%$ and $1.5\%$ against
$0.82\%$ and $0.93\%$, $6.0\cdot10^{-5}$ against $1.8\cdot10^{-5}$, maximum $5.6\cdot10^{-4}$), where it made
$398$ iterations against $548$. Figure~\ref{fig:highd_interface} shows the first class.
\begin{figure}[t]
\centering
\includegraphics[width=\textwidth]{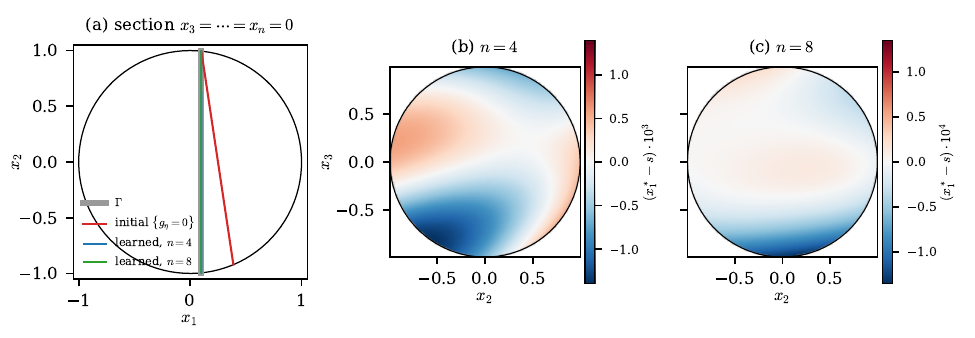}
\caption{Example~2, Level~II.  (a) Section $x_3=\cdots=x_n=0$ of the unit ball: $\Gamma=\{x_1=s\}$ (gray), the initial zero set of $g_\eta$ (red)
and the learned zero sets for $n=4$ (blue) and $n=8$ (green). (b), (c) Position of the learned zero set on the section $x_4=\cdots=x_n=0$: $x_1^*-s$ with $g_\eta(x_1^*,x_2,x_3,0,\ldots,0)=0$, for $n=4$ (scale $10^{-3}$) and $n=8$ (scale $10^{-4}$).}
\label{fig:highd_interface}
\end{figure}

\section{Curl--curl problems}\label{sec:curlcurl}

Let $\Omega\subset\R^3$ be a bounded Lipschitz domain, $\mu,\kappa\in L^\infty(\Omega)$ with $\mu\ge\mu_0>0$ and $\kappa\ge\kappa_0>0$, and $\bff\in L^2(\Omega)^3$. Seek $\bE\in H_0(\curll;\Omega)$ with
\beq\label{curlcurl3d}
 \curll\big(\mu^{-1}\curll\bE\big)+\kappa\bE=\bff\quad\mbox{in }\Omega,\qquad \bE\times\bn=0\quad\mbox{on }\partial\Omega .
\eeq
With $\bH:=\mu^{-1}\curll\bE\in H(\curll;\Omega)$ the first-order system is $\mu^{1/2}\bH-\mu^{-1/2}\curll\bE=0$, $\curll\bH+\kappa\bE=\bff$, and
\beq\label{LS_curl3d}
 \cL_3(\bF,\bW;\bff):=\|\mu^{1/2}\bW-\mu^{-1/2}\curll\bF\|_0^2+\|\kappa^{-1/2}(\curll\bW+\kappa\bF-\bff)\|_0^2 .
\eeq

\begin{lem}[exact norm identity]\label{lem:curl_identity}
For $\bF\in H_0(\curll;\Omega)$ and $\bW\in H(\curll;\Omega)$,
\[
 \cL_3(\bF,\bW;0)=\|\mu^{-1/2}\curll\bF\|_0^2+\|\kappa^{1/2}\bF\|_0^2+\|\mu^{1/2}\bW\|_0^2+\|\kappa^{-1/2}\curll\bW\|_0^2 .
\]
\end{lem}

\begin{proof}
Expanding, the cross terms are $-2(\bW,\curll\bF)+2(\curll\bW,\bF)$, and $(\curll\bW,\bF)=(\bW,\curll\bF)$ for $\bF\in H_0(\curll;\Omega)$.
\end{proof}

Since the system is linear, $\cL_3(\bF,\bW;\bff)$ is the
right-hand side of the lemma for $(\bE-\bF,\bH-\bW)$: the
functional is exactly the error in the energy norms,
and no separate stability argument for the least-squares formulation
is needed (the existence of the solution comes from the
well-posedness of \eqref{curlcurl3d}). This is Proposition~\ref{prop:best_approx} with $X=H_0(\curll;\Omega)\times H(\curll;\Omega)$ and $c_L=C_L=1$ in
these norms. We use the classes of Section~\ref{subsec:hcurl3d},
\[
 \bE_\theta=\grad\phi_\theta+\bzz_\theta,\qquad \bH_\theta=\grad\chi_\theta+\bw_\theta ,
\]
with $\phi_\theta\in H^1_0(\Omega)$ and $\bzz_\theta\in H^1_0(\Omega)^3$ (networks multiplied by a
function that vanishes on $\partial\Omega$), so that $\bE_\theta\times\bn=0$
on $\partial\Omega$ for every parameter value, and $\chi_\theta\in
H^1(\Omega)$, $\bw_\theta\in H^1(\Omega)^3$; by Proposition~\ref{prop:hcurl3d_density} the error is bounded by the
$H^1$ errors of $\phi_\theta$, $\bzz_\theta$, $\chi_\theta$ and $\bw_\theta$.

\noindent{\bf Interfaces.}
Across an interface of $\kappa$ or $\mu$, $\bE$ and $\bH$ have
continuous tangential traces, while $\bE\cdot\bn$ may jump, and so may $\bH\cdot\bn$ where
$\mu$ jumps, since $\mu\bH=\curll\bE\in H(\divvr;\Omega)$ has a continuous normal component.
Both normal jumps are kinks of the scalar potentials, $\jump{\bE\cdot\bn}=\jump{\p_{\bn}\phi}$
and $\jump{\bH\cdot\bn}=\jump{\p_{\bn}\chi}$ (Proposition~\ref{prop:hcurl_jump}), and the scalar interface modules of that
proposition are used for $\phi_\theta$ (inside the boundary factor) and $\chi_\theta$;
the regular parts $\bzz_\theta$ and $\bw_\theta$ have no jumps.

\begin{rem}[no equation built in]
Neither equation is built in. Building in $\curll\bE_\theta=\mu\bH_\theta$ would
require $\divv(\mu\bH_\theta)=0$, which a network $\bH_\theta\in H(\curll)$ does not
satisfy automatically; parameterizing $\mu\bH=\curll\bPsi$ instead loses the tangential continuity
of $\bH$ across jumps of $\mu$. Only the conformity is structural.
\end{rem}

\subsection{Numerical example}\label{subsec:curl_numerics}
$\Omega=(0,1)^3$, $\mu=1$, $\kappa=1$ for $x_1<\frac12$ and $\kappa=10$ for $x_1>\frac12$, and
\[
 \bE=\bE_1+\grad\phi^*,\qquad \bE_1=(\sin\pi x_2\sin\pi x_3,\ \sin\pi x_3\sin\pi x_1,\ \sin\pi x_1\sin\pi x_2),
\]
\[
 \phi^*=A\bigl(\tfrac12-|x_1-\tfrac12|\bigr)\sin\pi x_2\sin\pi x_3 ,
\]
with $A=9/11$, $\bH=\curll\bE_1$ and $\bff=2\pi^2\bE_1+\kappa\bE$. Then $\bE\times\bn=0$ on $\partial\Omega$, $\bE\times\be_1$ is continuous across $x_1=\frac12$, $\jump{\bE\cdot\be_1}=-2A\sin\pi x_2\sin\pi x_3$ and $\jump{\kappa\bE\cdot\be_1}=0$. The classes are $\bE_\theta=\grad(B\phi_\theta)+B\bzz_\theta$
and $\bH_\theta=\grad\chi_\theta+\bw_\theta$ with $1/B=1/(x_1(1-x_1))+1/(x_2(1-x_2))+1/(x_3(1-x_3))$, so that $B$ is Lipschitz in $\overline\Omega$, smooth in $\Omega$ and $B=0$ on $\partial\Omega$, and $\phi_\theta$, $\bzz_\theta$,
$\chi_\theta$, $\bw_\theta$ are outputs of one network with two $\tanh$ layers of width $30$. At Level~I, $\phi_\theta=\phi_{0,\theta}+|x_1-\frac12|\,N_\theta$; at Level~II, $\phi_\theta=\phi_{0,\theta}+\operatorname{ReLU}(g_\eta)N_\theta$ with $g_\eta$ as in Example~2, initially a
tilted plane; at Level~III the network has three further ReLU layers of width $20$ and no module. The training functional is that
of Section~\ref{sec:framework}, with the moments of $\grad(B\phi)$, $\curll(B\bzz)$, $\grad\chi$ and $\curll\bw$ computed from values on boxes; the boxes
have side $1/16$, graded three times towards $x_1=\frac12$ and once towards $\partial\Omega$, and are refined after every phase of ten minutes, up to $16000$ boxes. The check $\cJ_{\rm check}$
uses $28224$ boxes with $3^3$ Gauss points; for Level~III it is not cut along the kink
sets of the network. We report the relative $L^2$ errors and the relative error in the energy norm of Lemma~\ref{lem:curl_identity}.
\begin{center}
\small\setlength{\tabcolsep}{4pt}
\begin{tabular}{cccccccccc}
\toprule
level & seed & steps & time & $\cJ_{\rm check}$ & $\bE$ & $\curll\bE$ & $\bH$ & $\curll\bH$ & energy \\
\midrule
I & 1 & 96 & 13 min & $3.4\cdot10^{-3}$ & $0.49\%$ & $1.49\%$ & $0.14\%$ & $0.10\%$ & $0.42\%$ \\
II & 1 & 75 & 13 min & $7.9\cdot10^{-3}$ & $1.15\%$ & $2.19\%$ & $0.28\%$ & $0.22\%$ & $0.64\%$ \\
II & 2 & 72 & 13 min & $1.0\cdot10^{-2}$ & $1.73\%$ & $2.42\%$ & $0.39\%$ & $0.27\%$ & $0.72\%$ \\
II & 3 & 56 & 10 min & $2.6\cdot10^{-2}$ & $2.23\%$ & $3.58\%$ & $1.00\%$ & $0.78\%$ & $1.14\%$ \\
III & 1 & 183 & 60 min & $3.6\cdot10^{-2}$ & $4.47\%$ & $3.49\%$ & $1.23\%$ & $1.17\%$ & $1.35\%$ \\
\bottomrule
\end{tabular}
\end{center}
The runs at Levels~I and~II ended before their time limit
of $30$ minutes. At Level~II all three initial networks moved the zero
set of $g_\eta$ from the initial tilted plane to the
interface (Figure~\ref{fig:curl3d_interface}): with $g_\eta(x_1^*,x_2,x_3)=0$, the mean of $|x_1^*-\frac12|$ over
$(x_2,x_3)\in(0,1)^2$ is $7.5\cdot10^{-4}$, $1.0\cdot10^{-3}$ and $7.9\cdot10^{-5}$, and its
maximum $1.3\cdot10^{-3}$, $1.2\cdot10^{-3}$ and $2.1\cdot10^{-3}$. The energy-norm error is below $1.5\%$
at all levels and about $0.5\%$ with the interface known.
\begin{figure}[t]
\centering
\includegraphics[width=\textwidth]{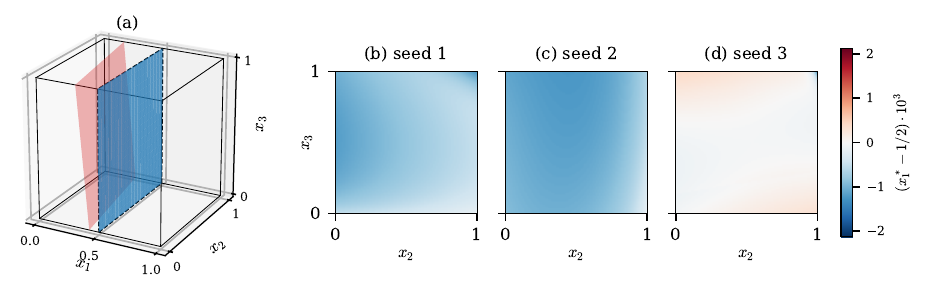}
\caption{Section~\ref{subsec:curl_numerics}, Level~II.  (a) The interface $x_1=\frac12$ (dashed), the initial zero set of $g_\eta$ (red) and the
learned zero set for the first initial network
(blue). (b)--(d) Position of the learned zero set for
the three initial networks: $x_1^*-\frac12$ with $g_\eta(x_1^*,x_2,x_3)=0$, over $(x_2,x_3)\in(0,1)^2$.}
\label{fig:curl3d_interface}
\end{figure}

\section{Linear transport}\label{sec:transport}

We use the flux formulation of \cite{LZ:18}, in which the flux is the $H(\divvr)$ unknown.

Let $\Omega\subset\R^2$ be a bounded, simply connected Lipschitz
polygon, $\bbet\in W^{1,\infty}(\Omega)^2$, $\gamma\in L^\infty(\Omega)$, $f\in L^2(\Omega)$, and
$\Gamma_-:=\{x\in\partial\Omega:\bbet\cdot\bn<0\}$ the inflow boundary. We seek $u$ with
\beq\label{transport}
 \divv(\bbet u)+\gamma u=f\quad\mbox{in }\Omega,\qquad u=g\quad\mbox{on }\Gamma_- ,
\eeq
where $\bbet$, $\gamma$ satisfy \cite[Assumption~2.1]{LZ:18} and $g\in L^2(|\bbet\cdot\bn|;\Gamma_-)$; then \eqref{transport} has a unique solution \cite[Theorem~2.2]{LZ:18}. The flux $\bsigma:=\bbet u$ satisfies $\bsigma-\bbet u=0$ and $\divv\bsigma+\gamma u=f$, so $\bsigma\in H(\divvr;\Omega)$, and $\bsigma\cdot\bn=(\bbet\cdot\bn)g$ on $\Gamma_-$. In the language of differential forms, $\bsigma$ is the contraction $\iota_{\bbet}$ of the density $u\,\mathrm{vol}$, and $\divv\bsigma$ is its Lie derivative along $\bbet$, by Cartan's formula $\mathcal L_{\bbet}=\dd\iota_{\bbet}+\iota_{\bbet}\dd$ for forms of top degree \cite[Section~5.1]{Hu:26review}; this is why the flux is the $H(\divvr)$ unknown. The least-squares functional is
\beq\label{LS_tr}
 \cJ_{\rm tr}(\btau,v):=\|\btau-\bbet v\|_0^2+\|\divv\btau+\gamma v-f\|_0^2+w_s\|\btau\cdot\bn-(\bbet\cdot\bn)g\|_{0,\Gamma_-}^2+w_u\|v-g\|_{0,\Gamma_-}^2 .
\eeq
Both inflow conditions hold for the solution and are imposed here through the boundary terms, $w_s,w_u>0$; for the last term we assume in addition $g\in L^2(\Gamma_-)$ (the weighted term $w_u\||\bbet\cdot\bn|^{1/2}(v-g)\|_{0,\Gamma_-}^2$ would need no additional assumption). With
the flux condition imposed in the trial space and $w_s=w_u=0$, \eqref{LS_tr} is the functional of \cite[Section~3]{LZ:18}; the native classes of Sections~\ref{sec:hdiv} and~\ref{sec:hdiv_reg}, however, do not preserve a prescribed normal trace on $\Gamma_-$. We consider the functional on admissible pairs $(\btau,v)\in H(\divvr;\Omega)\times L^2(\Omega)$ for
which the two boundary traces appearing in $\cJ_{\rm tr}$, $\btau\cdot\bn$ and $v$ on $\Gamma_-$, are defined and square integrable; this includes the exact solution and the neural trial classes used below.

\begin{lem}[error identity]\label{lem:tr_norm}
For admissible pairs $(\btau,v)$ let $|||(\btau,v)|||^2:=\|\btau-\bbet v\|_0^2+\|\divv\btau+\gamma v\|_0^2+w_s\|\btau\cdot\bn\|_{0,\Gamma_-}^2+w_u\|v\|_{0,\Gamma_-}^2$. Then
\beq\label{tr_identity}
 \cJ_{\rm tr}(\btau,v)=|||(\btau-\bsigma,v-u)|||^2 ,
\eeq
$|||(\btau,v)|||=0$ only for $(\btau,v)=0$, and $|||(\btau,v)|||\le C\big(\|\btau\|_{H(\divvr)}+\|v\|_0+\|\btau\cdot\bn\|_{0,\Gamma_-}+\|v\|_{0,\Gamma_-}\big)$.
\end{lem}

\begin{proof}
The problem is linear and $(\bsigma,u)$ makes every term vanish, which gives \eqref{tr_identity}. If $|||(\btau,v)|||=0$, then $\btau=\bbet v$, $\divv(\bbet v)+\gamma
v=0$ and $(\bbet\cdot\bn)v=\btau\cdot\bn=0$ on $\Gamma_-$, so $v=0$ there and $v=0$, $\btau=0$ by uniqueness \cite[Lemma~3.1]{LZ:18}. The upper bound is the triangle inequality.
\end{proof}

No lower bound by the norm of $H(\divvr;\Omega)\times L^2(\Omega)$ is claimed: the functional is the error in $|||\cdot|||$, and the statements below are in this norm; in particular nothing
is claimed for the $L^2$ error of $v$. For the flux classes of Sections~\ref{subsec:hdiv23} and~\ref{subsec:hdiv_reg_classes}, $\btau_\theta=\curl\psi_\theta+\Rop_2q_\theta$ or $\btau_\theta=\curl\psi_\theta+\bzz_\theta$, and a scalar network $v_\theta$, \eqref{tr_identity} bounds the error
of an approximate minimizer by the best approximation error in $|||\cdot|||$, and hence by the errors of the representing functions of $\bsigma=\curl\psi^*+\Rop_2q^*$, $q^*=f-\gamma u$ (Corollary~\ref{cor:hdiv23}), or
of $\bsigma=\curl\psi^*+\bzz^*$ (Theorem~\ref{thm:hdivd_reg}), of $u$ in $L^2(\Omega)$, and of $\bsigma\cdot\bn$ and $u$ in $L^2(\Gamma_-)$. The divergence equation is treated
at the first level of Remark~\ref{rem:div_levels}: the second residual is $q_\theta+\gamma v_\theta-f$, or $\divv\bzz_\theta+\gamma v_\theta-f$; eliminating $q$ by $q=f-\gamma v$, which only the first class allows, would be the second level.

Assume in addition that $f$ and $\gamma$ are
piecewise continuous near a characteristic $\Gamma$ through a
jump of $g$, with $\gamma$ continuous across $\Gamma$
and $f$ without a jump there. Across $\Gamma$, $u$ jumps and $\bbet\cdot\bn_\Gamma=0$,
so $\jump{\bsigma\cdot\bn}=0$ and $\jump{\bsigma}=\jump{u}\bbet$ is tangential. By Corollary~\ref{cor:jump2d}(b) the
jump is carried by the kink of the potential,
$\jump{\p_{\bn}\psi^*}=-\jump{\bsigma\cdot\bt}$, while $q^*$ jumps by $-\gamma\jump{u}$: in the first
class $\Rop_2q^*$ is continuous under the condition on the segments
of Lemma~\ref{lem:Rcont}, and in the second the jump
of $q^*=\divv\bzz^*$ is carried by a kink of
$\bzz^*\in H^1(\Omega)^2$, which has no interface jump (Lemma~\ref{lem:div_kink}). The
interface modules of Sections~\ref{subsec:interface} and~\ref{subsec:learned_interface} apply with the
characteristic as the interface. The value $v$ enters only
through $\btau-\bbet v$, $\gamma v$ and the inflow term,
and $v$ and $q$ are required only in $L^2$; with
an interface variable they may therefore jump across
it, $v=V_0+\chi V_1$ and $q=Q_0+\chi Q_1$ with $\chi$ the
indicator function of one side, or, without an interface variable,
by steps of the units of a generic network (Section~\ref{subsec:tr_numerics}), while
a continuous network $v_\theta$ approximates the jump of $u$ by a transition.

\subsection{Numerical example}\label{subsec:tr_numerics}
We take the example of \cite[Section~7.10]{LZ:18} in its discontinuous limit:
$\Omega=(0,1)^2$, $\bbet=(x_2+1,-x_1)/r$ with $r=|(x_1,x_2+1)|$, the unit clockwise field about $(0,-1)$, which
is divergence free, $\gamma=0.1$, $f=0$, the inflow boundary $\Gamma_-=\{x_1=0\}\cup\{x_2=1\}$, and
\[
 u=\tfrac{\pi}{8}\,e^{\gamma r\vartheta}\operatorname{sgn}(r-1.5),\qquad \vartheta=\operatorname{atan2}(x_2+1,x_1).
\]
Since $\bbet$ is tangent to the circles $r=\mathrm{const}$, $\bbet\cdot\grad u+\gamma u=0$ holds in
the sense of distributions, and $u$ jumps across the characteristic $\Gamma=\{r=1.5\}$, which runs from $(0,0.5)$ to $(1,0.118)$,
by $(\pi/4)e^{0.15\vartheta}\in[0.89,0.99]$. (The computations use the profile $\frac14e^{\gamma r\vartheta}\arctan((r-1.5)/\varepsilon)$ of \cite{LZ:18} with $\varepsilon=10^{-10}$, which differs from
$u$ only in a layer of width $O(\varepsilon)$ about $\Gamma$; all numerical errors reported below are measured against
this profile.) The interface is learned (Level~II of Section~\ref{subsubsec:hdiv_levels}): $\bsigma_\theta=\curl\psi_\theta+\Rop_2q_\theta$ with $x_0=(0.3,0.2)$ and
\[
 \psi_\theta=P_0+\operatorname{ReLU}(g_\eta)P_1,\qquad q_\theta=Q_0+\chi Q_1,\qquad v_\theta=V_0+\chi V_1,\qquad \chi=\one_{\{g_\eta>0\}},
\]
where $P_0,P_1$ are the outputs of a network with two $\tanh$ layers of width $30$, $Q_0,Q_1,V_0,V_1$ those of a network with two $\tanh$ layers of width
$20$, and $g_\eta=\tilde g-\tilde g(x_d)$ with $\tilde g(x)=\mathbf a\cdot x+c+\mathbf w^T\tanh(Ux+\mathbf b)$, $U$ with $10$ rows, so that the curve $\{g_\eta=0\}$ passes through the jump point $x_d=(0,0.5)$
of the inflow data for every parameter value. In the second class $\bsigma_\theta=\curl\psi_\theta+\bzz_\theta$ with the same $\psi_\theta$ and $v_\theta$ and
\[
 \bzz_\theta=\mathbf Z_0+\operatorname{ReLU}(g_\eta)\mathbf Z_1 ,
\]
$\mathbf Z_0,\mathbf Z_1$ four further outputs of the
flux network and $Q_0,Q_1$ dropped: the divergence $-\gamma
u$ of the exact flux jumps with $u$,
and by Lemma~\ref{lem:div_kink} the kink of $\bzz_\theta$ on
$\{g_\eta=0\}$ carries this jump while $\bzz_\theta$, a network with
continuous activations, stays continuous. Initially $\{g_\eta=0\}$ is the
line $x_2=0.5$, the tangent of the characteristic at
$x_d$, which is determined by the data. The
functional is \eqref{LS_tr} with $w_s=10$ and $w_u=100$. It
is integrated on squares of side $1/N_q$; a square crossed
by $\{g_\eta=0\}$ is cut along the chord of
$g_\eta$ linearized at its center, with the centroid rule
on the two pieces and $\chi$ given by
the side of the chord, and the inflow edges,
with $2N_q$ segments each, are cut likewise; in the
first class $\Rop_2q_\theta$ is evaluated on rays split at
the zeros of $g_\eta$. The derivative with respect to $\eta$
contains the motion of the cut: a term
on the chords, point terms where $\{g_\eta=0\}$ meets the
inflow edges, and, in the first class, the
moving zeros on the rays. The output coefficients
are computed by the linear least-squares solve and
the other parameters by L-BFGS. Two starts of
$8$ minutes with $N_q=128$ were made, and the
better was continued for $60$ minutes with $N_q=256$; $\cJ_{\rm check}$
is the functional on the rule with squares and
segments halved. By Lemma~\ref{lem:tr_norm} the functional is the error
in $|||\cdot|||$; the $L^2$ errors are reported for information.
\begin{center}
\footnotesize\setlength{\tabcolsep}{3pt}
\begin{tabular}{lcccccccc}
\toprule
 & $\cJ_{\rm check}$ & $\frac{\|u-v_\theta\|_0}{\|u\|_0}$ & $\frac{\|u-v_\theta\|_{L^1}}{\|u\|_{L^1}}$ & $\frac{\|\bsigma-\bsigma_\theta\|_0}{\|\bsigma\|_0}$ & $\frac{\|\divv(\bsigma-\bsigma_\theta)\|_0}{\|\divv\bsigma\|_0}$ & jump & \multicolumn{2}{c}{distance to $\Gamma$} \\
 & & & & & & & mean & max \\
\midrule
$\Rop_2q$, start, $N_q=128$ & $8.1\cdot10^{-10}$ & $5.3\cdot10^{-4}$ & $1.4\cdot10^{-4}$ & $5.3\cdot10^{-4}$ & $5.3\cdot10^{-4}$ & $3.5\cdot10^{-4}$ & $2.3\cdot10^{-5}$ & $4.0\cdot10^{-5}$ \\
$\Rop_2q$, main, $N_q=256$ & $1.7\cdot10^{-10}$ & $1.8\cdot10^{-4}$ & $6.4\cdot10^{-5}$ & $1.8\cdot10^{-4}$ & $1.8\cdot10^{-4}$ & $2.9\cdot10^{-4}$ & $7.2\cdot10^{-6}$ & $1.5\cdot10^{-5}$ \\
$\bzz$, start, $N_q=128$ & $3.9\cdot10^{-11}$ & $9.5\cdot10^{-5}$ & $2.1\cdot10^{-5}$ & $9.5\cdot10^{-5}$ & $1.1\cdot10^{-4}$ & $1.1\cdot10^{-4}$ & $2.1\cdot10^{-6}$ & $6.8\cdot10^{-6}$ \\
$\bzz$, main, $N_q=256$ & $1.2\cdot10^{-11}$ & $4.7\cdot10^{-5}$ & $1.9\cdot10^{-5}$ & $4.7\cdot10^{-5}$ & $5.5\cdot10^{-5}$ & $3.1\cdot10^{-5}$ & $2.6\cdot10^{-6}$ & $7.1\cdot10^{-6}$ \\
\bottomrule
\end{tabular}
\end{center}
The jump column is the relative $\ell^2$ error of
the jump of $v_\theta$ across $\Gamma$, from the traces at distance
$0.005$ along the normal, and the distance of $\{g_\eta=0\}$
from $\Gamma$ is taken on the chords of the
training rule. The other start reached $\cJ_{\rm check}=1.1\cdot10^{-9}$ with the
first class and $4.1\cdot10^{-11}$ with the second. The second
class gives smaller values in all reported quantities, by factors
from about $2$ (the maximal distance) to $14$ ($\cJ_{\rm
check}$). It needs no ray evaluation: in the main
run of $60$ minutes it made $7586$ iterations, against
$1560$ for the first. With the first class (Figure~\ref{fig:transport_learned}), starting
from a straight line, the learned curve moves onto
the characteristic, to within $1.5\cdot10^{-5}$ (Figure~\ref{fig:transport_learned}(a)), and the solution is resolved
as a jump, not as a transition layer (Figure~\ref{fig:transport_learned}(b)--(d)):
the jump of $v_\theta$ agrees with that of $u$ along
$\Gamma$, and the error is largest in a band downstream along $\Gamma$, at about
$3\cdot10^{-4}$, apart from points between the learned curve and $\Gamma$. The
functional was still decreasing when the time limit was reached.
\begin{figure}[t]
\centering
\includegraphics[width=0.92\textwidth]{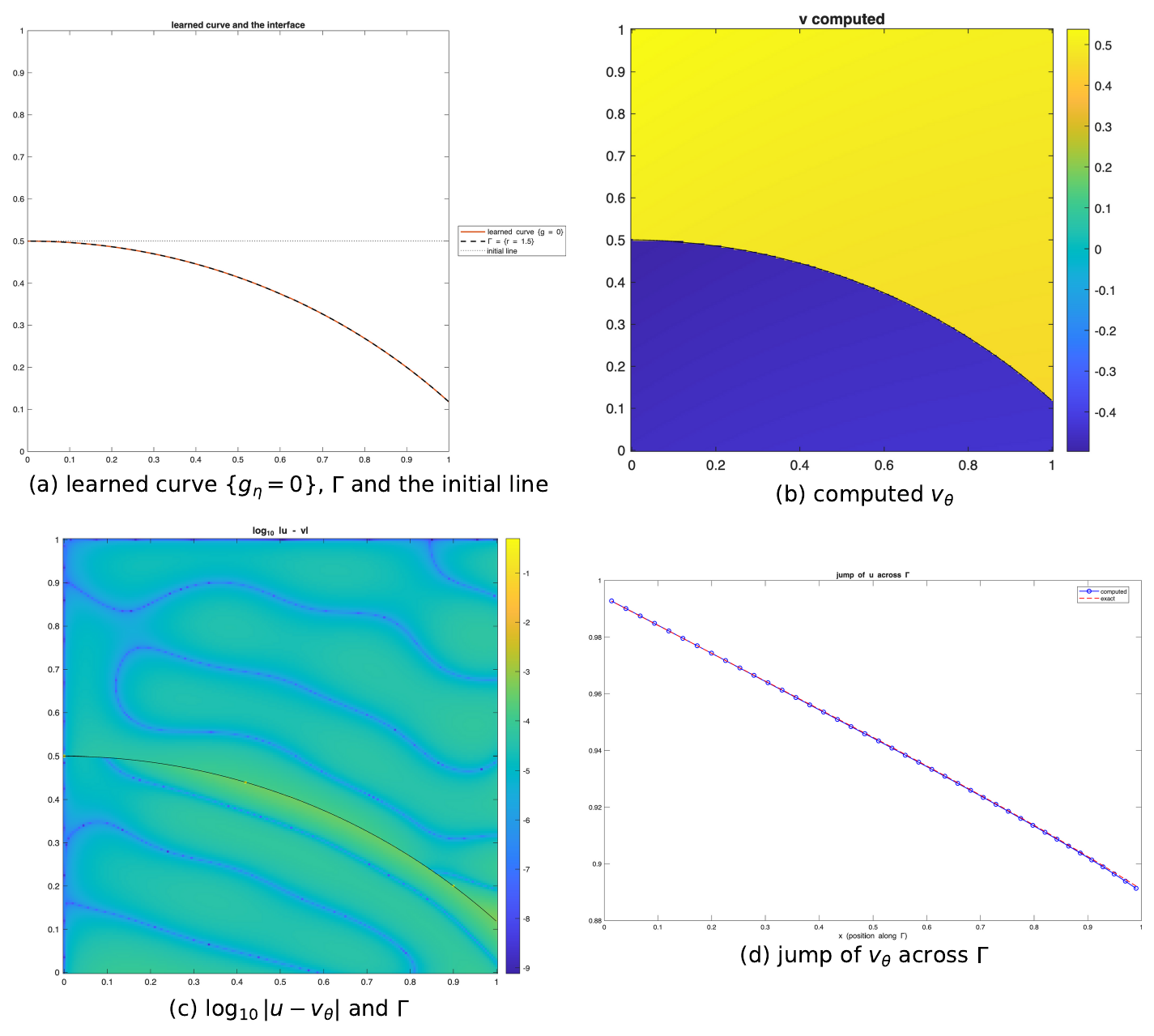}
\caption{Linear transport with a learned interface (Level~II), first class, after the main run: (a) the learned curve $\{g_\eta=0\}$ (solid), $\Gamma$ (dashed) and the
initial line $x_2=0.5$ (dotted); (b) $v_\theta$; (c) $\log_{10}|u-v_\theta|$,
with $\Gamma$; (d) the jump of $v_\theta$ across
$\Gamma$ (circles) and that of $u$ (dashed), against $x_1$.}\label{fig:transport_learned}
\end{figure}

\emph{Generic level.} Without an interface variable, the class is one network with two $\tanh$ layers $h$ of width $30$ and a last layer $a_k=W_3h+b_{3}$ of $20$ units, and
\[
 \psi_\theta=\mathbf c_\psi\cdot h+\sum_kc_k\operatorname{ReLU}(a_k),\qquad v_\theta=\mathbf c_v\cdot(h,1)+\sum_kb_k\operatorname{ReLU}(a_k)+\sum_k(d_k+\mathbf e_k\cdot x)\,H(a_k),
\]
$H$ the unit step. $v$ is only required in $L^2$, so its class may jump, and its jump sets are the zero sets of the same units that carry the kinks of $\psi_\theta$: across $\{a_k=0\}$, $\curl\psi_\theta$ jumps by $c_k\curl a_k$ and $v_\theta$ by $d_k+\mathbf e_k\cdot x$, which may vary along the zero set. When the zero set of a unit meets the inflow edge within $0.02$ of the jump point $x_d=(0,0.5)$ of the inflow data, the unit is tied to it, $b_{3,k}=-W_{3,k}h(x_d)$, so that $a_k(x_d)=0$
for all parameters; $x_d$ is a point of the data, and the tie removes the corner of $\cJ$ in the position of a step at the jump of the data. With the first class the divergence equation is built in at the second level of Remark~\ref{rem:div_levels}, $q=f-\gamma v$ and $\bsigma_\theta=\curl\psi_\theta+\Rop_2(f-\gamma v_\theta)$, so that $\Rop_2$ carries the divergence $-\gamma v_\theta$, which jumps with $v_\theta$. The rule is that of the learned level, cut along the zero sets of all units with the motion of every cut in the derivative, and $\Rop_2v_\theta$ is evaluated on rays split at
the zeros of all units. In the second class the divergence equation cannot be built in, so the first level is used: $\bsigma_\theta=\curl\psi_\theta+\bzz_\theta$ with $\bzz_\theta=\mathbf C_{\bzz}^T(h,1)+\sum_k\mathbf c_{\bzz,k}\operatorname{ReLU}(a_k)$ on the same last layer, so that $\bzz_\theta$ kinks where $v_\theta$ steps (Lemma~\ref{lem:div_kink}), and the residual $\divv\bzz_\theta+\gamma v_\theta-f$; no rays are needed. Two starts of $8$ minutes with $N_q=128$ reached $\cJ_{\rm check}=8.4\cdot10^{-5}$ and $3.7\cdot10^{-4}$ with the first class and $1.6\cdot10^{-3}$ and $6.7\cdot10^{-6}$ with the second; the better was continued for at most $60$ minutes with $N_q=256$:
\begin{center}
\footnotesize\setlength{\tabcolsep}{3pt}
\begin{tabular}{lccccccc}
\toprule
class & $\cJ_{\rm check}$ & $\frac{\|u-v_\theta\|_0}{\|u\|_0}$ & $\frac{\|u-v_\theta\|_{L^1}}{\|u\|_{L^1}}$ & $\frac{\|\bsigma-\bsigma_\theta\|_0}{\|\bsigma\|_0}$ & jump & \multicolumn{2}{c}{distance to $\Gamma$} \\
 & & & & & & mean & max \\
\midrule
$\Rop_2(f-\gamma v)$ & $3.4\cdot10^{-6}$ & $3.4\%$ & $0.51\%$ & $3.4\%$ & $1.5\%$ & $4.1\cdot10^{-4}$ & $8.4\cdot10^{-4}$ \\
$\bzz$ & $4.4\cdot10^{-6}$ & $4.9\%$ & $0.57\%$ & $4.9\%$ & $2.4\%$ & $7.0\cdot10^{-4}$ & $1.2\cdot10^{-3}$ \\
\bottomrule
\end{tabular}
\end{center}
The jump column is the relative $\ell^2$ error of the jump of $v_\theta$ across $\Gamma$ (traces at distance $0.005$), the distance is that of the zero set of the tied unit. Its zero set moved from
a line onto the characteristic, and it is the only unit with a step of $v_\theta$ along its zero set (Figure~\ref{fig:transport_generic}(a)); its step, measured along its own zero set, follows the exact jump along $\Gamma$, which decreases from $0.99$ to
$0.89$, with a relative error of $1.6\%$. The computed $v_\theta$ jumps, without a transition layer (Figure~\ref{fig:transport_generic}(b)--(d)). The $L^2$ error of $v_\theta$ is reported separately; Lemma~\ref{lem:tr_norm} does not give an $L^2$ estimate
for $u-v_\theta$. With the second class the zero set of the tied unit lies on the characteristic (Figure~\ref{fig:transport_generic_reg}(a)), at a mean distance of $7.0\cdot10^{-4}$ from $\Gamma$, its step has a relative error of $2.6\%$, and $\|\divv(\bsigma-\bsigma_\theta)\|_0/\|\bsigma\|_0=0.52\%$
(here $\|\divv\bsigma\|_0=\gamma\|\bsigma\|_0$). Among the four starts the smallest $\cJ_{\rm check}$, $6.7\cdot10^{-6}$, was obtained with the second class; the main run from it stopped after $25$ iterations, when L-BFGS found no descent step. Figures~\ref{fig:transport_generic} and~\ref{fig:transport_generic_reg} show the first and the second class.
\begin{figure}[t]
\centering
\includegraphics[width=0.92\textwidth]{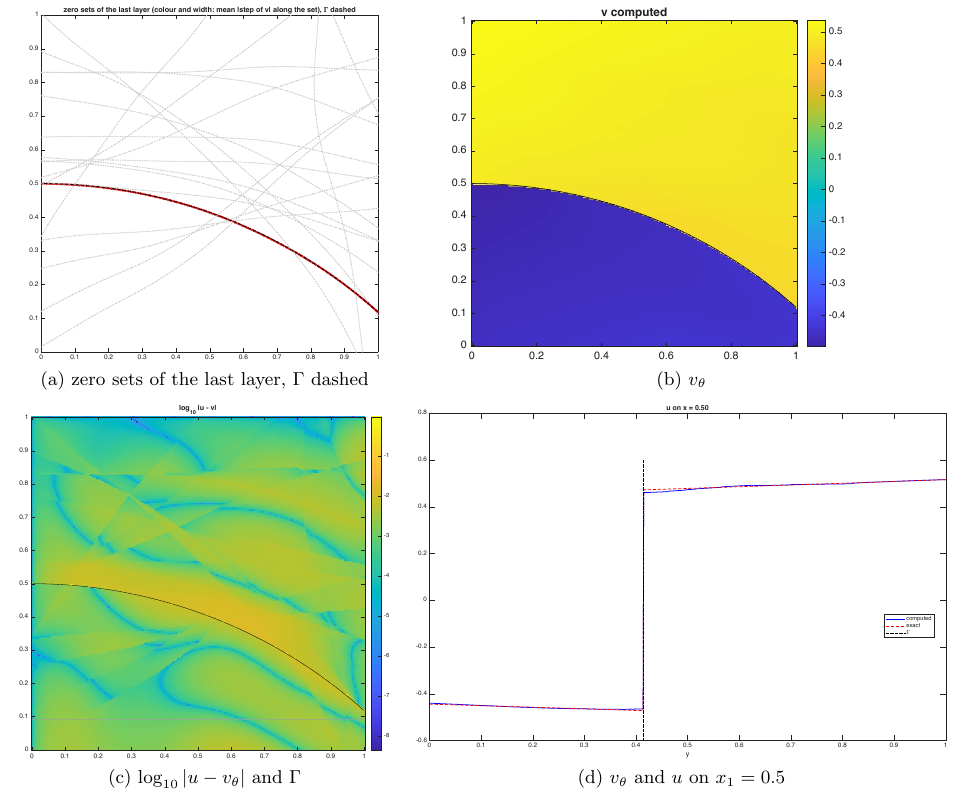}
\caption{Linear transport, generic level, first class: (a) the zero sets of the last layer, colored by the mean step of $v_\theta$ along them, $\Gamma$
dashed; (b) $v_\theta$; (c) $\log_{10}|u-v_\theta|$, with $\Gamma$; (d) $v_\theta$ and $u$ on $x_1=0.5$.}\label{fig:transport_generic}
\end{figure}
\begin{figure}[t]
\centering
\includegraphics[width=0.92\textwidth]{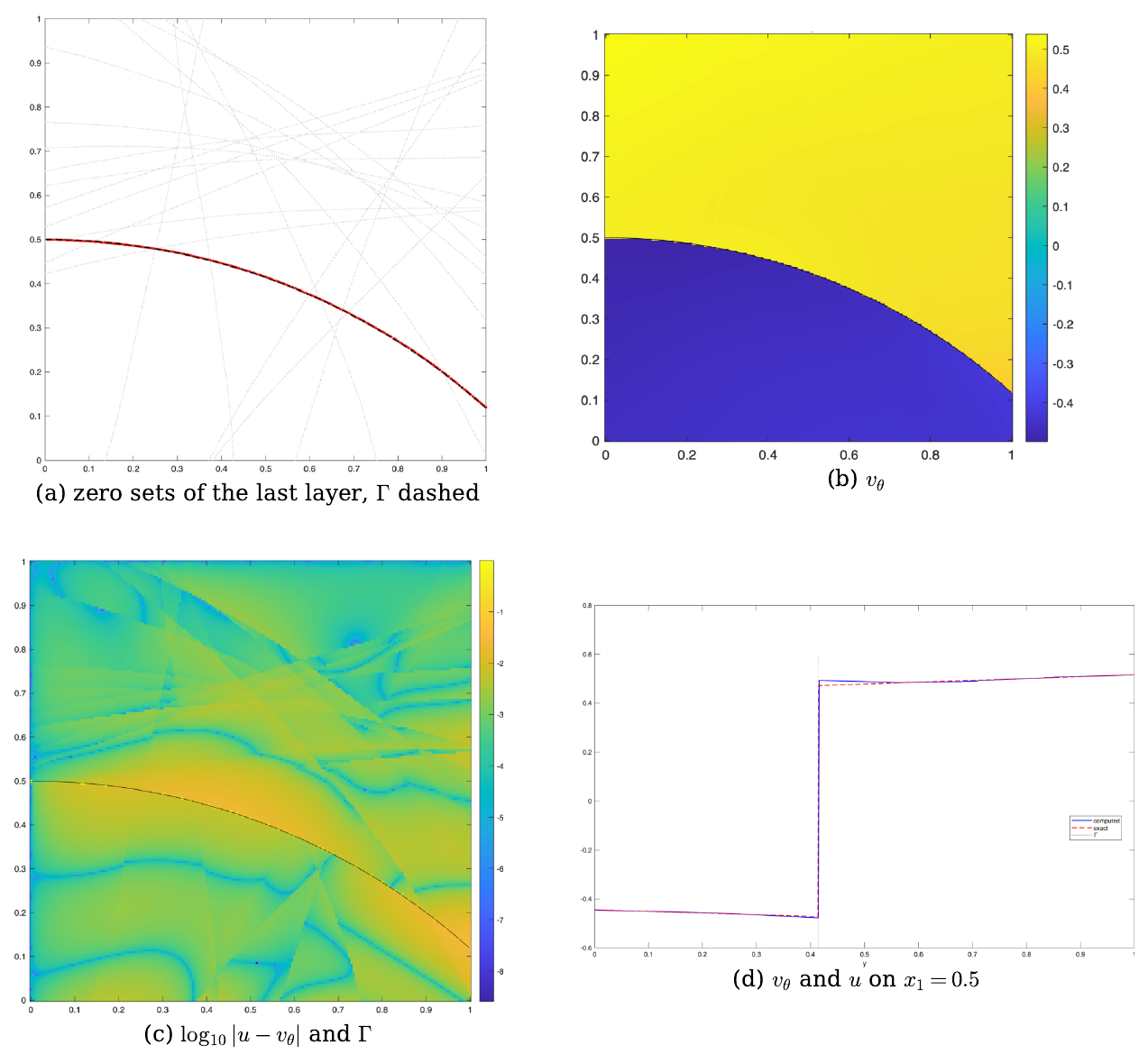}
\caption{Linear transport, generic level, second class ($\bsigma_\theta=\curl\psi_\theta+\bzz_\theta$): (a) the zero sets of the last layer, colored by the mean
step of $v_\theta$ along them, $\Gamma$ dashed; (b) $v_\theta$; (c) $\log_{10}|u-v_\theta|$, with $\Gamma$; (d) $v_\theta$ and $u$ on $x_1=0.5$.}\label{fig:transport_generic_reg}
\end{figure}

\section{Scalar conservation laws: an unknown moving interface}\label{sec:conslaw}

In Sections~\ref{sec:method} and~\ref{sec:curlcurl} the interfaces are those of the coefficients,
whose location is part of the data even
where the class does not use it (Levels~II
and~III), and in Section~\ref{sec:transport} they are characteristics determined by the
data. For a scalar conservation law the discontinuities are
shocks whose location is part of the solution. We use
the flux least-squares formulation of \cite{DMMO:05} as it stands; the only change
is that the $H(\divvr)$-conforming finite element flux is replaced by a
native neural class. No shock location or interface module is
supplied to the class; the generic deRhaNN flux class is used.

\noindent{\bf Formulation.}  Let $\Omega\subset\R^n$ be a bounded Lipschitz domain, $Q:=(0,T)\times\Omega$, and
\beq\label{conslaw}
 \partial_tu+\divv F(u)=0\ \ \mbox{in }Q,\qquad u(0)=u_0,
\eeq
with inflow data on the part of $(0,T)\times\partial\Omega$ where $F'(u)\cdot\bn<0$. With the space-time flux $\bff(u):=(u,F(u))\in\R^{n+1}$ the equation
reads $\divv_{t,x}\bff(u)=0$, and for a weak solution the space-time flux lies in $H(\divvr_{t,x};Q)$ with zero
divergence. Across a shock surface $\Gamma\subset Q$ with space-time normal $\bn=(n_t,\bn_x)$, continuity of the normal component,
\[
 [\bff(u)\cdot\bn]_\Gamma=[u]_\Gamma\,n_t+[F(u)]_\Gamma\cdot\bn_x=0,
\]
is the Rankine--Hugoniot condition, while the tangential components of $\bff(u)$ jump. A shock is therefore an
interface of exactly the kind that $H(\divvr)$ fields carry, as the interfaces of Section~\ref{sec:method},
except that its location is unknown. The flux least-squares functional of \cite{DMMO:05} takes the flux $\btau\approx\bff(u)$ as an independent variable,
\beq\label{LS_cl}
 \cJ_{\rm cl}(\btau,v):=\|\divv_{t,x}\btau\|_{0,Q}^2+\|\btau-\bff(v)\|_{0,Q}^2
 +\|\btau\cdot\bn-g\|_{0,\Gamma_{\rm in}}^2+\|v-u_{\rm in}\|_{0,\Gamma_{\rm in}}^2,
\eeq
where $\Gamma_{\rm in}$ is the inflow part of $\partial Q$, $u_{\rm in}$ the inflow data ($u_0$ on $\{t=0\}$) and $g=\bff(u_{\rm in})\cdot\bn$ (so $g=-u_0$ on $\{t=0\}$, where $\bn=(-1,\mathbf 0)$). We consider
$v$ with $\bff(v)\in L^2(Q)$, as is the case for the bounded neural trial functions used below. The inflow
condition can be imposed on $\btau\cdot\bn$, on $v$, or on both, since both hold for the solution; we use both. The functional is not quadratic,
and Proposition~\ref{prop:best_approx} does not apply. The LSNN method for conservation laws \cite{CCL:22nonlinear,CCL:23_divergence} uses an equivalent least-squares formulation on
a solution set that admits discontinuous solutions, with ReLU networks and a physics-preserving numerical differentiation, the discrete
divergence operator \cite{CCL:23_divergence}, and without artificial viscosity or entropy penalization \cite{LC:26lsnnchapter}; here the flux is an independent variable in a native $H(\divvr)$ class.

\noindent{\bf Native space-time flux classes.}  The flux is a field in $H(\divvr)$ on the $(n+1)$-dimensional
domain $Q$, and Theorem~\ref{thm:hdivd} in dimension $n+1$ gives the generic native class
\[
 \btau_\Theta=\operatorname{Div}_{t,x}A_\theta+\Rop_{n+1}q_\eta+\bh,\qquad A_\theta^T=-A_\theta,
\]
with a skew-symmetric matrix of scalar networks $A_\theta$, a scalar network $q_\eta$,
and $\bh=0$ when $\Omega$ has no cavities, since then $\R^{n+1}\setminus\overline Q$
has no bounded component. Then $\divv_{t,x}\btau_\Theta=q_\eta$, and the first term
of \eqref{LS_cl} is $\|q_\eta\|_{0,Q}^2$. Since the conservation law is homogeneous, it
can also be built in: with $q_\eta=0$, $\btau_\Theta=\operatorname{Div}_{t,x}A_\theta$ is divergence free for every parameter value,
the class is native to the divergence-free subspace of $H(\divvr_{t,x};Q)$ (Theorem~\ref{thm:hdivd}
with $q=0$), and the divergence residual vanishes, so that \eqref{LS_cl} reduces
to its last three terms. Space-time divergence-free networks of this
form are the construction of \cite{RichterPowellLipmanChen2022}, where they satisfy the continuity equation
exactly. The unknown $v$ is a separate scalar network, as in \cite{DMMO:05}.

For $n=1$ the skew matrix has one entry, and
the built-in class is the space-time curl of
one scalar network, $\btau_\Theta=(\partial_x\psi_\theta,-\partial_t\psi_\theta)$. Then $\psi_\theta$ plays the
role of the cumulative variable $\int^xu\,dx'$: $\partial_x\psi=u$ and
$\partial_t\psi=-F(u)$, and a shock of $u$ is a kink of
$\psi$. A ReLU network $\psi_\theta$ has kinks on hyperplanes whose positions are
trained parameters, so a shock can be represented by
a trained kink location. The represented flux has continuous
normal component across its kinks by construction; where
$\btau_\Theta=\bff(v_\theta)$ holds, this is the Rankine--Hugoniot relation. With a
smooth activation the shock is a layer, which has to
be resolved by the quadrature, whose refinement cannot be
prescribed in advance when the shock location is unknown.

\begin{prop}[consistency]\label{prop:cl_consistency}
Let $F$ be Lipschitz, or, for a locally Lipschitz flux such as Burgers' flux $F(u)=u^2/2$, let $\sup_m\|v_m\|_{L^\infty(Q)}+\|v\|_{L^\infty(Q)}<\infty$, and let $(\btau_m,v_m)$ with
$\btau_m\in H(\divvr_{t,x};Q)$, $v_m\in L^2(Q)$ satisfy $\cJ_{\rm cl}(\btau_m,v_m)\to0$, $\btau_m\to\btau$ and $v_m\to v$ in $L^2(Q)$. Then $\btau=\bff(v)$ and $\divv_{t,x}\btau=0$ in $Q$, and $v$ is a weak solution
of \eqref{conslaw} with the inflow data: $(\bff(v),\grad_{t,x}\phi)_Q=\la g,\phi\ra_{\Gamma_{\rm in}}$ for every smooth $\phi$ with $\phi=0$ near $\partial Q\setminus\Gamma_{\rm in}$.
\end{prop}

\begin{proof}
$\|\btau_m-\bff(v_m)\|_{0,Q}\to0$ and $\|\bff(v_m)-\bff(v)\|_{0,Q}\le(1+\mathrm{Lip}(F))\|v_m-v\|_{0,Q}\to0$, with $\mathrm{Lip}(F)$ the Lipschitz constant of $F$ on an interval containing the ranges of the $v_m$
and of $v$, give $\btau=\bff(v)$. For $\phi$ as stated, the Green formula in $H(\divvr_{t,x};Q)$
gives $(\btau_m,\grad_{t,x}\phi)_Q=-(\divv_{t,x}\btau_m,\phi)_Q+\la\btau_m\cdot\bn,\phi\ra_{\Gamma_{\rm in}}$; the first term on the right tends to $0$, and the second to $\la g,\phi\ra_{\Gamma_{\rm in}}$ because $\|\btau_m\cdot\bn-g\|_{0,\Gamma_{\rm in}}\to0$.
\end{proof}

This is the neural counterpart of the conservation property proved
for $H(\divvr)$-conforming least-squares finite elements in \cite{DMMO:05}, and its proof
uses only the conformity $\btau_m\in H(\divvr_{t,x};Q)$ for every parameter value.

\subsection{Numerical example}\label{subsec:cl_numerics}
Burgers' equation, $F(u)=u^2/2$ and $n=1$, on $Q=(0,1)\times(-1,1)$ with
the variables $(t,x)$, with $u_0=1$ for $x<0$, $u_0=0$ for
$x>0$ and $u=1$ on $\{x=-1\}$; on $\{x=1\}$, where $u=0$,
the boundary is characteristic. The entropy solution is the
shock $x=t/2$ between the states $1$ and $0$. Across
it $\bff(u)=(u,u^2/2)$ has the purely tangential jump $(1,\frac12)$, of
length $1.118$, and its position is not given
to the class. The class is the generic
first class, $\btau_\Theta=\curl\psi_\theta+\Rop_2q_\theta$ with $\curl\psi=(\p_x\psi,-\p_t\psi)$ and $x_0=(0.5,0)$; $\psi_\theta$ and
$q_\theta$ are outputs of a network with two $\tanh$
layers of width $30$ followed by three ReLU layers
of width $20$, the first fitted to lines in
$Q$ at the start and the deeper ones initialized as
the identity, and $v_\theta$ is the output of a
separate network with two $\tanh$ layers of width
$20$ and three ReLU layers of width $20$.
The inflow terms of \eqref{LS_cl}, on $\Gamma_{\rm in}=\{t=0\}\cup\{x=-1\}$, have the
weights $10$ on $\btau\cdot\bn-g$ and $100$ on $v-u_{\rm in}$. Since $\bff(v)$
is nonlinear in $v$, the coefficients of $v_\theta$,
its output layer included, are trained by Levenberg--Marquardt together with the
hidden parameters of the flux network; the output
coefficients of $\psi_\theta$ and $q_\theta$, which enter linearly
for fixed $v_\theta$, are the least-squares solution. The
derivative residual $\curl\psi_\theta$ enters by the cellwise moment
treatment of Section~\ref{subsec:computation}. The training cells are squares of side
$1/32$. Two starts of six minutes from different
initial networks are run, and the better one
is continued for at most $48$ minutes in phases of
ten minutes, with refinement and up to four
new units after each phase. The check $\cJ_{\rm check}$
uses the training squares halved, cut along all kink
sets of both networks. We report $\cJ_{\rm check}$, the
$L^1$ error of $v_\theta$ relative to $\|u\|_{L^1(Q)}=1.25$, the
relative $L^2$ error of $\btau_\Theta$, $\|q_\theta\|_0$ (the divergence of
$\btau_\Theta$, zero for the solution), and the largest error of
the shock position at $t=0.2,0.4,0.6,0.9,1$, located as the
level $\frac12$ of $v_\theta$ and of the first component $(\btau_\Theta)_t$.
\begin{center}
\small\setlength{\tabcolsep}{4pt}
\begin{tabular}{cccccccccc}
\toprule
width & cells & steps & time & $\cJ_h$ & $\cJ_{\rm check}$ & $\|u-v_\theta\|_{L^1}$ & $\|\bff(u)-\btau_\Theta\|_0$ & $\|q_\theta\|_0$ & shock ($v_\theta$, $\tau_t$) \\
\midrule
35 & 2606 & 156 & 33 min & $4.05\cdot10^{-4}$ & $1.22\cdot10^{-3}$ & $1.97\%$ & $7.65\%$ & $1.07\cdot10^{-3}$ & $0.014$, $0.013$ \\
\bottomrule
\end{tabular}
\end{center}
$\cJ_{\rm check}$ decreased in every phase, from $5.7\cdot10^{-2}$
to $1.2\cdot10^{-3}$, and the $L^1$ error with it, from
$4.2\%$ to $2.0\%$; the run stopped when Levenberg--Marquardt made
no further progress after the last phase. The shock
is resolved within one to two training cells
(Figure~\ref{fig:burgers_riemann}), without oscillations. It lies ahead of $x=t/2$ by
at most $0.014$ for $t\le0.6$ and by $0.003$ at
$t=1$; at $t=1$, $v_\theta$ exceeds $1$ by $4\%$ and $(\btau_\Theta)_t$
by $8\%$ just behind the shock. The active
kink sets of $\psi_\theta$ cluster along the shock;
Figure~\ref{fig:burgers_riemann}(b) shows those across which $\btau_\Theta$ has a
tangential jump of at least $2\%$ of the exact
one. The normal continuity $[\btau_\Theta\cdot\bn]=0$ holds by construction.
\begin{figure}[t]
\centering
\includegraphics[width=\textwidth]{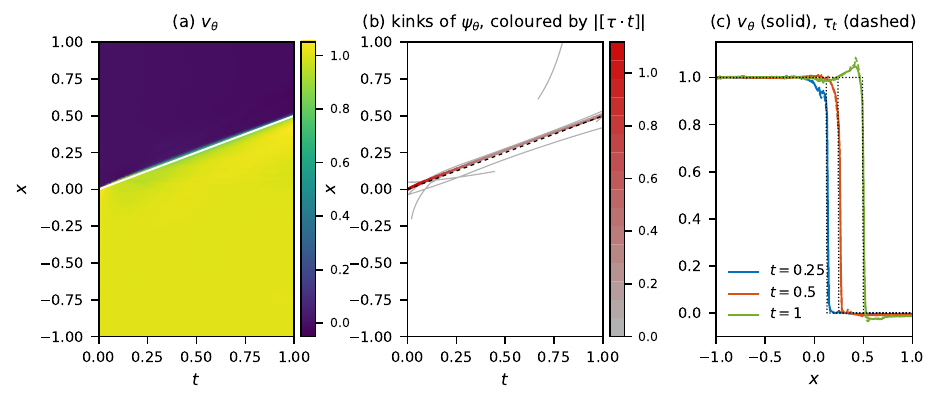}
\caption{Burgers' equation, Riemann data, generic level.  (a) $v_\theta$ on $Q$, the exact shock $x=t/2$ white.  (b) The kinks of $\psi_\theta$
across which $\btau_\Theta$ has a tangential jump of at least $2\%$ of the exact one,
colored by the jump; the exact shock dashed. (c) $v_\theta$ (solid) and $(\btau_\Theta)_t$ (dashed) at $t=0.25$, $0.5$, $1$; the exact solution dotted.}
\label{fig:burgers_riemann}
\end{figure}

\section{Concluding remarks}\label{sec:conclusion}

We constructed two deRhaNN classes for $H(\divvr)$ and native deRhaNN classes for $H(\curll)$. In each case the graph-space field
is assembled from ordinary neural representing functions by fixed operators of the de Rham complex. The classes are conforming for every
parameter value, and the admissible interface jumps are carried at finite width by kinks of the potential.

Combined with first-order system least squares, these classes give the FOSLS-deRhaNN method. The elliptic and curl--curl examples retain the natural graph-space
stability, while the transport and conservation-law examples use the same native flux constructions for discontinuous solutions and moving interfaces.

\end{document}